\documentclass[12pt]{amsart}
\usepackage{a4wide}
\usepackage{amssymb}
\usepackage{amsthm} 
\usepackage{amsmath} 
\usepackage{amscd}
\usepackage{verbatim}
\usepackage{mathrsfs}
\usepackage{dsfont,bbm} 
\usepackage{youngtab}
\usepackage[usenames]{color}
\usepackage{tikz,graphicx}
\usepackage{hyperref}
\usepackage{multicol}

\numberwithin{equation}{section}
\theoremstyle{plain}
\newtheorem{theorem}{Theorem}[section]
\newtheorem{corollary}[theorem]{Corollary}
\newtheorem{lemma}[theorem]{Lemma}
\newtheorem{proposition}[theorem]{Proposition}
\newtheorem{question}[theorem]{Question}
\theoremstyle{definition}

\newtheorem*{example}{Example}
\newtheorem{conjecture}[theorem]{Conjecture}
\newtheorem{problem}[theorem]{Problem}
\theoremstyle{remark}
\newtheorem*{remark}{Remark}

\DeclareMathOperator{\SL}{SL}
\DeclareMathOperator{\GL}{GL}

\DeclareMathOperator{\ch}{ch}

\renewcommand{\H}{\mathbb{H}}

\newcommand{\Real}{\mathbb R}
\newcommand{\N}{\mathbb N}

\newcommand{\Z}{\mathbb Z}

\newcommand{\C}{\mathrm C}
\newcommand{\D}{\mathrm D}

\newcommand{\Vnat}{V^{\natural}}
\newcommand{\m}{\mathbf{m}}
\newcommand{\n}{\mathbf{n}}

\newcommand{\Lee}{\Lambda}
\newcommand{\MM}{\mathbb{M}}
\newcommand{\Co}{\textsl{Co}}
\newcommand{\CC}{\mathbb C}
\newcommand{\RR}{\mathbb R}
\newcommand{\HH}{\mathbb H}
\renewcommand{\t}{\tau}
\usepackage{ctable}
\newcommand{\tr}{{\rm tr}}
\newcommand{\dimq}{\dim_*}
\newcommand{\qdim}{\operatorname{qdim}}
\newcommand{\gt}[1]{\mathfrak{#1}}
\newcommand{\Vir}{\mathcal{V}}

\newcommand{\Id}{{\rm Id}}

\newcommand{\Vsnat}{V^{s\natural}}
\DeclareMathOperator{\td}{td}
\newcommand{\QQ}{\mathbb{Q}}
\DeclareMathOperator{\Aut}{Aut}
\DeclareMathOperator{\Out}{Out}

\newcommand{\md}{M}
\newcommand{\tw}{{\rm tw}}

\begin{document}

\title[Moonshine]
{Moonshine}

\author{John F. R. Duncan, Michael J. Griffin and Ken Ono}

\address{Department of Mathematics, Applied Mathematics and Statistics,
Case Western Reserve University, Cleveland, Ohio 44106}
\email{john.duncan@case.edu}

\address{Department of Mathematics and Computer Science,
Emory University, Atlanta, Georgia 30322}
\email{mjgrif3@emory.edu}

\address{Department of Mathematics and Computer Science,
Emory University, Atlanta, Georgia 30322}
\email{ono@mathcs.emory.edu}

\thanks{The authors are supported by the NSF. The first author also thanks the Simons Foundation (\#316779), and the third author also thanks the support of the Asa Griggs Candler Fund.
The authors thank John McKay and and John Thompson for answering questions about the history related to the study of
distributions within the moonshine module. 
Thanks are due also to Miranda Cheng, Igor Frenkel, Bob Griess, Daniel Grumiller, Jeff Harvey, Shamit Kachru, Barry Mazur, John McKay, Jean-Pierre Serre and the anonymous referee, for comments, corrections and excellent suggestions.}


\begin{abstract} 
{\it Monstrous moonshine} relates distinguished modular functions to the representation theory of the Monster $\MM$. The celebrated observations that 
\begin{equation}\label{star}
1=1,\ \ \ 196884=1+196883,\ \ \ 21493760=1+196883+21296876,\ \ \ \dots\dots  \tag{*}
\end{equation}
illustrate the case of $J(\tau)=j(\tau)-744$, whose coefficients turn out to be sums of the dimensions of the 194 irreducible representations of $\MM$. Such formulas are dictated by the structure of the graded monstrous moonshine modules. Recent works in moonshine suggest deep relations between number theory and physics. Number theoretic Kloosterman sums have reappeared in quantum gravity, and mock modular forms have emerged as candidates for the computation of black hole degeneracies. This paper is a survey of past and present research on moonshine. We also compute the quantum dimensions of the monster orbifold, and obtain exact formulas for the multiplicities of the irreducible components of the moonshine modules. 
These formulas imply that such multiplicities are asymptotically proportional to dimensions. For example, the proportion of 1's in (*) tends to
$$
\frac{\dim(\chi_1)}{\sum_{i=1}^{194}\dim(\chi_i)}=\frac{1}{5844076785304502808013602136}=1.711\ldots \times 10^{-28}.
$$
\end{abstract}
\maketitle

\section{Introduction}
This story begins innocently with peculiar numerics, and in its present form exhibits connections to conformal field theory, string theory, quantum gravity, and the arithmetic of mock modular forms. This paper is an introduction to the many facets of this beautiful theory. 

We begin with a review of the principal results in the development of monstrous moonshine in \S\S\ref{sec:edays}-\ref{sec:intro:mod}. We refer to the introduction of \cite{FLM}, and the more recent survey \cite{MR2201600}, for more on these topics. 
After describing these classic works, we discuss the interplay between moonshine and Rademacher sums in \S\ref{sec:intro:radsums}, and related observations which suggest a connection between monstrous moonshine and 
three-dimensional quantum gravity in \S\ref{sec:intro:qgrav}. 
The remainder of the paper is devoted to more recent works. 
We describe a generalization of moonshine, the 
{\it moonshine tower}, in \S\ref{sec:tower}, and 
in \S\ref{sec:mmdist} we compute the quantum dimensions of the monster orbifold and discuss the distributions of irreducible representations of the monster arising in monstrous moonshine. 
We present a survey of the recently discovered {\it umbral moonshine} phenomenon in \S\ref{sec:um}, and conclude, in \S\ref{sec:open}, with problems for future work.
 
\section{Early Days}\label{sec:edays}

The classification of finite simple groups \cite{MR2072045} distinguishes twenty-six examples above the others; namely, the {\em sporadic} simple groups, which are those that belong to none of the naturally occurring infinite families: cyclic groups, alternating groups, or finite groups of Lie type. Distinguished amongst the sporadic simple groups is the {\em Fischer--Griess monster} $\MM$, on account of its size, which is
\begin{gather}\label{eqn:intro-morder}
	|\MM|
	=
	2^{46}\cdot 3^{20}\cdot 5^9\cdot 7^6\cdot 11^2\cdot 13^3\cdot 17\cdot 19\cdot 23\cdot 29\cdot 31\cdot 41\cdot 47\cdot 59\cdot 71
\end{gather}
(cf. \cite{MR0399248}). Note that the margin is not small, for the order of the monster is 
\begin{gather}
2^5\cdot 3^7\cdot 5^3\cdot 7^4\cdot 11\cdot 13^2\cdot 29\cdot 41\cdot 59\cdot 71
\end{gather}
times that of the next largest sporadic simple group, the baby monster (cf. \cite{MR0444765}).

Fischer and Griess independently produced evidence for the monster group in 1973 (cf. \cite{MR0399248}). Well before it was proven to exist, Tits gave a lecture on its conjectural properties at the Coll\`ege de France in 
1975. In particular, he described its order (\ref{eqn:intro-morder}). Around this time, Ogg had been considering the automorphism groups of certain algebraic curves, and had arrived at the set of primes 
\begin{gather}
\left\{2,3,5,7,11,13,17,19,23,29,31,41,47,59,71\right\}
\end{gather} 
in a purely geometric way (cf. the Corollaire of \cite{Ogg_AutCrbMdl}). Making what may now be identified as the first observation of monstrous moonshine, Ogg offered a bottle of Jack Daniels\footnote{We refer the reader to \cite{2014arXiv1411.5354D} for a recent analysis of the Jack Daniels problem.} for an explanation of this coincidence (cf. Remarque 1 of \cite{Ogg_AutCrbMdl}).

Ogg's observation would ultimately be recognized as reflecting another respect in which the monster is distinguished amongst finite simple groups: as demonstrated by the pioneering construction of Frenkel--Lepowsky--Meurman \cite{FLMPNAS,FLMBerk,FLM}, following the astonishing work of Griess \cite{MR605419,MR671653}, the ``most natural'' representation of the monster, is infinite-dimensional. 

The explanation of this statement takes us back to McKay's famous observation, that
\begin{gather}\label{eqn:intro-mckobs}
	196884=1+196883
\end{gather}
(cf. \cite{MR554399,Tho_NmrlgyMonsEllModFn}), and the generalizations of this observed by Thompson \cite{Tho_NmrlgyMonsEllModFn}, including\footnote{As was pointed out to us by J.-P. Serre, the 
decomposition for $20245856256$ that appears in \cite{Tho_NmrlgyMonsEllModFn} differs from that which develops from the monster's natural, infinite-dimensional representation. We reproduce Serre's decomposition in (\ref{eqn:intro-thomckobs}).}
\begin{gather}\label{eqn:intro-thomckobs}
	\begin{split}
	21493760&=1+196883+21296876,\\
	864299970&=2\times 1+2\times 196883+21296876+842609326,\\
	20245856256&=2\times 1+3\times 196883+2\times 21296876+842609326 + 19360062527.
	\end{split}
\end{gather}

Of course the left hand sides of (\ref{eqn:intro-mckobs}) and (\ref{eqn:intro-thomckobs}) are familiar to number theorists and algebraic geometers, as coefficients in the Fourier coefficients of the {\em normalized elliptic modular invariant}
\begin{equation}\label{eqn:intro-expJ}
\begin{split}
J(\t)
&:=\frac{1728g_2(\tau)^3}{g_2(\tau)^3-27g_3(\tau)^2}-744\\
&= q^{-1} + 196884 q + 21493760 q^2 + 864299970 q^3 + 20245856256 q^4+\ldots
\end{split}
\end{equation}
Here $q:=e^{2\pi i \t}$, and we set $g_{2}(\t):=60G_4(\t)$ and $g_3(\t):=140G_6(\t)$, where $G_{2k}(\t)$ denotes the Eisenstein series of weight $2k$, 
\begin{gather}
	G_{2k}(\tau):=\sum_{(m,n)\neq (0,0)}(m+n\tau)^{-2k},
\end{gather}
for $k\geq 2$. The functions $g_2$ and $g_3$ serve to translate between the two most common parameterizations of a complex elliptic curve: as a quotient $\CC/(\Z+\Z\tau)$ for $\tau$ in the upper-half plane, $\H:=\{\tau\in\CC\mid \Im(\tau)>0\}$, and as the locus of a Weierstrass equation, $y^2=4x^3-g_2x-g_3$. 

The fundamental property of $J(\tau)$, from both the number theoretic and algebro-geometric points of view, is
that it is a modular function for $\SL_2(\Z)$.
In fact, and importantly for the monster's natural infinite-dimensional representation, $J(\tau)$ is a generator for the field of $\SL_2(\Z)$-invariant holomorphic functions on $\H$ that have at most exponential growth as $\Im(\tau)\to \infty$.

The right hand sides of (\ref{eqn:intro-mckobs}) and (\ref{eqn:intro-thomckobs}) are familiar to finite group theorists, as simple sums of dimensions of irreducible representations of the monster $\MM$. In fact, the irreducible representations appearing in (\ref{eqn:intro-mckobs}) and (\ref{eqn:intro-thomckobs}) are just the first six, of a total of $194$, in the character table of $\MM$ (cf. \cite{ATLAS}), when ordered by size.  We have that
\begin{gather}\label{eqn:intro-dimchiprifac}
\begin{split}
\chi_1(e)&=1\\
\chi_2(e)&=196883\\
\chi_3(e)&=21296876\\
\chi_4(e)&=842609326\\
\chi_5(e)&=18538750076\\
\chi_6(e)&=19360062527\\
&\hspace{3em}\vdots\\
\chi_{194}(e)&=258823477531055064045234375.
\end{split}
\end{gather}
Here $e$ denotes the identity element of $\MM$, so $\chi_i(e)$ is just the dimension of the irreducible representation of $\MM$ with character $\chi_i$. 

\section{Classical Moonshine}\label{previouswork}

The coincidences (\ref{eqn:intro-mckobs}) and (\ref{eqn:intro-thomckobs}) led Thompson to make the following conjecture \cite{Tho_NmrlgyMonsEllModFn} 
which realizes the natural representation of the monster alluded to above.

\begin{conjecture}[Thompson]\label{conj:Thompson}
There is a naturally defined
graded infinite-dimensional monster module, denoted $\Vnat=\bigoplus_{n=-1}^{\infty}\Vnat_n$, which satisfies
\begin{gather}\label{eqn:intro-dimVnatcn}
	\dim(\Vnat_n)=c(n)
\end{gather}
for $n\geq -1$ (Cf. (\ref{eqn:intro-expJ})), such that the decompositions into irreducible representations of the monster satisfy (\ref{eqn:intro-mckobs}) and  (\ref{eqn:intro-thomckobs}) for $n=1, 2, 3$ and $4$ (and a similar condition for $n=5$).
\end{conjecture}

At the time that Thompson's conjecture was made, the monster had not yet been proven to exist, but Griess \cite{MR0399248}, and Conway--Norton \cite{MR554399}, had independently conjectured the existence of a faithful representation of dimension $196883$, and Fischer--Livingstone--Thorne had constructed the character table of $\MM$, by assuming the validity of this claim (cf. \cite{MR554399}) together with conjectural statements (cf. \cite{MR0399248}) about the structure of $\MM$.

Thompson also suggested \cite{Tho_FinGpsModFns} to investigate the properties of the graded-trace functions
\begin{gather}\label{eqn:intro-Tg}
	T_g(\tau):=\sum_{n=-1}^{\infty}\tr(g|\Vnat_n)q^n,
\end{gather}
for $g\in \MM$, now called the {\em monstrous McKay--Thompson series}, that would arise from the conjectural monster module $\Vnat$. 
Using the character table constructed by Fischer--Livingstone--Thorne, it was observed \cite{MR554399,Tho_FinGpsModFns} that the functions $T_g$ are in many cases directly similar to $J$ in the following respect: the first few coefficients of each $T_g$ coincide with those of a generator for the function field of a discrete group\footnote{The relevant groups $\Gamma_g$ shall be discussed in detail in Section~\ref{Gammag}.} $\Gamma_g<\SL_2(\RR)$, commensurable with $\SL_2(\Z)$, containing $-I$, and having {\em width one at infinity}, meaning that the subgroup of upper-triangular matrices in $\Gamma_g$ coincides with 
\begin{gather}\label{eqn:intro:class-Gammainfty}
	\Gamma_\infty:=\left\{ \pm\begin{pmatrix} 1&n\\0&1\end{pmatrix}\mid n\in\Z\right\},
\end{gather}
for all $g\in \MM$.

This observation was refined and developed by Conway--Norton \cite{MR554399}, leading to their famous {\em monstrous moonshine conjectures}:.

\begin{conjecture}[Monstrous Moonshine: Conway--Norton]\label{conj:MonstrousMoonshine}
For each $g\in\MM$ there is a specific
group $\Gamma_g<\SL_2(\RR)$ such that $T_g$ is the unique {\em normalized principal modulus}\footnote{A principal modulus is also referred to as a {\em Hauptmodul}.} for $\Gamma_g$.
\end{conjecture}

\noindent
This means that each $T_g$ is
 the unique $\Gamma_g$-invariant holomorphic function on $\HH$ which satisfies
\begin{gather}\label{eqn:intro-Tgatinfty}
T_g(\tau)=q^{-1}+O(q),
\end{gather}
as $\Im(\tau)\to \infty$, and remains bounded as $\tau$ approaches any non-infinite cusp of $\Gamma_g$. We refer to this feature of the $T_g$ as the {\em principal modulus property} of monstrous moonshine.

The hypothesis that $T_g$ is $\Gamma_g$-invariant, satisfying (\ref{eqn:intro-Tgatinfty}) near the infinite cusp of $\Gamma_g$ but having no other poles, implies that $T_g$ generates the field of $\Gamma_g$-invariant holomorphic functions on $\HH$ that have at most exponential growth at cusps, in direct analogy with $J$. In particular, the natural Riemann surface structure on $\Gamma_g\backslash\HH$ (cf. e.g. \cite{Shi_IntThyAutFns}) must be that of the Riemann sphere $\widehat{\CC}=\CC\cup\{\infty\}$ with finitely many points removed, and for this reason the groups $\Gamma_g$ are said to have {\em genus zero}, and the principal modulus property is often referred to as the {\em genus zero property} of monstrous moonshine.

The reader will note the astonishing predictive power that the principal modulus property of monstrous moonshine bestows: the fact that a normalized principal modulus for a genus zero group $\Gamma_g$ is unique, means that we can compute the trace of an element $g\in\MM$, on any homogeneous subspace of the monster's natural infinite-dimensional representation $\Vnat$, without any information about the monster, as soon as we can guess correctly the discrete group $\Gamma_g$. The analysis of Conway--Norton in \cite{MR554399} establishes very strong guidelines for the determination of $\Gamma_g$, and once $\Gamma_g$ has been chosen, the ``theory of replicability'' (cf. \cite{MR1200252, MR554399,MR760657}) allows for efficient computation of the coefficients of the normalized principal modulus $T_g$, given the knowledge of just a few of them (cf. \cite{MR1291027}, or (\ref{eqn:intro:class-recrsn})). 

It was verified by Atkin--Fong--Smith \cite{MR822245}, using results of Thompson \cite{Tho_FinGpsModFns} (cf. also \cite{MR1213794}), that a graded (possibly virtual) infinite-dimensional monster module $\Vnat$, such that the functions $T_g$ of (\ref{eqn:intro-Tg}) are exactly those predicted by Conway--Norton in \cite{MR554399}, exists. 

\begin{theorem}[Atkin--Fong--Smith]\label{firstthm}
There exists a (possibly virtual) graded $\MM$-module $\Vnat=\bigoplus_{n=-1}^{\infty}\Vnat_n$ such that if $T_g$ is defined by (\ref{eqn:intro-Tg}), then $T_g$ is the Fourier expansion of the unique $\Gamma_g$-invariant holomorphic function on $\HH$ that satisfies $T_g(\tau)=q^{-1}+O(q)$ as $\tau$ approaches the infinite cusp, and has no poles at any non-infinite cusps of $\Gamma_g$, where $\Gamma_g$ is the discrete subgroup of $\SL_2(\RR)$ specified by Conway--Norton in \cite{MR554399}. \end{theorem}
  
Thus Thompson's conjecture was confirmed, albeit indirectly. By this point in time, Griess, in an astonishing tour de force, had constructed the monster explicitly, by hand, by realizing it as the automorphism group of a commutative but non-associative algebra of dimension $196884$ \cite{MR605419,MR671653}. (See also \cite{Con_CnstM,MR768989}.) Inspired by Griess' construction, and by the representation theory of affine Lie algebras, which also involves graded infinite-dimensional vector spaces whose graded dimensions enjoy good modular properties (cf. e.g. \cite{MR0374210,MR513845,MR563927,MR750341}), Frenkel--Lepowsky--Meurman established Thompson's conjecture in a strong sense.

\begin{theorem}[Frenkel--Lepowsky--Meurman]
Thompson's Conjecture is true. In particular, the
 {\em moonshine module} $\Vnat$ is constructed in \cite{FLMPNAS,FLMBerk}.
 \end{theorem}

Frenkel--Lepowsky--Meurman generalized the {homogeneous} realization of the basic representation of an affine Lie algebra $\hat{\mathfrak{g}}$ due, independently, to Frenkel--Kac \cite{FreKac_AffLieDualRes} and Segal \cite{MR626704}, in such a way that {\em Leech's lattice} $\Lee$ \cite{Lee_SphPkgHgrSpc,Lee_SphPkgs}---the unique \cite{Con_ChrLeeLat} even self-dual positive-definite lattice of rank $24$ with no roots---could take on the role played by the root lattice of $\mathfrak{g}$ in the Lie algebra case. In particular, their construction came equipped with rich algebraic structure, furnished by vertex operators, which had appeared first in the physics literature in the late 1960's. 

We refer to \cite{FreKac_AffLieDualRes}, and also the introduction to \cite{FLM} for accounts of the role played by vertex operators in physics (up to $1988$) along with a detailed description of their application to the representation theory of affine Lie algebras. The first application of vertex operators to affine Lie algebra representations was obtained by Lepowsky--Wilson in \cite{MR0573075}.

Borcherds described a powerful axiomatic formalism for vertex operators in \cite{Bor_PNAS}. In particular, he introduced the notion of a {\em vertex algebra}, which can be regarded as similar to a commutative associative algebra, except that multiplications depend upon formal variables $z_i$, and can be singular, in a certain formal sense, along the canonical divisors $\{z_i=0\}$, $\{z_i=z_j\}$ (cf. \cite{MR1865087,MR2082709}). 

The appearance of affine Lie algebras above, as a conceptual ingredient for the Frenkel--Lepowsky--Meurman construction of $\Vnat$ hints at an analogy between complex Lie groups and the monster. Borcherds' vertex algebra theory makes this concrete, for Borcherds showed \cite{Bor_PNAS} that both in the case of the basic representation of an affine Lie algebra, and in the case of the moonshine module $\Vnat$, the vertex operators defined by Frenkel--Kac, Segal, and Frenkel--Lepowsky--Meurmann, extend naturally to vertex algebra structures. 

In all of these examples the {\em Virasoro algebra}, $\Vir=\bigoplus_{n}\CC L(n)\oplus \CC {\bf c}$, being the unique universal central extension of the Lie algebra $\CC[t,t^{-1}]\frac{{\rm d}}{{\rm d}t}$ of polynomial vector fields on the circle,
\begin{gather}\label{eqn:intro:class-vir}
	[L(m),L(n)]=(m-n)L(m+n)+\frac{m^3-m}{12}\delta_{m+n,0}{\bf c},\quad [L(m),{\bf c}]=0,
\end{gather}
acts naturally on the underlying vector space. (See \cite{MR1021978} for a detailed analysis of $\Vir$. The generator $L(m)$ lies above the vector field $-t^{m+1}\frac{{\rm d}}{{\rm d}t}$.)
This Virasoro structure, which has powerful applications, was axiomatized in \cite{FLM}, with the introduction of the notion of a {\em vertex operator algebra}. If $V$ is a vertex operator algebra and the central element ${\bf c}$ of the Virasoro algebra acts as $c$ times the identity on $V$, for some $c\in \CC$, then $V$ is said to have {\em central charge} $c$.
 
For the basic representation of an affine Lie algebra $\hat{\gt{g}}$, the group of vertex operator algebra automorphisms---i.e. those vertex algebra automorphisms that commute with the Virasoro action---is the adjoint complex Lie group associated to $\gt{g}$. For the moonshine module $\Vnat$, it was shown by Frenkel--Lepowsky--Meurman in \cite{FLM}, that the group of vertex operator algebra automorphisms is precisely the monster.

\begin{theorem}[Frenkel--Lepowsky--Meurman]
The moonshine module $\Vnat=\bigoplus_{n=-1}^{\infty}\Vnat_n$ is a vertex operator algebra of central charge $24$ whose graded dimension is given by $J(\tau)$, and whose automorphism group is $\MM$.
\end{theorem}
Vertex operator algebras are of relevance to physics, for we now recognize them as ``chiral halves'' of two-dimensional conformal field theories (cf. \cite{Gaberdiel:1999mc,Gaberdiel:2005sk}). From this point of view, the construction of $\Vnat$ by Frenkel--Lepowsky--Meurman constitutes one of the first examples of an orbifold conformal field theory (cf. \cite{MR968697, MR818423,MR851703}). In the case of $\Vnat$, the underlying geometric orbifold is the quotient 
\begin{gather}\label{eqn:class-Leechtorusorb}
	\left(\RR^{24}/\Lee\right)/(\Z/2\Z),
\end{gather}
of the $24$-dimensional torus $\Lee\otimes_{\Z}\RR/\Lee\simeq \RR^{24}/\Lee$ by the Kummer involution $x\mapsto -x$, where $\Lee$ denotes the Leech lattice. So in a certain sense, $\Vnat$ furnishes a ``24-dimensional'' construction of $\MM$. We refer to \cite{MR2082709,FLM,MR1651389,MR2023933} for excellent introductions to vertex algebra, and vertex operator algebra theory.

Affine Lie algebras are special cases of Kac--Moody algebras, first considered by Kac \cite{MR0259961} and Moody \cite{MR0207783,MR0229687}, independently. Roughly speaking, a Kac--Moody
algebra is ``built'' from copies of $\mathfrak{sl}_2$, in such a way that most examples are infinite-dimensional, but much of the finite-dimensional theory carries through (cf. \cite{MR1104219}). Borcherds generalized this further, allowing also copies of the three-dimensional Heisenberg Lie algebra to serve as building blocks, and thus arrived \cite{Bor_GKM} at the notion of {\em generalized Kac--Moody algebra}, or {\em Borcherds--Kac--Moody (BKM) algebra}, which has subsequently found many applications in mathematics and mathematical physics (cf. \cite{Hohenegger:2012zz, MR1803076}).

One of the most powerful such applications occurred in moonshine, when Borcherds introduced a particular example---the {\em monster Lie algebra} $\gt{m}$---and used it to prove \cite{borcherds_monstrous} the moonshine conjectures of Conway--Norton. His method entailed using monster-equivariant versions of the denominator identity for $\gt{m}$ to verify that the coefficients of the McKay--Thompson series $T_g$, defined by (\ref{eqn:intro-Tg}) according to the Frenkel--Lepowsky--Meurman construction of $\Vnat$, satisfy the replication formulas conjectured by Conway--Norton in \cite{MR554399}. This powerful result reduced the proof of the moonshine conjectures to a small, finite number of identities, that he could easily check by hand. 
\begin{theorem}[Borcherds]\label{thm:intro:class-borcherds}
Let $\Vnat$ be the moonshine module vertex operator algebra constructed by Frenkel--Lepowsky--Meurman, whose automorphism group is $\MM$. If $T_g$ is defined by (\ref{eqn:intro-Tg}) for $g\in \MM$, and if $\Gamma_g$ is the discrete subgroup of $\SL_2(\RR)$ specified by Conway--Norton in \cite{MR554399}, then $T_g$ is the unique normalized principal modulus for $\Gamma_g$.
\end{theorem}
Recall that an even self-dual lattice of signature $(m,n)$ exists if and only if $m-n=0\pmod{8}$ (cf. e.g. \cite{MR1662447}). Such a lattice is unique up to isomorphism if $mn>0$, and is typically denoted $I\!I_{m,n}$. In the case that $m=n=1$ we may take 
\begin{gather}\label{eqn:class-II11}
	I\!I_{1,1}:=\Z e +\Z f,
\end{gather} 
where $e$ and $f$ are isotropic, $\langle e,e\rangle=\langle f,f\rangle=0$, and $\langle e,f\rangle =1$. Then $me+nf\in I\!I_{1,1}$ has square-length $2mn$. Note that $I\!I_{25,1}$ and $\Lee\oplus I\!I_{1,1}$ are isomorphic, for $\Lee$ the Leech lattice, since both lattices are even and self-dual, with signature $(25,1)$. 

In physical terms the monster Lie algebra $\gt{m}$ is (``about half'' of) the space of ``physical states'' of a bosonic string moving in the quotient
\begin{gather}
	\left(\RR^{24}/\Lee\oplus \RR^{1,1}/ I\!I_{1,1}\right)/(\Z/2\Z)
\end{gather}
of the $26$-dimensional torus $I\!I_{25,1}\otimes_{\Z}\RR/I\!I_{25,1}\simeq \RR^{24}/\Lee\oplus \RR^{1,1}/ I\!I_{1,1}$ by the involution $(x,y)\mapsto (-x,y)$ (acting as the Kummer involution on the Leech summand, and trivially on the hyperbolic summand). 
The monster Lie algebra $\gt{m}$ is constructed in a functorial way from $\Vnat$ (cf. \cite{Carnahan:2012gx}), inherits an action by the monster from $\Vnat$, and admits a monster-invariant grading by $I\!I_{1,1}$.

The denominator identity for a Kac--Moody algebra $\gt{g}$ equates a product indexed by the positive roots of $\gt{g}$ with a sum indexed by the Weyl group of $\gt{g}$. A BKM algebra also admits a denominator identity, which for the case of the monster Lie algebra $\gt{m}$ is the beautiful {\em Koike--Norton--Zagier formula}
\begin{gather}\label{eqn:class-mondenidid}
	p^{-1}\prod_{\substack{m,n\in\Z\\ m>0}}(1-p^mq^n)^{c(mn)}=
	J(\sigma)-J(\tau),
\end{gather}
where $\sigma\in\HH$ and $p=e^{2\pi i \sigma}$ (and $c(n)$ is the coefficient of $q^n$ in $J(\tau)$, cf. (\ref{eqn:intro-expJ})). Since the right hand side of (\ref{eqn:class-mondenidid}) implies that the left hand side has no terms $p^mq^n$ with $mn\neq 0$, this identity imposes many non-trivial polynomial relations upon the coefficients of $J(\tau)$. Among these is
\begin{gather}\label{eqn:intro:class-recrsn}
	c(4n+2)=c(2n+2)+\sum_{k=1}^nc(k)c(2n-k+1),
\end{gather}
which was first found by Mahler \cite{MR0441867} by a different method, along with similar expressions for $c(4n)$, $c(4n+1)$, and $c(4n+3)$, which are also entailed in (\ref{eqn:class-mondenidid}). Taken together these relations allow us to compute the coefficients of $J(\tau)$ recursively, given just the values 
\begin{gather}
\begin{split}
	c(1)&=196884,\\
	c(2)&=21493760,\\
	c(3)&=864299970,\\
	c(5)&=333202640600.
\end{split}
\end{gather}

To recover the replication formulas of \cite{MR554399,MR760657} we require to replace $J$ with $T_g$, and $c(n)=\dim(\Vnat_n)$ with $\tr(g|\Vnat_n)$ in (\ref{eqn:class-mondenidid}), and for this we require a categorification of the denominator identity, whereby the positive integers $c(mn)$ are replaced with $\MM$-modules of dimension $c(mn)$. 

A categorification of the denominator formula for a finite-dimensional simple complex Lie algebra was obtained by Kostant \cite{MR0142696}, following an observation of Bott \cite{MR0089473}. This was generalized to Kac--Moody algebras by Garland--Lepowsky \cite{MR0414645}, and generalized further to BKM algebras by Borcherds in \cite{borcherds_monstrous}. In its most compact form, it is the identity of virtual vector spaces
\begin{gather}\label{eqn:class-catdenid}
	\bigwedge\nolimits_{-1}(\gt{e})=H(\gt{e}),
\end{gather}
where $\gt{e}$ is the sub Lie algebra of a BKM algebra corresponding to its positive roots (cf. \cite{MR1395597,MR1600542, MR1104219}). 

In (\ref{eqn:class-catdenid}), which is a version of the {\em Euler--Poincar\'e principle},  we understand $\bigwedge_{-1}(\gt{e})$ to be the specialization of the formal series
\begin{gather}\label{eqn:class-bigwedget}
\bigwedge\nolimits_t(\gt{e}):=\sum_{k\geq 0}\wedge^k(\gt{e})t^k
\end{gather}
to $t=-1$, where $\wedge^k(\gt{e})$ is the $k$-th exterior power of $\gt{e}$, and we write 
\begin{gather}
H(\gt{e}):=\sum_{k\geq 0} (-1)^kH_k(\gt{e})
\end{gather} 
for the alternating sum of the Lie algebra homology groups of $\gt{e}$. 

In the case of the monster Lie algebra $\gt{m}$, the spaces $\wedge^k(\gt{e})$ and $H_k(\gt{e})$ are graded by $I\!I_{1,1}$, and acted on naturally by the monster. If we use the variables $p$ and $q$ to keep track of the $I\!I_{1,1}$-gradings, then the equality of (\ref{eqn:class-catdenid}) holds in the ring $R(\MM)[[p,q]][q^{-1}]$ of formal power series in $p$ and $q$ (allowing finitely many negative powers of $q$), with coefficients in the (integral) representation ring of $\MM$. The so-called no-ghost theorem (cf. Theorem 5.1 in \cite{borcherds_monstrous}) allows to express the $H_k(\gt{e})$ in terms of the homogenous subspaces of $\Vnat$, and the identity (\ref{eqn:class-catdenid}) becomes
\begin{gather}\label{eqn:class-catdenidm}
\bigwedge\nolimits_{-1}\left(\sum_{\substack{m,n\in\Z\\m>0}}\Vnat_{mn}p^mq^n\right)
=
\sum_{m\in\Z}\Vnat_mp^{m+1}-\sum_{n\in\Z}\Vnat_npq^n.
\end{gather}
(It turns out that $H_k(\gt{e})=0$ for $k\geq 3$.)

The identity (\ref{eqn:class-catdenidm}) returns (\ref{eqn:class-mondenidid}) once we replace $\Vnat_{k}$ everywhere with $\dim (\Vnat_k)=c(k)$, and divide both sides by $p$. More generally, replacing $\Vnat_k$ with $\tr(g|\Vnat_k)$ for $g\in \MM$, we deduce from (\ref{eqn:class-catdenidm}) that
\begin{gather}\label{eqn:class-mondenidg}
	p^{-1}\exp\left(-\sum_{k>0}\sum_{\substack{m,n\in\Z\\m>0}}\frac{1}{k}\tr(g^k|\Vnat_{mn})p^{mk}q^{nk}\right)
	=
	T_g(\sigma)-T_g(\tau)
\end{gather}
(cf. \cite{borcherds_monstrous}, and also \cite{MR1327230}), which, in turn, implies the replication formulas formulated in \cite{MR554399,MR760657}. Taking $g=e$ in (\ref{eqn:class-mondenidg}) we recover (\ref{eqn:class-mondenidid}), so (\ref{eqn:class-mondenidg}) furnishes a natural, monster-indexed family of analogues of the identity (\ref{eqn:class-mondenidid}).

\section{Modularity}\label{sec:intro:mod}

Despite the power of the BKM algebra 
theory developed by Borcherds, and despite some conceptual improvements (cf. \cite{MR1465329,MR1600542,MR2681781}) upon Borcherds' original proof of the moonshine conjectures, a conceptual explanation for the principal modulus property of monstrous moonshine is yet to be established.  Indeed, there are generalizations and analogs of the notion of replicability which hold for generic modular functions and forms (for example, see \cite{BruinierKohnenOno}), not just those modular functions which are principal moduli.

Zhu explained the modularity of the graded dimension of $\Vnat$ in \cite{Zhu_ModInv}, by proving that this is typical for vertex operator algebras satisfying certain natural (but restrictive) hypotheses, and Dong--Li--Mason extended Zhu's work in \cite{Dong2000}, obtaining modular invariance results for graded trace functions arising from the action of a finite group of automorphisms. 

To prepare for a statement of the results of Zhu and Dong--Li--Mason, we mention that the module theory for vertex operators algebras includes so-called {\em twisted modules}, associated to finite order automorphisms. If $g$ is a finite order automorphism of $V$, then $V$ is called {\em $g$-rational} in case every $g$-twisted $V$-module is a direct sum of simple $g$-twisted $V$-modules. Dong--Li--Mason proved \cite{MR1615132} that a $g$-rational vertex operator algebra has finitely many simple $g$-twisted modules up to isomorphism. So in particular, a rational vertex operator algebra has finitely many simple (untwisted) modules.

Given a module $M$ for the vertex operator algebra $V$, let us write $\dimq M$ for its graded dimension, 
\begin{gather}\label{eqn:mod-dimqM}
	\dimq M:=\sum_n\dim(M_n)q^n.
\end{gather}
When the substitution $q=e^{2\pi i \tau}$ in (\ref{eqn:mod-dimqM}) yields a locally uniformly convergent series for $\tau\in\HH$ write $Z_M(\tau)$ for the resulting holomorphic function on the upper half plane,
\begin{gather}\label{eqn:mod-ZM}
	Z_M(\tau):=\dimq M|_{q=e^{2\pi i \tau}}.
\end{gather}

\begin{theorem}[Zhu, Dong--Li--Mason]\label{thm:intro:mod-mod}
Let $V$ be rational $C_2$-cofinite vertex operator algebra. Then the graded dimensions $\dimq M^i$ of its simple modules $M^i$ define holomorphic functions $Z_{M^i}(\tau)$ which span a finite-dimensional representation of $\SL_2(\Z)$. More generally, if $G$ is a finite subgroup of ${\rm Aut}(V)$ and $V$ is $g$-rational for every $g\in G$, then the graded trace functions $\sum_n \tr(\tilde{h}|M_n)q^n$, attached to the triples $(g,\tilde{h},M)$, where $g,h\in G$ commute, $M$ is a simple $h$-stable $g$-twisted module for $V$, and $\tilde{h}$ is a lift of $h$ to $\GL(M)$, span a finite-dimensional representation of $\SL_2(\Z)$.
\end{theorem}
We refer to the Introduction of \cite{Dong2000} (see also \S2 of \cite{MR1284796}) for a discussion of $h$-stable twisted modules, and the relevant notion of  {\it lift}. Note that any two lifts for $h$ differ only up to multiplication by a non-zero scalar, so $\sum_n \tr(\tilde{h}|M_n)q^n$ is uniquely defined by $(g,h,M)$, up to a non-zero scalar.

In the case of $\Vnat$, there is a unique simple $g$-twisted module $\Vnat_g=\bigoplus_n(\Vnat_g)_{n}$ for every $g\in \MM={\rm Aut}(\Vnat)$ (cf. Theorem 1.2 of \cite{MR1615132}), and $\Vnat_g$ is necessarily $h$-stable for any $h\in\MM$ that commutes with $g$.  Therefore, Theorem \ref{thm:intro:mod-mod} suggests that the functions 
\begin{gather}\label{eqn:intro:mod-Tgh}
	T_{(g,\tilde h)}(\tau):=\sum_n \tr(\tilde{h}|(\Vnat_g)_n)q^n,
\end{gather}
associated to pairs $(g,h)$ of commuting elements of $\MM$, may be of interest. 

Indeed, this was anticipated a decade earlier by Norton, following observations of Conway--Norton \cite{MR554399} and Queen \cite{MR628715}, which associated  principal moduli to elements of groups that appear as centralizers of cyclic subgroups in the monster. Norton formulated his {\em generalized moonshine} conjectures in \cite{generalized_moonshine} (cf. also \cite{MR1877765}, and the Appendix to \cite{MR933359}).

\begin{conjecture}[Generalized Moonshine: Norton]\label{conj:GeneralizedMoonshine}
There is an assignment of holomorphic functions $T_{(g,\tilde{h})}:\HH\to\CC$ to every pair $(g,h)$ of commuting elements in the monster, such that 
the following are true:
\begin{enumerate}
\item For every $x\in \MM$ we have  $T_{(x^{-1}gx,x^{-1}\tilde{h}x)}=T_{(g,\tilde h)}$.  
\item  For every $\gamma\in \SL_2(\Z)$ we have that $T_{(g,\tilde{h})\gamma}(\tau)$ is a scalar multiple of $T_{(g,\tilde{h})}(\gamma\tau)$. 
\item The coefficient functions $\tilde{h}\mapsto \tr(\tilde{h}|(\Vnat_g)_n)$, for fixed $g$ and $n$, define characters of a projective representation of the centralizer of $g$ in $\MM$,
\item We have that $T_{(g,\tilde{h})}$ is either constant or a generator for the function field of a genus zero group $\Gamma_{(g,h)}<\SL_2(\RR)$.
\item We have that $T_{(g,\tilde{h})}$ is a scalar multiple of $J$ if and only if $g=h=e$.
\end{enumerate}
\end{conjecture}

\begin{remark}
In Conjecture~\ref{conj:GeneralizedMoonshine} (2) above, the right-action of $\SL_2(\Z)$ on commuting pairs of elements of the monster is given by 
\begin{gather}
(g,h)\gamma:=(g^ah^c,g^bh^d)
\end{gather}
for $\gamma=\left(\begin{smallmatrix} a&b\\c&d\end{smallmatrix}\right)$. The (slightly ambiguous) $T_{(g,\tilde{h})\gamma}$ denotes the graded trace of a lift of $g^bh^d$ to $\GL(\Vnat_{g^ah^c})$. Norton's generalized moonshine conjectures reduce to the original Conway--Norton moonshine conjectures of \cite{MR554399} when $g=e$. 
\end{remark}

Conjecture~\ref{conj:GeneralizedMoonshine} is yet to be proven in full, but has been established for a number of special cases. Theorem \ref{thm:intro:mod-mod} was used by Dong--Li--Mason in \cite{Dong2000}, following an observation of Tuite (cf. \cite{MR1372716}, and \cite{Tuite_CMP92,Tuite_CMP95,Tuite_CM95} for broader context), to prove Norton's conjecture for the case that $g$ and $h$ generate a cyclic subgroup of $\MM$, and this approach, via twisted modules for $\Vnat$, has been extended by Ivanov--Tuite in \cite{MR1915258,MR1915259}. H\"ohn obtained a generalization of Borcherds' method by using a particular twisted module for $\Vnat$ to construct a BKM algebra adapted to the case that $g$ is in the class named $2A$ in \cite{ATLAS}---the smaller of the two conjugacy classes of involutions in $\MM$---and in so doing established \cite{Hoe_GenMnsBbyMns} generalized moonshine for the functions $T_{(g,\tilde{h})}$ with $g\in 2A$.
So far the most general results in generalized moonshine have been obtained by Carnahan \cite{MR2728485,MR2904095,Carnahan:2012gx}. (See \cite{Car_MnsLieAlg} for a recent summary.)

Theorem \ref{thm:intro:mod-mod} explains why the McKay--Thompson series $T_g(\tau)$ of (\ref{eqn:intro-Tg}), and the $T_{(g,\tilde{h})}(\tau)$ of (\ref{eqn:intro:mod-Tgh}) more generally, should be invariant under the actions of (finite index) subgroups of $\SL_2(\Z)$, but it does not explain the surprising predictive power of monstrous moonshine. That is, it does not explain why the full invariance groups $\Gamma_g$ of the $T_g$ should be so large that they admit normalized principal moduli, nor does it explain why the $T_g$ should actually be these normalized principal moduli.

A program to establish a conceptual foundation for the principal modulus property of monstrous moonshine, via {\it Rademacher sums} and
{\it  three-dimensional gravity}, was initiated in \cite{DunFre_RSMG} by the first author and Frenkel.

\section{Rademacher Sums}\label{sec:intro:radsums}

To explain the conjectural connection between gravity and moonshine, we 
first recall some history.
The roots of the approach of \cite{DunFre_RSMG} extend back almost a hundred years, to Einstein's theory of general relativity, formulated in 1915, and the introduction of the circle method in analytic number theory, by Hardy--Ramanujan \cite{MR2280879}. At the same time that pre-war efforts to quantize Einstein's theory of gravity were gaining steam (see \cite{MR1663704} for a review), the circle method was being refined and developed, by Hardy--Littlewood (cf. \cite{MR0201267}), and  Rademacher \cite{Rad_PtnFn}, among others. (See \cite{MR1435742} for a detailed account of what is now known as the {\em Hardy--Littlewood circle method}.) Despite being contemporaneous, these works were unrelated in science until this century: as we will explain presently, Rademacher's analysis led to a Poincar\'e series-like expression---the prototypical Rademacher sum---for the elliptic modular invariant $J(\tau)$. 
It was suggested first in \cite{Dijkgraaf2007} (see also \cite{Manschot2007}) that this kind of expression might be useful for the computation of partition functions in quantum gravity.

Rademacher ``perfected''   the circle method introduced by Hardy--Ramanujan, and he obtained an exact convergent series expression for the combinatorial partition function $p(n)$.
In 1938 he generalized this work \cite{Rad_FouCoeffMdlrInv} and obtained such exact formulas
for the Fourier coefficients of general modular functions.   For the elliptic modular invariant $J(\tau)=\sum_n c(n)q^n$ (cf. (\ref{eqn:intro-expJ})), Rademacher's formula (which was obtained earlier by Petersson \cite{Pet_UbrEntAutFrm}, via a different method) may be written as
\begin{gather}\label{eqn:intro:qgrav-Radcn}
	c(n)=4\pi^2\sum_{c>0}\sum_{\substack{0<a<c\\(a,c)=1}}
	\frac{e^{-2\pi i \frac{a}{c}}e^{2\pi i n\frac{d}{c}}}
	{c^2}\sum_{k\geq 0}
	\frac{(4\pi^2)^{k}}{c^{2k}}\frac{1}{(k+1)!}\frac{n^k}{k!},
\end{gather}
where $d$, in each summand, is a multiplicative inverse for $a$ modulo $c$, and $(a,c)$ is the greatest common divisor of $a$ and $c$.
Having established the formula (\ref{eqn:intro:qgrav-Radcn}), Rademacher sought to reverse the process, and use it to derive the modular invariance of $J(\tau)$. That is, he set out to prove directly that $J_0(\tau+1)=J_0(-1/\tau)=J_0(\tau)$, when $J_0(\tau)$ is defined by setting $J_0(\tau)=q^{-1}+\sum_{n>0}c(n)q^n$, with $c(n)$ defined by (\ref{eqn:intro:qgrav-Radcn}).

Rademacher achieved this goal in \cite{Rad_FuncEqnModInv}, by reorganizing the summation
\begin{gather}
	\sum_{n>0}c(n)q^n=
	4\pi^2
	\sum_{n>0}
	\sum_{c>0}\sum_{\substack{0<a<c\\(a,c)=1}}
	\frac{e^{-2\pi i \frac{a}{c}}e^{2\pi i n(\tau+\frac{d}{c})}}
	{c^2}\sum_{k\geq 0}
	\frac{(4\pi^2)^{k}}{c^{2k}}\frac{1}{(k+1)!}\frac{n^k}{k!}
\end{gather}
into a Poincar\'e series-like expression for $J$. More precisely, Rademacher proved that
\begin{gather}\label{eqn:intro:qgrav-JRadsum}
	J(\tau)+12=e^{-2\pi i\tau}+\lim_{K\to \infty}\sum_{\substack{0<c<K\\-K^2<d<K^2\\(c,d)=1}}
	e^{-2\pi i\frac{a\tau+b}{c\tau+d}}-e^{-2\pi i\frac{a}{c}},
\end{gather}
where $a,b\in\Z$ are chosen arbitrarily, in each summand, so that $ad-bc=1$. We call the right hand side of (\ref{eqn:intro:qgrav-JRadsum}) the first {\em Rademacher sum}.

Rademacher's expression (\ref{eqn:intro:qgrav-JRadsum}) for the elliptic modular invariant $J$ is to be compared to the formal sum 
\begin{gather}\label{eqn:intro:qgrav-Jformal}
	\sum_{\substack{c,d\in\Z\\
	(c,d)=1}}e^{-2\pi i m \frac{a\tau+b}{c\tau+d}},
\end{gather}
for $m$ a positive integer, which we may regard as a (formal) Poincar\'e series of weight zero for $\SL_2(\Z)$.
In particular, (\ref{eqn:intro:qgrav-Jformal}) is (formally) invariant for the action of $\SL_2(\Z)$, as we see by recognizing the matrices $\left(\begin{smallmatrix} a&b\\c&d\end{smallmatrix}\right)$ as representatives for the right coset space $\Gamma_\infty\backslash\Gamma$, where $\Gamma=\SL_2(\Z)$ and $\Gamma_\infty$ is defined in (\ref{eqn:intro:class-Gammainfty}):
for a fixed bottom row $(c,d)$ of matrices in $\SL_2(\Z)$, any two choices for the top row $(a,b)$ are related by left-multiplication by some element of $\Gamma_\infty$. 

The formal sum (\ref{eqn:intro:qgrav-Jformal}) does not converge for any $\tau\in \HH$, so a regularization procedure is required. Rademacher's sum (\ref{eqn:intro:qgrav-JRadsum}) achieves this, for $m=1$, by constraining the order of summation, and subtracting the limit as $\Im(\tau)\to\infty$ of each summand $e^{-2\pi i \frac{a\tau+b}{c\tau+d}}$, whenever this limit makes sense. 
Rademacher's method has by now been generalized in various ways by a number of authors. The earliest generalizations are due to Knopp \cite{Kno_ConstMdlrFnsI,Kno_ConstMdlrFnsII,Kno_ConstAutFrmsSuppSeries,Kno_AbIntsMdlrFns}, and a very general negative weight version of the Rademacher construction was given by Niebur in \cite{Nie_ConstAutInts}. We refer to \cite{Cheng:2012qc} for a detailed review and further references. A nice account of the original approach of Rademacher appears in \cite{Kno_RadonJPoinSerNonPosWtsEichCohom}. 

We note here that one of the main difficulties in establishing formulas like (\ref{eqn:intro:qgrav-JRadsum}) is the demonstration of convergence. When the weight $w$ of the Rademacher sum under consideration lies in the range $0\leq w\leq 2$, then one requires non-trivial estimates on {\em sums of Kloosterman sums}, like  
\begin{gather}
\sum_{c>0}
\sum_{\substack{0<a<c\\(a,c)=1}}
	\frac{e^{-2\pi i m \frac{a}{c}}e^{2\pi i n\frac{d}{c}}}{c^{2}}
\end{gather}
(for the case that $w=0$ or $w=2$). The demonstration of convergence generally becomes more delicate as $w$ approaches $1$.

In \cite{DunFre_RSMG} the convergence of a weight zero Rademacher sum $R_\Gamma^{(-m)}(\tau)$ is shown, for $m$ a positive integer and $\Gamma$ an arbitrary subgroup of $\SL_2(\RR)$ that is commensurable with $\SL_2(\Z)$. Assuming that $\Gamma$ contains $-I$ and has width one at infinity (cf. (\ref{eqn:intro:class-Gammainfty})), we have
\begin{gather}\label{eqn:intro:qgrav-RGamma}
	R_{\Gamma}^{(-m)}(\tau)=e^{-2\pi i m \tau}+\lim_{K\to \infty}
	\sum_{	(\Gamma_\infty\backslash\Gamma)_{<K}^\times}
	e^{-2\pi i m \frac{a\tau+b}{c\tau+d}}-e^{-2\pi i m \frac{a}{c}},
\end{gather}
where the summation, for fixed $K$, is over non-trivial right cosets of $\Gamma_\infty$ in $\Gamma$ (cf. (\ref{eqn:intro:class-Gammainfty})), having representatives $\left(\begin{smallmatrix}a&b\\c&d\end{smallmatrix}\right)$ such that $0<c<K$ and $|d|<K^2$. 

The modular properties of the $R_{\Gamma}^{(-m)}$ are also considered in \cite{DunFre_RSMG}, and it is at this point that the significance of Rademacher sums in monstrous moonshine appears. 
To state the relevant result we give the natural generalization (cf. \S3.2 of \cite{DunFre_RSMG}) of the Rademacher--Petersson formula (\ref{eqn:intro:qgrav-Radcn}) for $c(n)$, which is
\begin{gather}\label{eqn:intro:qgrav-cGamman}
	c_\Gamma(-m,n)=
	4\pi^2\lim_{K\to \infty}\sum_{(\Gamma_\infty\backslash\Gamma/\Gamma_\infty)_{<K}^\times}
\frac{e^{-2\pi i m \frac{a}{c}}e^{2\pi i n\frac{d}{c}}}
	{c^2}\sum_{k\geq 0}
	\frac{(4\pi^2)^{k}}{c^{2k}}\frac{m^{k+1}}{(k+1)!}\frac{n^k}{k!},
\end{gather}
where the summation, for fixed $K$, is over non-trivial double cosets of $\Gamma_\infty$ in $\Gamma$ (cf. (\ref{eqn:intro:class-Gammainfty})), having representatives $\left(\begin{smallmatrix}a&b\\c&d\end{smallmatrix}\right)$ such that $0<c<K$. Note that this formula simplifies for $n=0$, to 
\begin{gather}\label{eqn:intro:qgrav-cGamma0}
	c_{\Gamma}(-m,0)=4\pi^2m\lim_{K\to \infty}\sum_{(\Gamma_\infty\backslash\Gamma/\Gamma_\infty)_{<K}^\times}
	\frac{e^{-2\pi i m\frac{a}{c}}}{c^2}.
\end{gather}
The value $c_{\Gamma}(-1,0)$ is the {\em Rademacher constant} attached to $\Gamma$. (Cf. \S6 of \cite{MR760657} and \S5.1 of \cite{DunFre_RSMG}.) 

A {\em normalized Rademacher sum} $T_\Gamma^{(-m)}(\tau)$ is defined in \S4.1 of \cite{DunFre_RSMG} by introducing an extra complex variable and taking a limit. It is shown in \S4.4 of \cite{DunFre_RSMG} that
\begin{gather}\label{eqn:intro:radsums-TRc}
T_\Gamma^{(-m)}(\tau)=R_\Gamma^{(-m)}(\tau)-\frac12c_{\Gamma}(-m,0)
\end{gather}
for any group $\Gamma<\SL_2(\RR)$ that is commensurable with $\SL_2(\Z)$. If $\Gamma$ has width one at infinity (cf. (\ref{eqn:intro:class-Gammainfty})), then also
\begin{gather}\label{eqn:intro:radsums-TmGammaFouExp}
	T_\Gamma^{(-m)}(\tau)=q^{-m}+\sum_{n>0}c_\Gamma(-m,n)q^n,
\end{gather}
so in particular, $T_\Gamma^{(-m)}(\tau)=q^{-m}+O(q)$ as $\Im(\tau)\to\infty$.
The following theorem by the first author and Frenkel summarizes the central role of Rademacher sums and 
the principal modulus property.

\begin{theorem}[Duncan--Frenkel \cite{DunFre_RSMG}]
\label{thm:intro:qgrav-Rgammainv}
Let $\Gamma$ be a subgroup of $\SL_2(\RR)$ that is commensurable with $\SL_2(\Z)$. 
Then the normalized Rademacher sum $T_\Gamma^{(-m)}$ is $\Gamma$-invariant if and only if $\Gamma$ has genus zero. Furthermore, if $\Gamma$ has genus zero then $T_\Gamma^{(-1)}$ is the normalized principal modulus for $\Gamma$. 
\end{theorem}

In the case that the normalized Rademacher sum $T_\Gamma^{(-1)}$ is not $\Gamma$-invariant, $T_\Gamma^{(-m)}$ is an {\em abelian integral of the second kind} for $\Gamma$, in the sense that it has at most exponential growth at the cusps of $\Gamma$, and satisfies $T_{\Gamma}^{(-m)}(\gamma\tau)=T_\Gamma^{(-m)}(\tau)+\omega(\gamma)$ for $\gamma\in \Gamma$, for some function $\omega:\Gamma\to \CC$ (depending on $m$). 

Theorem \ref{thm:intro:qgrav-Rgammainv} is used as a basis for the formulation of a characterization of the discrete groups $\Gamma_g$ of monstrous moonshine in terms of Rademacher sums in \S6.5 of \cite{DunFre_RSMG}, following earlier work \cite{ConMcKSebDiscGpsM} of Conway--McKay--Sebbar. It also facilitates a proof of the following result, which constitutes a uniform construction of the McKay--Thompson series of monstrous moonshine.
\begin{theorem}[Duncan--Frenkel \cite{DunFre_RSMG}]\label{thm:radsums-TgTGammag}
Let $g\in \MM$. Then the McKay--Thompson series $T_g$ coincides with the normalized Rademacher sum $T_{\Gamma_g}^{(-1)}$.
\end{theorem}
\begin{proof}
Theorem \ref{thm:intro:class-borcherds} states that $T_g$ is a normalized principal modulus for $\Gamma_g$, and in particular, all the $\Gamma_g$ have genus zero. Given this, it follows from Theorem \ref{thm:intro:qgrav-Rgammainv} that $T_{\Gamma_g}^{(-1)}$ is also a normalized principal modulus for $\Gamma_g$. A normalized principal modulus is unique if it exists, so we conclude $T_g=T_{\Gamma_g}^{(-1)}$ for all $g\in \MM$, as we required to show.
\end{proof}

Perhaps most importantly, Theorem \ref{thm:intro:qgrav-Rgammainv} is an indication of how the principal modulus property of monstrous moonshine can be explained conceptually. For if we can develop a mathematical theory in which the underlying objects are graded with graded traces that are provably
\begin{enumerate}
\item modular invariant, for subgroups of $\SL_2(\RR)$ that are commensurable with $\SL_2(\Z)$, and
\item given explicitly by Rademacher sums, such as (\ref{eqn:intro:qgrav-RGamma}),
\end{enumerate}
then these graded trace functions are necessarily normalized principal moduli, according to Theorem \ref{thm:intro:qgrav-Rgammainv}.

We are now led to ask: what kind of mathematical theory can support such results? As we have alluded to above, Rademacher sums have been related to quantum gravity by articles in the physics literature. Also, a possible connection between the monster and three-dimensional quantum gravity was discussed in \cite{Witten2007}. This suggests the possibility that three-dimensional quantum gravity and moonshine are related via Rademacher sums, and was a strong motivation for the work \cite{DunFre_RSMG}. In the next section we will give a brief review of quantum gravity, since it is an important area of physical inquiry which has played a role in the development of moonshine, 
but we must first warn the reader: problems have been identified with the existing conjectures that relate the monster to gravity, and the current status of this connection is uncertain.  

\section{Quantum Gravity}\label{sec:intro:qgrav}

Witten was the first to predict a role for the monster in quantum gravity. In \cite{Witten2007} Witten considered {\em pure quantum gravity} in three dimensions with negative cosmological constant, and presented evidence that the moonshine module $\Vnat$ is a chiral half of the conformal field theory dual to such a quantum gravity theory, at the most negative possible value of the cosmological constant. 

To explain some of the content of this statement, 
note that the action in pure three-dimensional quantum gravity is given explicitly by 
\begin{gather}\label{eqn:QG-IEH}
	I_{\rm EG}:=\frac{1}{16\pi G}\int {\rm d}^3x\sqrt{-g}(R-2\Lambda),
\end{gather}
where $G$ is the {\em Newton} or {\em gravitational constant}, $R$ denotes the Ricci scalar, and the {\em cosmological constant} is the scalar denoted by $\Lambda$. 

The case that the cosmological constant $\Lambda$ is negative is distinguished, since then there exist black hole solutions to the action (\ref{eqn:QG-IEH}), as was discovered by Ba\~nados--Teitelboim--Zanelli \cite{BanTeiZan_BTZBlckHle}. These black hole solutions---the {\em BTZ black holes}---are locally isomorphic to three-dimensional anti-de Sitter space \cite{MR1236812}, which is a Lorentzian analogue of hyperbolic space, and may be realized explicitly as the universal cover of a hyperboloid 
\begin{gather}\label{eqn:intro:qgrav-hyp}
	-X_{-1}^2-X_0^2+X_1^2+X_2^2+X_3^3=-\ell^2
\end{gather} 
in $\RR^{2,3}$ (cf. e.g. \cite{2012qchs.confE...9M}). The parameter $\ell$ in (\ref{eqn:intro:qgrav-hyp}) is called the {\em radius of curvature}. For a locally anti-de Sitter (AdS) solution to (\ref{eqn:QG-IEH}), the radius of curvature is determined by the cosmological constant, according to 
\begin{gather}\label{eqn:intro:qgrav-ellLambda}
\ell^2={-1/\Lambda}.
\end{gather}

In what has become the most cited\footnote{Maldacena's groundbreaking paper \cite{MaldacenaAdv.Theor.Math.Phys.2:231-2521998} on the gauge/gravity duality has over 10,000 citations at the time of writing, according to {\tt inspirehep.net}.} paper in the history of high energy physics, Maldacena opened the door on a new, and 
powerful approach to quantum gravity in \cite{MaldacenaAdv.Theor.Math.Phys.2:231-2521998}, by presenting 
evidence for a gauge/gravity duality, in which gauge theories serve as duals to gravity theories in one dimension higher. (See \cite{2012qchs.confE...9M} for a recent review.) In the simplest examples, the gauge theories are conformal field theories, and the gravity theories involve locally AdS spacetimes. The gauge/gravity duality for these cases is now known as the {\em AdS/CFT correspondence}. 

Maldacena's duality furnishes a concrete realization of the {\em holographic principle}, introduced by 't Hooft \cite{1993gr.qc....10026T}, and elaborated on by Susskind \cite{1995JMP....36.6377S}. For following refinements to Maldacena's proposal due to Gubser--Klebanov--Polyakov \cite{Gubser:1998bc}, and Witten \cite{Witten:1998qj}, it is expected that gravity theories with $(d+1)$-dimensional locally AdS spacetimes can be understood through the analysis of $d$-dimensional conformal field theories defined on the boundaries of these AdS spaces. Thus in the case of AdS solutions to three-dimensional quantum gravity, a governing role may be played by two-dimensional conformal theories, which can be accessed mathematically via vertex operator algebras (as we have mentioned in \S\ref{previouswork}). 

The conjecture of \cite{Witten2007} is that the two-dimensional conformal field theory corresponding to a tensor product of two copies of the moonshine module $\Vnat$ (one ``left-moving,'' the other ``right-moving'') is the holographic dual to pure three-dimensional quantum gravity with $\ell=16G$, and therefore 
\begin{gather}\label{eqn:intro:qgrav-minLambda}
\Lambda=-\frac{1}{256G^2}.
\end{gather}
It is also argued that the only physically consistent values of $\ell$ are $\ell=16Gm$, for $m$ a positive integer, so that (\ref{eqn:intro:qgrav-minLambda}) is the most negative possible value for $\Lambda$, by force of (\ref{eqn:intro:qgrav-ellLambda}).

Shortly after this conjecture was formulated, problems with the quantum mechanical interpretation were identified by Maloney--Witten in \cite{MalWit_QGPtnFn3D}. 
Moreover, Gaiotto \cite{Gaiotto:2008jt} and H\"ohn \cite{Hoe_SDVOSALgeMinWt} cast doubt on the relevance of the monster to gravity by demonstrating  that it cannot act on a holographically dual conformal field theory corresponding to $\ell=32G$ (i.e. $m=2$), at least under the hypotheses (namely, an extremality condition, and holomorphic factorization) presented in \cite{Witten2007}. 

Interestingly, the physical problems with the analysis of \cite{Witten2007} seem to disappear in the context of {\em chiral three-dimensional gravity}, which was introduced 
and discussed in detail by Li--Song--Strominger in \cite{LiSonStr_ChGrav3D} (cf. also \cite{Maloney:2009ck,Strominger:2008dp}). 
This is the gravity theory which motivates much of the discussion in \S7 of \cite{DunFre_RSMG}.

In order to define chiral three-dimensional gravity, we first describe {\em topologically massive gravity}, which was introduced in 1982 by  Deser--Jackiw--Templeton \cite{MR665601,Deser:1982vy}. (See also \cite{MR965585}.) The action for topologically massive gravity is given by 
\begin{gather}\label{eqn:intro:qgrav-ITMG}
I_{\rm TMG}:=I_{\rm EG}+I_{\rm CSG}, 
\end{gather}
where $I_{\rm EG}$ is the Einstein--Hilbert action (cf. (\ref{eqn:QG-IEH})) of pure quantum gravity, and $I_{\rm CSG}$ denotes the {\em gravitational Chern--Simons term} 
\begin{gather}\label{eqn:QG-ICS}
	I_{\rm CSG}
	:=\frac{1}{32\pi G\mu}\int{\rm d}^3x\sqrt{-g}\epsilon^{\lambda\mu\nu}\Gamma^{\rho}_{\;\lambda\sigma}
	\left(
	\partial_{\mu}\Gamma^\sigma_{\;\rho\nu}+\frac23\Gamma^\sigma_{\;\mu\tau}\Gamma^\tau_{\;\nu\rho}
	\right).
\end{gather}
The $\Gamma^{*}_{\;**}$ are Christoffel symbols, and the parameter $\mu$ is called the {\em Chern--Simons coupling constant}. 

{Chiral three-dimensional gravity} is the special case of topologically massive gravity in which the Chern--Simons coupling constant is set to $\mu=1/\ell=\sqrt{-\Lambda}$. It is shown in \cite{LiSonStr_ChGrav3D} that at this special value of $\mu$, the left-moving central charges of the boundary conformal field theories vanish, and the right-moving central charges are 
\begin{gather}\label{eqn:intro:qgrav-c24m}
c=\frac{3\ell}{2G}=24m,
\end{gather}
for $m$ a positive integer, $\ell=16Gm$. 

Much of the analysis of \cite{Witten2007} still applies in this setting, and the natural analogue of the conjecture mentioned above states that $\Vnat$ is holographically dual to chiral three-dimensional quantum gravity at $\ell=16G$, i.e. $m=1$. However, as argued in detail in \cite{Maloney:2009ck}, the problem of quantizing chiral three-dimensional gravity may be regarded as equivalent to the problem of constructing a sequence of extremal chiral two-dimensional conformal field theories (i.e. vertex operator algebras), one for each central charge $c=24m$, for $m$ a positive integer. Here, a vertex operator algebra $V=\bigoplus_n V_n$ with central charge $c=24m$ 
is called {\em extremal}, if its  graded\footnote{We regard all vertex operator algebras as graded by $L(0)-{\bf c}/24$. Cf. (\ref{eqn:intro:class-vir}).} dimension function satisfies
\begin{gather}\label{eqn:qgrav-ext}
	\dimq V =
	q^{-m}\frac{1}{\prod_{n>1}(1-q^n)}+O(q)
\end{gather}
(cf. (\ref{eqn:mod-dimqM}).) The moonshine module is the natural candidate for $m=1$ (indeed, it is the only candidate if we assume the uniqueness conjecture of \cite{FLM}), as the right hand side of (\ref{eqn:qgrav-ext}) reduces to $q^{-1}+O(q)$ in this case, but the analysis of \cite{Gaiotto:2008jt,Hoe_SDVOSALgeMinWt} also applies here, indicating that the monster cannot act non-trivially on any candidate\footnote{The existence of extremal vertex operator algebras with central charge $c=24m$ for $m>1$ remains an open question. We refer to \cite{MR2725054,Gaiotto:2007xh,Hoe_SDVOSALgeMinWt,MR2442775} for analyses of this problem.} for $m=2$. Thus the role of the monster in quantum gravity is still unclear, even in the more physically promising chiral gravity setting.

Nonetheless, the moonshine module $\Vnat$ may still serve as the holographic dual to chiral three-dimensional quantum gravity at $\ell=16G$, $m=1$. In this interpretation, the graded dimension, or {\em genus one partition function} for $\Vnat$---namely, the elliptic modular invariant $J$---serves as the exact spectrum of physical states of chiral three-dimensional gravity 
at $\mu=\sqrt{-\Lambda}=1/16G$, in spacetime asymptotic to the three-dimensional anti-de Sitter space (cf. (\ref{eqn:intro:qgrav-hyp})). 

Recall that if $V$ is a representation of the Virasoro algebra $\Vir$ (cf. (\ref{eqn:intro:class-vir})), then $v\in V$ is called a {\em Virasoro highest weight vector} with {\em highest weight} $h\in \CC$ if $L(m)v=h\delta_{m,0}v$ whenever $m\geq 0$. A {\em Virasoro descendant} is a vector of the form
\begin{gather}\label{eqn:intro:qgrav-virdesc}
	L(m_1)\cdots L(m_k)v,
\end{gather}
where $v$ is a Virasoro highest weight vector, and $m_1\leq \cdots\leq m_k\leq -1$.

Assuming 
that $\Vnat$ is dual to chiral three-dimensional gravity at $m=1$, the Virasoro highest weight vectors in $\Vnat$ define operators that create black holes, and the Virasoro descendants of a highest weight vector describe black holes embellished by boundary excitations. In particular, the $196883$-dimensional representation of the monster which is contained in the $196884$-dimensional homogenous subspace $\Vnat_1<\Vnat$ (cf. (\ref{eqn:intro-mckobs}) and (\ref{eqn:intro-dimVnatcn})), represents an $196883$-dimensional space of black hole states in the chiral gravity theory. 

More generally, the black hole states in the theory are classified, by the monster, into $194$ different kinds, according to which monster irreducible representation they belong to. 

\begin{question} 
Assuming that the moonshine module $\Vnat$ serves as the holographic dual to chiral three-dimensional quantum gravity at $m=1$, how are the $194$ different kinds of black hole states distributed amongst the homogeneous subspaces $\Vnat_n<\Vnat$. Are some kinds of black holes more or less common than others? 
\end{question}

This question will be answered precisely in \S\ref{sec:mmdist}.

A positive solution to the conjecture 
that $\Vnat$ is dual to chiral three-dimensional gravity 
at $m=1$ may furnish a conceptual explanation for why the graded dimension of $\Vnat$ 
is the normalized principal modulus for $\SL_2(\Z)$. For on the one hand, modular invariance is a consistency requirement of the physical theory---the genus one partition function function is really defined on the moduli space $\SL_2(\Z)\backslash\HH$ of genus one curves, rather than on $\HH$---and on the other hand, the genus one partition function of chiral three-dimensional gravity is given by a Rademacher sum, as explained by Manschot--Moore \cite{Manschot2007}, following earlier work \cite{Dijkgraaf2007} by Dijkgraaf--Maldacena--Moore--Verlinde. (Cf. also \cite{Maloney:2009ck,MalWit_QGPtnFn3D,Man_AdS3PFnsRecon}.) So, as we discussed in \S\ref{sec:intro:radsums}, the genus one partition function must be the normalized principal modulus $J(\tau)$ for $\SL_2(\Z)$, according to Theorem \ref{thm:intro:qgrav-Rgammainv}.

In the analysis of \cite{Man_AdS3PFnsRecon,Manschot2007}, the genus one partition function of chiral three-dimensional gravity is a Rademacher sum (\ref{eqn:intro:qgrav-RGamma}), because it is obtained as a sum over three-dimensional hyperbolic structures on a solid torus with genus one boundary, 
and such structures are naturally parameterized by the coset space $\Gamma_\infty\backslash\SL_2(\Z)$ (cf. (\ref{eqn:intro:class-Gammainfty})), as explained in \cite{Maldacena:1998bw} (see also \S5.1 of \cite{Dijkgraaf2007}). 
The terms $e^{-2\pi i m\frac{a\tau+b}{c\tau+d}}$ in (\ref{eqn:intro:qgrav-RGamma}) are obtained by evaluating $e^{-I_{\rm TMG}}$, with $\mu=\sqrt{-\Lambda}=1/16Gm$, on a solution with boundary curve $\CC/(\Z+\tau\Z)$, and the subtraction of $e^{-2\pi i m\frac{a}{c}}$ represents quantum corrections to the classical action.

In \cite{DunFre_RSMG}, the above conjecture is extended so as to encompass the principal modulus property for all elements of the monster, with a view to establishing a conceptual foundation for monstrous moonshine. More specifically, the first main conjecture of \cite{DunFre_RSMG} states the following.

\begin{conjecture}[Duncan--Frenkel]\label{conj:qgrav-tw3dg}
There exists a monster-indexed family of {\em twisted chiral three-dimensional gravity} theories, whose genus one partition functions at 
\begin{gather}
\mu=\sqrt{-\Lambda}=1/16G
\end{gather}
are given by $T_{\Gamma_g}^{(-1)}(-1/\tau)$, where $T^{(-1)}_{\Gamma_g}(\tau)$ is the normalized Rademacher sum attached to $\Gamma_g$, satisfying (\ref{eqn:intro:radsums-TRc}).
\end{conjecture}

From the point of view of vertex operator algebra theory, $T_{g}(-1/\tau)$---which coincides with $T_{\Gamma_g}^{(-1)}(-1/\tau)$ according to Theorems \ref{thm:intro:class-borcherds} and \ref{thm:intro:qgrav-Rgammainv}---is the graded dimension of the unique simple $g$-twisted $\Vnat$-module $\Vnat_g$ (cf. \S\ref{sec:intro:mod}). This non-trivial fact about the functions $T_g(-1/\tau)$ is proven by Carnahan in Theorem 5.1.4 of \cite{Carnahan:2012gx}.

Geometrically, the twists of the above conjecture are defined by imposing (generalized) spin structure conditions on solutions to the chiral gravity equations, and allowing orbifold solutions of certain kinds. See \S7.1 of \cite{DunFre_RSMG} for a more complete discussion. The corresponding sums over geometries are then indexed by coset spaces $\Gamma_\infty\backslash\Gamma$, for various groups $\Gamma<\SL_2(\RR)$, commensurable with $\SL_2(\Z)$. According to Theorem \ref{thm:intro:qgrav-Rgammainv}, the genus one partition function corresponding to such a twist, expected to be a Rademacher sum on physical grounds, will only satisfy the basic physical consistency condition of $\Gamma$-invariance if $\Gamma$ is a genus zero group. One may speculate that a finer analysis of physical consistency will lead to the list of conditions given in \S6.5 of \cite{DunFre_RSMG}, which characterize the groups $\Gamma_g$ for $g\in \MM$, according to Theorem 6.5.1 of \cite{DunFre_RSMG}. Thus the discrete groups $\Gamma_g$ of monstrous moonshine may ultimately be recovered as those defining physically consistent twists of chiral three-dimensional gravity.

On the other hand, it is reasonable to expect that twisted chiral gravity theories are determined by symmetries of the underlying untwisted theory. Conceptually then, but still conjecturally, the monster group appears as the symmetry group of chiral three-dimensional gravity, for which the corresponding twists exist. The principal modulus property of monstrous moonshine may then be explained: as a consequence of Theorem \ref{thm:intro:qgrav-Rgammainv}, together with the statement that the genus one partition function of a twisted theory is $T^{(-1)}_{\Gamma}(-1/\tau)$, where $T^{(-1)}_{\Gamma}(\tau)$ is the normalized Rademacher sum attached to the subgroup $\Gamma<\SL_2(\RR)$ that parameterizes the geometries of the twist.

Before concluding this section we mention two further variations on three-dimensional gravity which may ultimately prove relevant to moonshine. The first of these is {\em log gravity} which was initiated by work \cite{2008JHEP...07..134G} of Grumiller--Johansson, and discussed also in \cite{Maloney:2009ck}. We refer to \cite{Grumiller:2013at} for a detailed review. The second is {\em flat space chiral gravity} which was introduced in \cite{Bagchi:2012yk}, and is also reviewed in \cite{Grumiller:2013at}.

For more background on the mathematics and physics of black holes we refer the reader to \cite{Grumiller:2014qma}. We refer to \cite{Carlip:2005zn,Carlip:2011vr} for reviews that focus on the particular role of conformal field theory in understanding quantum gravity. 

\section{Moonshine Tower}\label{sec:tower}

An optimistic view on the relationship between moonshine and gravity is adopted in \S7 of \cite{DunFre_RSMG}. In particular, in \S7.2 of \cite{DunFre_RSMG} the consequences of Conjecture \ref{conj:qgrav-tw3dg} for the {\em second quantization} of chiral three-dimensional gravity are explored. (We warn the reader that the notion of second quantized gravity is very speculative at this stage.) 

Motivated in part by the results on second quantized string theory in \cite{Dijkgraaf1997}, the existence of a 
tower of monster modules 
\begin{gather}\label{eqn:tower-Vminusm}
V^{(-m)}=\bigoplus_{n=-m}^\infty V^{(-m)}_n,
\end{gather}
parameterized by positive integer values of $m$,
is predicted in \S7.2 of \cite{DunFre_RSMG}. Moreover, it is suggested that the graded dimension of $V^{(-m)}$ should be given by
\begin{gather}\label{Jmfunctions}
J^{(-m)}: =m\hat{T}(m)J,
\end{gather}
where 
$\hat{T}(m)$ denotes the {\em (order $m$)}
{\em Hecke operator}, acting on $\SL_2(\Z)$-invariant holomorphic functions 
on $\HH$ according to the rule
\begin{gather}
	(\hat{T}(m)f)(\tau):=\frac{1}{m}\sum_{\substack{ad=m\\0\leq b<d}}f\left(\frac{a\tau+b}{d}\right).
\end{gather}

Standard calculations (cf. e.g. Chp.VII, \S5 of \cite{Ser_CrsArth}) determine that $m\hat{T}(m)J$ is  an $\SL_2(\Z)$-invariant holomorphic function on $\HH$, whose Fourier coefficients 
\begin{gather}
J^{(-m)}(\tau)=\sum_nc(-m,n)q^n
\end{gather}
are expressed in terms of those of $J(\tau)=\sum_{n=-1}^{\infty}c(n)q^n$, by $c(-m,n)=\delta_{-m,n}$ for $n\leq 0$, and 
\begin{gather}\label{eqn:intro:tower-cmnsumcn}
	c(-m,n)=\sum_{\substack{k>0\\k|(m,n)}}\frac{m}{k}c(mn/k^2),
\end{gather}
for $n>0$, where $(m,n)$ denotes the greatest common divisor of $m$ and $n$. In particular, $J^{(-m)}(\tau)=q^{-m}+O(q)$ as $\Im(\tau)\to \infty$. There is only one such $\SL_2(\Z)$-invariant holomorphic function on $\HH$, so we have 
\begin{gather}
	J^{(-m)}(\tau)=\sum_{n=-m}^{\infty}\dim(V^{(-m)})q^n=T^{(-m)}_{\Gamma}(\tau)
\end{gather}
according to (\ref{eqn:intro:radsums-TmGammaFouExp}) and Theorem \ref{thm:intro:qgrav-Rgammainv}, when $\Gamma=\SL_2(\Z)$. So the graded dimension of $V^{(-m)}$ is also a normalized Rademacher sum. 

We would like to 
investigate the higher order analogues 
of the McKay--Thompson series $T_g$ (cf. (\ref{eqn:intro-Tg})), encoding the graded traces of monster elements on $V^{(-m)}$, but for this we must first determine the $\MM$-module structure on each homogeneous subspace $V^{(-m)}_n$.  

A solution to this problem is entailed in Borcherds' proof \cite{borcherds_monstrous} of the monstrous moonshine conjectures, and the identity (\ref{eqn:class-catdenidm}), 
in particular. To explain this, recall 
the {\em Adams operation} $\psi^k$ on virtual $G$-modules, defined, for $k\geq 0$ and $G$ a finite group, 
by
requiring that 
\begin{gather}\label{eqn:tower-adamsop}
	\tr(g|\psi^k(V))=\tr(g^k|V)
\end{gather}
for $g\in G$. (Cf. \cite{MR1043170,MR0364425} for more details on Adams operations.) Using the $\psi^k$ we may equip 
$V^{(-m)}$ with a virtual $\MM$-module structure (we will see momentarily that it is actually an $\MM$-module, cf. Proposition \ref{prop:tower-Vmnhonest}) by defining 
$V^{(-m)}_{-m}:=\CC$ to be the one-dimensional trivial $\MM$-module, $V^{(-m)}_{n}:=0$ for $-m<n\leq 0$, and 
\begin{gather}\label{eqn:tower-Vmn}
	V^{(-m)}_n:=\bigoplus_{\substack{k>0\\k|(m,n)}}\CC^{m/k}\otimes\psi^k(\Vnat_{mn/k^2})
\end{gather}
for $n>0$, where 
$\CC^{m/k}$ denotes the trivial $\MM$-module of dimension $m/k$. For convenience later on, we also define $V^{(0)}=V^{(0)}_0:=\CC$ to be the trivial, one-dimensional $\MM$-module, regarded as graded, with grading concentrated in degree $n=0$.

Evidently $\psi^k$ preserves dimension, so 
the graded dimension of $V^{(-m)}$ is still given by $J^{(-m)}$, according to (\ref{eqn:intro:tower-cmnsumcn}).
Define the {\em order $m$ McKay--Thompson series} $T^{(-m)}_g$, for $m\geq 0$ and $g\in \MM$, by setting 
\begin{gather}\label{eqn:tower-Tmg}
	T^{(-m)}_g(\tau):=q^{-m}+\sum_{n>0}\tr(g|V^{(-m)}_n)q^n.
\end{gather}
Then $T^{(0)}_g=1$ for all $g\in \MM$, and $T^{(-1)}_g$ is the original McKay--Thompson series $T_g$. 
More generally, we have the following result, which 
constructs the $T^{(-m)}_g$ uniformly and explicitly as Rademacher sums. 
\begin{theorem}\label{secondthm}
For $m> 0$ and $g\in \MM$ we have $T^{(-m)}_g(\tau)=T^{(-m)}_{\Gamma_g}(\tau)$, where $\Gamma_g$ is the invariance group of $T_g(\tau)$, and $T^{(-m)}_\Gamma$ denotes the normalized Rademacher sum of order $m$ attached to $\Gamma$, as in (\ref{eqn:intro:radsums-TRc}). In particular, $T^{(-m)}_g(\tau)$ is a monic integral polynomial of degree $m$ in $T_g(\tau)$.
\end{theorem}
\begin{proof} 
We will use Borcherds' identity (\ref{eqn:class-catdenidm}). To begin, note that $T^{(-m)}_g$ is given explicitly in terms of traces on $\Vnat$ by
\begin{gather}\label{eqn:tower-Tmgexp}
	T^{(-m)}_g(\tau)=q^{-m}+\sum_{n>0}\sum_{k|(m,n)}\frac mk\tr(g|\psi^k(\Vnat_{mn/k^2}))q^n
\end{gather}
according to (\ref{eqn:tower-Vmn}) and (\ref{eqn:tower-Tmg}). Recall that $R(G)$ denotes the integral representation ring of a finite group $G$. Extend the $\psi^k$ from $R(G)$ to $R(G)[[p,q]][q^{-1}]$, by setting 
$\psi^k(Mp^mq^n)=\psi^k(M)p^{km}q^{km}$ for $M\in R(G)$. Then it is a general property of the Adams operations (cf. \S5.2 of \cite{MR1327230}) that 
\begin{gather}
\log\bigwedge\nolimits_{-1}(X)=-\sum_{k>0}\frac 1k\psi^k(X)
\end{gather}
in $R(G)[[p,q]][q^{-1}]\otimes_{\Z}\QQ$, for $X\in R(G)[[p,q]][q^{-1}]$. 
So taking $X=\sum_{\substack{m,n\in\Z\\m>0}} \Vnat_{mn}p^mq^n$ we obtain 
\begin{gather}
\begin{split}\label{eqn:tower-logLambda}
\log\bigwedge\nolimits_{-1}\left(\sum_{\substack{m,n\in\Z\\m>0}}\Vnat_{mn}p^mq^n\right)&=
-\sum_{k>0}\sum_{\substack{m,n\in\Z\\m>0}}\frac 1k\psi^k(\Vnat_{mn})p^{km}q^{kn}\\
&=
-\sum_{\substack{m,n\in\Z\\m>0}}\sum_{k|(m,n)}\frac 1k\psi^k(\Vnat_{mn/k^2})p^mq^n
\end{split}
\end{gather}
for the logarithm of the left hand side of (\ref{eqn:class-catdenidm}).
If we now define $V^{(-m)}(q):=\sum_n V^{(-m)}_nq^n$, an element of $R(\MM)[[q]][q^{-1}]$, then the generating series $\sum_{m>0}p^mV^{(-m)}(q)$ 
is obtained when we apply $-p\partial_p$ to (\ref{eqn:tower-logLambda}), according to the definition (\ref{eqn:tower-Vmn}) of the $V^{(-m)}_n$ as elements of $R(\MM)$. 
So apply $-p\partial_p\log(\,\cdot\,)$ to both sides of (\ref{eqn:class-catdenidm}) to obtain the identity
\begin{gather}\label{eqn:tower-Vmqpm}
	\sum_{m>0}V^{(-m)}(q)p^m=-1-
	({p\partial_p\Vnat(p)})
	\sum_{k\geq 0}
	{\Vnat(q)^k}{\Vnat(p)^{-k-1}}
\end{gather}
in $R(\MM)[[p,q]][q^{-1}]$, where $\Vnat(q)=V^{(-1)}(q)=q^{-1}+\sum_{n>0}\Vnat_nq^n$. The right hand side of (\ref{eqn:tower-Vmqpm}) really is a taylor series in $p$, for we use $\Vnat(p)^{-1}$ as a short hand for $\sum_{k\geq0}(-1)^kp^{k+1}\Vnat_+(p)$, where $\Vnat_+(p):=\sum_{n>0}\Vnat_np^n$ is the regular part of $\Vnat(p)$.

The McKay--Thompson series $T^{(-m)}_g(\tau)$ is just the trace of $g$ on $V^{(-m)}(q)$, so an application of $\tr(g|\,\cdot\,)$ to (\ref{eqn:tower-Vmqpm}) replaces $V^{(-m)}(q)$ with $T^{(-m)}_g(\tau)$, and $\Vnat(q)$ with $T_g(\tau)$, 
etc. and shows that $T^{(-m)}_g$ is indeed a polynomial in $T_g$, of degree $m$ since the leading term of $T^{(-m)}_g$ is $q^{-m}$ by definition. In particular, $T^{(-m)}_g$ is a modular function for $\Gamma_g$, with no poles away from the infinite cusp. Since $\Gamma_g$ has genus zero, such a function is uniquely determined (up to an additive constant) by the polar terms in its Fourier expansion. The McKay--Thompson series $T^{(-m)}_g$ and the Rademacher sum $T_{\Gamma_g}^{(-m)}$ both satisfy $q^{-m}+O(q)$ as $\Im(\tau)\to \infty$ (cf. (\ref{eqn:intro:radsums-TmGammaFouExp})), and neither have poles away from the infinite cusp, so they must coincide. This completes the proof.
\end{proof}

\begin{remark}
The identity obtained by taking the trace of $g\in \MM$ on (\ref{eqn:tower-Vmqpm}) may be compactly rewritten
\begin{gather}\label{eqn:qg-dTgpTgpTgq}
	\sum_{m\geq 0}T^{(-m)}_g(\tau)p^m
	=\frac{p\partial_p T_g(\sigma)}{T_g(\tau)-T_g(\sigma)},
\end{gather}
where $p=e^{2\pi i \sigma}$ and $T_g(\sigma)=\sum_m\tr(g|\Vnat_m)p^m$. This expression (\ref{eqn:qg-dTgpTgpTgq}) is proven for some special cases by a different method in \cite{2014arXiv1407.4479B}.
\end{remark}

Recall that the monster group has $194$ irreducible ordinary representations, up to equivalence. Let us denote these by $M_i$, for $1\leq i\leq 194$, where the ordering is as in \cite{ATLAS}, so that the character of $M_i$ is the function denoted $\chi_i$ in \cite{ATLAS}. Define $\m_i(-m,n)$ to be the multiplicity of $M_i$ in $V^{(-m)}_n$, so that
\begin{gather}\label{eqn:tower-mi}
	V^{(-m)}_n\simeq \bigoplus_{i=1}^{194} 
	M_i^{\oplus \m_i(-m,n)}
\end{gather}
as $\MM$-modules, and $c(-m,n)=\sum_{i=1}^{194}\m_i(-m,n)\chi_i(e)$. 

A priori, the $\MM$-modules $V^{(-m)}_n$ may be virtual, meaning that some of the integers $\m_i(-m,n)$ are negative. 
\begin{proposition}\label{prop:tower-Vmnhonest}
The $V^{(-m)}_n$ are all (non-virtual) modules for the monster. In particular, the integers $\m_i(-m,n)$ are all non-negative. 
\end{proposition}
\begin{proof}
The claim follows from the modification of Borcherds' proof of Theorem \ref{thm:intro:class-borcherds} presented by Jurisich--Lepowsky--Wilson in \cite{MR1327230}. In \cite{MR1327230} a certain free Lie sub algebra $\gt{u}^-$ of the monster Lie algebra $\gt{m}$ is identified, for which the identity $\Lambda(\gt{u}^-)=H(\gt{u}^-)$ (or rather, the logarithm of this) yields 
\begin{gather}\label{eqn:tower-JLWid}
	\sum_{m,n>0}\sum_{k|(m,n)}\frac{1}{k}\psi^k(\Vnat_{mn/k^2})p^mq^n
	=
	\sum_{k>0}\frac{1}{k}\left(\sum_{m,n>0}\Vnat_{m+n-1}p^mq^n\right)^k
\end{gather} 
in $R(\MM)[[p,q]][q^{-1}]\otimes\QQ$. (Notice the different range of summation, compared to (\ref{eqn:class-catdenidm}).) We apply $p\partial_p$ to (\ref{eqn:tower-JLWid}), and recall the definition (\ref{eqn:tower-Vmn}) of $V^{(-m)}_n$ to obtain
\begin{gather}\label{eqn:tower-pdpJLWid}
	\sum_{m,n>0}V^{(-m)}_np^mq^n
	=
	\sum_{k>0}\left(\sum_{m,n>0}m\Vnat_{m+n-1}p^mq^n\right)
	\left(\sum_{m,n>0}\Vnat_{m+n-1}p^mq^n\right)^{k-1}.
\end{gather} 
The coefficient of $p^mq^n$ in the right hand side of (\ref{eqn:tower-pdpJLWid}) is evidently a non-negative integer combination of the $\MM$-modules $\Vnat_n$, so the proof of the claim is complete.
\end{proof}

In \S\ref{sec:mmdist} we will determine the 
behavior of the multiplicity functions $\m_i(-m,n)$ (cf. (\ref{eqn:tower-mi})) as $n\to \infty$. For applications to gravity a slightly different statistic is more relevant.
Recall from \S\ref{sec:intro:qgrav} that it is the Virasoro highest weight vectors---i.e. those $v\in \Vnat_n$ with $L(k)v=0$ for $k>0$---that represent black hole states in chiral three-dimensional gravity at $m=1$. 
Such vectors generate {\em highest weight modules} for $\Vir$, the structure of which has been determined by Feigin--Fuchs in \cite{MR770243}. (See \cite{MR1463812} for an alternative treatment.) Specializing to the case that the central element ${\bf c}$ (cf. (\ref{eqn:intro:class-vir})) acts as $c=24m$ times the identity, for some positive integer $m$, we obtain from the results of \cite{MR770243} that the isomorphism type of an irreducible highest weight module for $\Vir$ depends only on the $L(0)$-eigenvalue of a generating highest weight vector, $v$, and if 
$L(0)v=hv$ for $h$ a non-negative integer, then 
\begin{gather}\label{eqn:tower-gdimLhc}
	\dimq L(h,c)=
		\begin{cases}
		q^{-m}(1-q)(q)_{\infty}^{-1}&\text{ if $h=0$,}\\
		q^{h-m}(q)_{\infty}^{-1}&\text{ if $h>0$,}
		\end{cases}
\end{gather}
(cf. (\ref{eqn:mod-dimqM})) where $L(h,c)$ denotes the irreducible highest weight $\Vir$-module generated by $v$, 
and 
\begin{gather}\label{eqn:tower-qinfty}
	(q)_{\infty}:=\prod_{n>0}(1-q^n).
\end{gather}
(Note that all the Virasoro modules in this work are graded by $L(0)-{\bf c}/24$. See \cite{MR1650637} for details of the calculation that returns (\ref{eqn:tower-gdimLhc}) in the case that $m=1$.)

\begin{remark}
We may now recognize the leading terms in (\ref{eqn:qgrav-ext}) as exactly those of the graded dimension of the Virasoro module $L(0,24m)$.
\end{remark}

It is known that $\Vnat$ is a direct sum of highest weight modules for the Virasoro algebra (cf. e.g. \cite{MR1650637}). Since the Virasoro and monster actions on $\Vnat$ commute, we have an isomorphism
\begin{gather}\label{eqn:tower-VnatVirMmod}
	\Vnat\simeq L(0,24)\otimes W^\natural_{-1}\oplus \bigoplus_{n>0}L(n+1,24)\otimes W^\natural_n
\end{gather}
of modules for $\Vir\times \MM$, where $W^\natural_n$ denotes the subspace of $\Vnat_n$ spanned by Virasoro highest weight vectors. To investigate how the black hole states in $\Vnat$ are organized by the representation theory of the monster, we define non-negative integers $\n_i(n)$ by requiring that
\begin{gather}\label{eqn:tower-Wnat}
	W^{\natural}_n\simeq \bigoplus_{i=1}^{194} 
	M_i^{\oplus \n_i(n)},
\end{gather}
for $n\geq -1$. 

Evidently $\n_i(n)\leq \m_i(-1,n)$ for all 
$i$ and $n$ since $W^\natural_n$ is a subspace of $\Vnat_n$. 
To determine the precise relationship between the $\n_i(n)$ and $\m_i(-1,n)$, define $U_g(\tau)$ for $g\in \MM$ by setting
\begin{gather}\label{eqn:tower-Ug}
	U_g(\tau):=\sum_{n=-1}^\infty\tr(g|W^{\natural}_n)q^n,
\end{gather}
so that $U_g(\tau)=q^{-1}+\sum_{n>0}\sum_{i=1}^{194}\n_i(n)\chi_i(g)q^n$ (cf. (\ref{eqn:tower-Wnat})). 
Combining (\ref{eqn:tower-gdimLhc}), (\ref{eqn:tower-VnatVirMmod}) and (\ref{eqn:tower-Wnat}), together with the definitions (\ref{eqn:intro-Tg}) of $T_g$ and (\ref{eqn:tower-Ug}) of $U_g$, we obtain
\begin{gather}
	T_g(\tau)=q^{-1}\frac{(1-q)}{(q)_{\infty}}+\sum_{n>0}q^n\frac{1}{(q)_\infty}\sum_{i=1}^{194}\n_i(n)\chi_i(g),
\end{gather}
or equivalently,
\begin{gather}\label{eqn:tower-UgTg}
	U_g(\tau)={(q)_{\infty}}T_g(\tau)+1
\end{gather}
for all $g\in \MM$. (This computation also appears in \cite{MR1650637}.)

In \S\ref{sec:mmdist} we will use (\ref{eqn:tower-UgTg}) to determine the asymptotic behavior of the $\n_i(n)$ (cf. Theorem \ref{distribution}), and thus the statistics of black hole states, at $\ell=16G$, in the conjectural chiral three-dimensional gravity theory dual to $\Vnat$.

\begin{remark}
Note that we may easily construct modules for $\Vir\times\MM$ satisfying the extremal condition (\ref{eqn:qgrav-ext}), for each positive integer $m$, by considering direct sums of the monster modules $V^{(-m)}$ constructed in Proposition \ref{prop:tower-Vmnhonest}. A very slight generalization of the argument just given will then yield formulas for the graded traces of monster elements on the corresponding Virasoro highest weight spaces. Since it has been shown \cite{Gaiotto:2008jt,Hoe_SDVOSALgeMinWt} that such modules cannot admit vertex operator algebra structure, we do not pursue this here.
\end{remark}

\section{Monstrous Moonshine's Distributions}\label{sec:mmdist}

We now address the problem of determining exact formulas and asymptotic distributions of irreducible components. 
This work will rely heavily on the modularity of the underlying McKay--Thompson series (i.e. Theorems~\ref{firstthm} and \ref{secondthm}).

We prove formulas for the multiplicities $\m_i(-m,n)$ and $\n_i(n)$ which in turn imply the following asymptotics.

\begin{theorem}\label{distribution} If $m$ is a positive integer and $1\leq i\leq 194$, then as $n\rightarrow +\infty$ we have 
\begin{displaymath}
\begin{split}
\m_i(-m,n)&\sim \frac{\dim(\chi_i)|m|^{1/4}
}{\sqrt{2} |n|^{3/4}|\MM|} \cdot e^{4 \pi \sqrt{|mn|}}\\
\n_i(n)&\sim 
\frac{\sqrt{12}~\dim(\chi_i)}{|24n+1|^{1/2}|\MM|} \cdot e^{\frac{\pi}{6} \sqrt{23|24n+1|}}
\end{split}
\end{displaymath}
\end{theorem}

These asymptotics immediately imply that the following limits are well-defined
\begin{gather}
\begin{split}
\delta\left(\m_i(-m)\right)&:=\lim_{n\rightarrow +\infty}\frac{\m_i(-m,n)}{\sum_{i=1}^{194}\m_i(-m,n)}
\\\delta\left(\n_i\right)&:=\lim_{n\rightarrow +\infty}\frac{\n_i(n)}{\sum_{i=1}^{194}\n_i(n)}.
\end{split}
\end{gather}

\begin{corollary}\label{cor:dist}
In particular, we have that
\begin{gather*}\label{eqn:intro-deltas}
	\delta\left(\m_i{(-m)}\right)=
	\delta\left(\n_i\right)=
	\frac{\dim(\chi_i)}{\sum_{j=1}^{194}\dim(\chi_j)}=
	\frac{\dim(\chi_i)}{5844076785304502808013602136}.
\end{gather*}
\end{corollary}
We illustrate these asymptotics explicitly, for $\chi_1$, $\chi_2$, and $\chi_{194}$, in the case that $m=1$, in Table 1, where we 
let $\delta\left(\m_i(-m,n)\right)$ denote the
proportion of components corresponding to $\chi_i$ in $V^{(m)}_n$.

\begin{table}[h]\label{tab:dist-exp}
\caption{}
\begin{tabular}{|c|c|c|c|c|}\hline
$n$&$\delta\left(\m_1(-1,n)\right)$&$\delta\left(\m_2(-1,n)\right)$&$\cdots$&$\delta\left(\m_{194}(-1,n)\right)$
\\\hline 
-1&1&0&$\cdots$&0\\
0&0&0&$\cdots$&0\\
1& 1/2& 1/2&$\cdots$&0\\
2& 1/3& 1/3 &$\cdots$ &0\\
$\vdots$&$\vdots$&$\vdots$&$\vdots$ &$\vdots$\\
40& $4.011\ldots \times 10^{-4}$     &  $2.514\ldots \times 10^{-3}$   &  $\cdots$  &  0.00891\ldots\\
60& $2.699\ldots \times 10^{-9}$     &  $2.732\ldots \times 10^{-8}$   &  $\cdots$  &  0.04419\ldots\\
80  &$4.809\ldots \times 10^{-14}$  & $7.537\ldots \times 10^{-13}$ & $\cdots$&  0.04428\ldots \\
100& $4.427\ldots \times 10^{-18}$ & $1.077\ldots \times 10^{-16}$ & $\cdots$& 0.04428\ldots\\
120&  $1.377\ldots \times 10^{-21}$ & $5.501\ldots \times 10^{-20}$ & $\cdots$& 0.04428\ldots \\
140&$1.156\ldots \times 10^{-24}$ & $1.260\ldots \times 10^{-22}$ & $\cdots$& 0.04428\ldots  \\
160& $2.621\ldots \times 10^{-27}$ & $3.443\ldots \times 10^{-23}$ & $\cdots$& 0.04428\ldots \\
180& $1.877\ldots \times 10^{-28}$ & $3.371\ldots \times 10^{-23}$ & $\cdots$& 0.04428\ldots \\
200& $1.715\ldots \times 10^{-28}$ & $3.369\ldots \times 10^{-23}$ & $\cdots$& 0.04428\ldots  \\
220& $1.711\ldots \times 10^{-28}$ & $3.368\ldots \times 10^{-23}$ & $\cdots$& 0.04428\ldots \\
240& $1.711\ldots \times 10^{-28}$ & $3.368\ldots \times 10^{-23}$ & $\cdots$& 0.04428\ldots  \\
$\vdots$&$\vdots$&$\vdots$&&$\vdots$\\
$\infty$&$\frac{1}{5844076785304502808013602136}$&$\frac{196883}{5844076785304502808013602136}$&$\cdots$&
$\frac{258823477531055064045234375}{5844076785304502808013602136}$\\
\ & \ & \ & & \ 
\\\hline
\end{tabular}
\end{table}
The precise values given in the bottom row of Table~1 admit the following decimal approximations:
\begin{gather}
	\begin{split}
\delta\left(\m_1{(-m)}\right)&=
\frac{1}{5844076785304502808013602136}\approx 1.711\ldots\times 10^{-28}\\ \ \\
\delta\left(\m_2{(-m)}\right)&=
\frac{196883}{5844076785304502808013602136}\approx 3.368\ldots\times 10^{-23}\\ \ \ \\
\delta\left(\m_{194}{(-m)}\right)&=
\frac{258823477531055064045234375}{5844076785304502808013602136}\approx 4.428\ldots\times 10^{-2}
	\end{split}
\end{gather}

Before explaining the proof of Theorem \ref{distribution} we describe an application to the computation\footnote{We are grateful to the referee for suggesting this.} of quantum dimensions. Suppose that $V$ is a vertex operator algebra and $M$ is a $V$-module such that the graded-dimension functions $Z_V$ and $Z_M$ are defined (cf. (\ref{eqn:mod-ZM})). Then the {\em quantum dimension} of $M$ relative to $V$ is defined by setting
\begin{gather}
	\qdim_VM:=\lim_{y\to 0}\frac{Z_M(iy)}{Z_V(iy)}
\end{gather}
(for $y$ real and positive), assuming the limit exists. (Cf. \S3.1 of \cite{MR3105758}.) 

Dong--Mason initiated a vertex algebraic {\em quantum Galois theory} in \cite{MR1430435}. (Cf. also \cite{MR1420556,MR1684904,MR1682988}.) In this theory, inclusions of vertex operator algebras take on the role played by inclusions of fields in the classical setting. It is established in \cite{MR3105758} that the quantum dimension $\qdim_UV$, for $U$ a sub vertex operator algebra of $V$, serves as the quantum analogue of the relative dimension $\dim_EF$ of a field $F$ over a subfield $E$. 

Write $V^\MM$ for the fixed points of the action of $\MM$ on $\Vnat$. Then $V^\MM$ is a sub vertex operator algebra of $\Vnat$, called the {\em monster orbifold}. Using Theorem \ref{distribution}, together with a result from \cite{MR3105758}, we will compute the quantum dimensions of the monster orbifold. To formulate this precisely, note that, according to the main theorem of \cite{MR1420556}, we have
\begin{gather}
	\Vnat\simeq\bigoplus_{i=1}^{194}V^\MM_i\otimes M_i,
\end{gather}
as $V^\MM\times\MM$-modules, 
for some $V^\MM$-submodules $V^\MM_i$ in $\Vnat$, where $M_i$ denotes an irreducible module for $\MM$ with character $\chi_i$, as in (\ref{eqn:tower-mi}).
\begin{corollary}\label{cor:qdims}
We have $\qdim_{V^\MM}V^\MM_i=\chi_i(e)$. In particular, the quantum dimension of $V^\MM_i$ relative to $V^\MM$ exists, for all $1\leq i\leq 194$.
\end{corollary}
\begin{proof}
Recall (cf. (\ref{eqn:mod-ZM})) that $Z_{V^{\MM}_i}(\tau)$ is the function obtained by substituting $q=e^{2\pi i \tau}$ in $\dimq V^\MM_i=\sum_n (V^\MM_i)_nq^n$, assuming this series converges. Note that $V^\MM=V^\MM_1$. According to Proposition 3.6 of \cite{MR3105758}, if the limit
\begin{gather}
d_i:=\lim_{n\to \infty} \frac{\dim (V^\MM_i)_n}{\dim(V^{\MM}_1)_n}
\end{gather}
exists, then the quantum dimension of $V^\MM_i$ relative to $V^\MM_1$ also exists, and equals $d_i$.
Comparing with (\ref{eqn:tower-mi}) we see that $\dim(V^\MM_i)_n=\m_i(-1,n)$.  Applying Theorem \ref{distribution} we obtain $d_i=\chi_i(e)$, and this completes the proof.
\end{proof}
Note that Corollary \ref{cor:qdims} confirms a special case of Conjecture 6.7 of \cite{MR3105758}. It would be interesting to see how generally this method can be applied, to the computation of quantum dimensions of orbifolds $V^G$, where $V$ is a vertex operator algebra and $G$ is a group of automorphisms of $\Aut(G)$. (See also Problem \ref{prob:qdim}.)

\subsection{The modular groups in monstrous moonshine}\label{Gammag}
To obtain exact formulas, we begin by recalling the modular groups which arise in monstrous moonshine.
Suppose $\Gamma_*<\GL_2(\RR)$ is a discrete group which is commensurable with $\SL_2(\Z).$ If $\Gamma_*$ defines a genus zero quotient of $\HH$, then the field of modular functions which are invariant under $\Gamma_*$ is generated by a single element, the  principal modulus (cf. (\ref{eqn:intro-Tgatinfty})). Theorem~\ref{thm:intro:class-borcherds} implies that the  $T_g$ (defined by (\ref{eqn:intro-Tg})) are 
principal moduli for certain groups $\Gamma_g$. We can describe these groups in terms of groups $E_g$ which in turn may be described in terms of the congruence subgroups 
\begin{equation}\label{Gamma0N}
\Gamma_0(N):=\left\{\begin{pmatrix}a&b\\c&d\end{pmatrix}\in \SL_2(\Z) \ : \ c\equiv 0 \pmod{N}\right\},
\end{equation}
and the Atkin--Lehner involutions $W_e$ for $\Gamma_0(N)$ given by
\begin{equation}
 W_e:=\begin{pmatrix} ae&b\\cN&de\end{pmatrix},
\end{equation}
where $e$ is an exact divisor of $N$ (i.e. $e|N$, and $(e,N/e)=1$), and  $a,b,c,$ and $d$ are integers chosen so that $W_e$ has determinant $e$.

Following Conway--Norton \cite{MR554399} and Conway--McKay--Sebbar \cite{ConMcKSebDiscGpsM}, we denote the groups $E_g$ by symbols of the form $\Gamma_0(N|h)+e,f,\dots$ (or simply $N|h+e,f,\dots$), where $h$ divides $(N,24)$, and each of $e,f,$ etc. exactly divide $N/h$. This symbol represents the group 
$$\Gamma_0(N|h)+e,f,\dots:=\begin{pmatrix} 1/h&0\\0&1\end{pmatrix}\langle\Gamma_0(N/h),W_e,W_f,\dots \rangle\begin{pmatrix} h&0\\0&1\end{pmatrix},$$
where $W_e, W_f,$ etc. are representative of Atkin--Lehner involutions on $\Gamma_0(N/h).$ We use the notation $\mathcal W_g:=\{1,e,f,\dots\}$ to denote this list of Atkin--Lehner involutions contained in $E_g.$ We also note that $\Gamma_0(N|h)+e,f,\dots$ contains $\Gamma_0(Nh).$

The groups $E_g$ are eigengroups for the $T_g$, so that if $\gamma\in E_g$, then $T_g(\gamma \tau)=\sigma_g(\gamma)T_g$, where $\sigma_g$ is a multiplicative group homomorphism from $E_g$ to the group of $h$-th roots of unity. Conway and Norton \cite{MR554399} give the following values for $\sigma_g$ evaluated on generators of $N|h+e,f,\dots$.

\begin{lemma}[Conway--Norton]\label{GammaEigenvalues}
Assuming the notation above, the following are true:
\begin{displaymath}
\begin{array}{ll}
\text{(a)}& \sigma_g(\gamma)=1$  if $\gamma\in \Gamma_0(Nh)\\
\text{(b)} &\sigma_g(\gamma)=1$ if $\gamma$ is an Atkin--Lehner involution of $\Gamma_0(Nh)$ inside $E_g\\
\text{(c)}& \sigma_g(\gamma)=e^{\frac{-2 \pi i}{h} }$ if $\gamma=\begin{pmatrix}1&1/h\\0&1\end{pmatrix}\\
\text{(d)} &\sigma_g(\gamma)=e^{-\lambda_g \frac{2 \pi i}{h}}$ if $\gamma=\begin{pmatrix}1&0\\N&1\end{pmatrix},
\end{array}
\end{displaymath}
where $\lambda_g$ in (d) is $-1$ if $ N/h\in \mathcal W_g,$ and $+1$ otherwise. 
\end{lemma}

This information is sufficient to properly describe the modularity of the series $T_g^{(-m)}(\tau)$ on $E_g$. In section \ref{SectionExactTg}, we give an explicit procedure for evaluating $\sigma_g$. The invariance group $\Gamma_g,$ denoted by $\Gamma_0(N||h)+e,f,\dots$ (or by the symbol $N||h+e,f,\dots$), is defined as the kernel of $\sigma_g.$ A complete list of the groups $\Gamma_g$ can be found in the Appendix (\S \ref{Appendix}) of this paper, or in table 2 of \cite{MR554399}.

Theorems~\ref{firstthm} and \ref{secondthm} are summarized by the following uniform statement.
\begin{theorem}\label{usefultheorem}
Let $g\in \MM$ and $m\geq 1$. Then $T^{(-m)}_g$ is the unique weakly holomorphic modular form of weight zero for $\Gamma_g$ that satisfies $T^{(-m)}_g=q^{-m}+O(q)$ as $\tau$ approaches the infinite cusp, and has no poles at any cusps inequivalent to the infinite one. 
\end{theorem}

\begin{remark}
A {\it weakly holomorphic} modular form is a meromorphic modular form whose poles (if any) are supported at cusps.
\end{remark}

\subsection{Harmonic Maass forms}

Maass--Poincar\'e series allow us to obtain formulas for weakly holomorphic modular forms and mock modular forms.
We begin by briefly recalling the definition of a \emph{harmonic Maass form} of weight $k\in \frac{1}{2}\Z$ and multiplier $\nu$ (a generalization of the notion of a Nebentypus). If $\tau=x+iy$ with $x$ and $y$ real, we define the weight $k$ hyperbolic Laplacian by
\begin{equation}
\Delta_k := -y^2\left( \frac{\partial^2}{\partial x^2} +
\frac{\partial^2}{\partial y^2}\right) + iky\left(
\frac{\partial}{\partial x}+i \frac{\partial}{\partial y}\right),
\end{equation}
and if $\gamma=\begin{pmatrix}a&b\\ c&d\end{pmatrix}\in \SL_2(\Z),$ define 
\[(\gamma:\tau):=(c\tau+d).\]
Suppose $\Gamma$ is a subgroup of finite index in $\SL_2(\Z)$ and $\frac{3}{2}\leq k\in \frac{1}{2}\Z$. Then a real analytic function $F(\tau)$ is a {\it harmonic Maass form}  of weight $k$ on $\Gamma$ with multiplier $\nu$ if: 

\begin{enumerate}
\item[(a)]  The function $F(\tau)$ satisfies the modular transformation with respect to the weight $k$ slash operation,
$$F(\tau)|_k \gamma :=(\gamma:\tau)^{-k}F(\gamma\tau)=\nu(\gamma)F(\tau)$$ for every matrix $\gamma\in \Gamma,$ where if $k\in \Z+\frac12,$ the square root is taken to be the principal branch. In particular, if $\nu$ is trivial, then $F$ is invariant under the action of the slash operator.
\item[(b)] We have that $\Delta_kF(\tau)=0,$
\item[(c)] The singularities of $F$ (if any) are supported at the cusps of $\Gamma$, and for each cusp $\rho$ there is a polynomial $P_{F,\rho}(q^{-1})\in \C[q^{-1/t_\rho}]$ and a constant $c>0$ such that $F(\tau)-P_{F,\rho}(e^{-2\pi i \tau})=O(e^{-cy})$ as $\tau\to \rho$ from inside a fundamental domain. Here $t_\rho$ is the width of the cusp $\rho$.
\end{enumerate}

\begin{remark}
The polynomial $P_{F,\rho}$ above is referred to as the \emph{principal part of $F$ at $\rho$}. In certain applications, condition (c) of the definition may be relaxed to admit larger classes of harmonic Maass forms. However, for our purposes we will only be interested in those satisfying the given definition, having a holomorphic principal part.
\end{remark}

We denote the complex vector space of such functions by $H_k(\Gamma,\nu),$ and note that in order for $H_k(\Gamma,\nu)$ to be nonzero, $\nu$ must satisfy 
$(\gamma:\delta\tau)^k(\delta:\tau)^k\nu(\gamma)\nu(\delta)=(\gamma\delta:\tau)^k\nu(\gamma\delta)$ for every $\gamma,\delta \in \Gamma.$ 

Let $\mathcal S(\Gamma)$ denote some fixed complete set of inequivalent representatives of the cusps of $\Gamma.$ For each representative $\rho=\frac{\alpha}{\gamma}$ with $(\alpha,\gamma)=1,$ fix a matrix $L_\rho=\begin{pmatrix}-\delta&\beta\\ \gamma&-\alpha\end{pmatrix}\in \SL_2(\Z)$ so that $\rho=L_\rho^{-1}\infty.$
 Following Rankin \cite{MR0498390}, let $t_\rho$ be the cusp width and let $\kappa_\rho$ be the cusp parameter, defined as the least nonnegative integer so that $\nu(L_\rho T^{t_\rho}L_{\rho}^{-1})=e^{2\pi i \kappa_\rho},$ where $T:=\begin{pmatrix}1&1\\0&1\end{pmatrix}.$
The stabilizer of $\rho$ in $\Gamma$ is given by ${\Gamma}_{\rho}:=\left\langle \pm T^{t_\rho}\right\rangle,$ so for example $\Gamma_\infty=\langle \pm T\rangle.$
Given $F(\tau)\in H_{2-k}(\Gamma,\nu),$ we refer to $F_{\rho}(\t):=F(\tau)|_{2-k}L_\rho$ as the expansion of $F$ at the cusp $\rho.$
We note that this expansion depends on the choice of $L_\rho$. These facts imply that the expansion of $F_\rho$ can be given as a Fourier series of the form 
$$F_\rho(\tau)=\sum_{n}a(n,y)e^{2\pi i x(n+\kappa_\rho)/t_\rho}.$$ 
More precisely, we have the following. The  Fourier expansion of harmonic Maass forms $F$ at a cusp $\rho$ (see Proposition 3.2 of \cite{BruFun_TwoGmtThtLfts}) splits into two components. 
As before, we let $q:=e^{2\pi i \tau}$.
\begin{lemma}\label{HMFparts}
If $F(\tau)$ is a harmonic Maass form of weight $2-k$ for $\Gamma$ where $\frac{3}{2}\leq k\in \frac12\Z$, and if $\rho$ is a cusp of $\Gamma$, then

\begin{displaymath}
F_\rho(\tau)=F_\rho^+(\tau)+F_\rho^{-}(\tau)
\end{displaymath}
where $F_\rho^+$ is the holomorphic part of $F_\rho$, given by
$$
F_\rho^+(\tau):=\sum_{n\gg -\infty} c_{F,\rho}^+(n) q^{(n+\kappa_\rho)/\t_\rho},
$$
and $F_\rho^{-}$ is the nonholomorphic part, given by
$$
F_\rho^{-}(\tau) + \sum_{n<0} c_{F\rho}^-(n)
\Gamma(k-1,4\pi y|(n+\kappa_\rho)/t_\rho|) q^{(n+\kappa_\rho)/t_\rho}.
$$
\end{lemma}

By inspection, we see that weakly holomorphic modular forms are themselves harmonic Maass forms. In fact, under the given definition, all harmonic Maass forms of positive weight are weakly holomorphic. On the other hand, if the
weight is non-positive, then the space of harmonic Maass forms may be larger than the space of weakly holomorphic modular forms. However, as with the weakly holomorphic modular forms, a harmonic Maass form is uniquely defined by its principal parts at all cusps. This is clear if the form is weakly holomorphic. If the nonholomorphic part is non-zero, we have the following lemma which follows directly from the work of Bruinier and Funke \cite{BruFun_TwoGmtThtLfts}.
\begin{lemma}\label{ppexists}
If $F\in H_{2-k}(\Gamma_0(N))$ has the property that
$F^{-}\neq 0$, then the principal part of $F$
is nonconstant for at least one cusp.
\end{lemma}
\begin{proof}[Sketch of the Proof]
Bruinier and Funke define a pairing $\{\bullet,\bullet\}$ on harmonic weak Maass Forms. The particular quantity $\{2iy^k\overline{\frac{\partial}{\partial \overline \tau}F},F\}$ can be described in terms of either a Petersson norm,
 or in terms of products of coefficients of the principal part of $F$ with coefficients of the nonholomorphic part. Since the Petersson norm is non-zero, at least one coefficient of the principal part is also non-zero.
\end{proof}

Hence, if $F$ and $G$ are harmonic Maass forms of non-positive weight whose principal parts are equal at all cusps, then $F-G$ is holomorphic and vanishes at cusps, and therefore identically $0$.

\subsection{Maass--Poincar\'e series}
The Maass--Poincar\'e series define a basis for a space of harmonic Maass forms and provide exact formulas for their coefficients. The following construction of the Maass--Poincar\'e series follows the method and notation of Bringmann and the third author \cite{BriOno_CoeffHmcMaaFrms} which builds on the early work of Rademacher, followed by more contemporary work of Fay, Niebur, among many others  \cite{MR0506038,Nie_ClassNonAnalyticAutFuns,Nie_ConstAutInts}. The Poincar\'e series we construct in this section are modular for congruence subgroups $\Gamma_0(N)$ which we will then use to construct the McKay--Thompson series. Although we could follow similar methods to construct Poincar\'e series for the groups $\Gamma_g$ directly, we restrict our attention to these groups since the congruence subgroups $\Gamma_0(N)$ are more standard.

For $s \in \C,$  $w \in \Real \setminus \{0\}$, and $k\geq 3/2$, $k\in \frac12 \Z$, let
\begin{equation}
\mathcal{M}_s(w)
:=|y|^{ \frac{k}{2}-1} M_{\text{sign}(w) (1-k/2),s-\frac12} (|w|),
\end{equation}
where $M_{\nu,\mu}(z)$ is the 
 $M$-Whittaker function 
which is a solution to the differential equation
\begin{displaymath}
 \frac{\partial^2 u}{\partial z^2}+
  \left(-\frac{1}{4}+\frac{\nu}{z}+\frac{\frac{1}{4}-\mu^2}{z^2}\right)u=0,
\end{displaymath}
and (here and throughout this paper) $\tau=x+iy.$
Using this function, let
\begin{equation}
\phi_s(\tau):= \mathcal{M}_s(4 \pi y) e^{2 \pi i x}.
\end{equation}
 
Given a positive integer $m$ and a cusp $\rho$, Maass--Poincar\'e series provide a form with principal part equal to $q^{(-m+\kappa_\rho)/t_\rho}$ plus a constant at the cusp $\rho$, and constant at all other cusps, thereby forming a basis for $H_{2-k}(\Gamma,\nu).$

Suppose $m>0$ and $L\in \SL_2(\Z)$ with $\rho=L^{-1}\infty$. Then we have the Maass--Poincar\'e series

\begin{equation}  \label{Poincare}
\mathcal{P}_{L}(\tau,m,\Gamma,2-k,s,\nu) 
:=
\sum_{ M \in {\Gamma}_{\rho} \backslash \Gamma} \frac{\phi_s \left( \frac{-m+\kappa_\rho}{t_\rho} \cdot L^{-1}M\tau\right)}{(L^{-1}:M\tau)^{2-k}(M:\tau)^{2-k}\nu(M)}. 
\end{equation} 
It is easy to check that $\phi_s(\tau)$ is an eigenfunction of $\Delta_{2-k}$ with eigenvalue
$$
s(1-s) + \frac{k^2-2k}{4} .
$$
The right hand side of (\ref{Poincare}) converges absolutely for $\Re(s)>1,$ however Bringmann and the third author establish conditional convergence when $s\geq 3/4$  \cite{BriOno_CoeffHmcMaaFrms}, giving Theorem \ref{PoincareTheorem} below. The theorem is stated for the specific case $\Gamma=\Gamma_0(N)$ for some $N$ and $k\geq 3/2$, in which case we modify the notation slightly and define
\begin{equation}  \label{PoincareModified}
\mathcal{P}_L(\tau,m,N,2-k,\nu) 
:= \frac{1}{\Gamma(k)}\mathcal{P}_L(\tau,m,\Gamma_0(N),2-k,\frac k2,\nu)
\end{equation}
In the statement of the theorem below, $K_c$ is a modified Kloosterman sum given by
\begin{equation}\label{Kloostermann}
K_c(2-k,L,\nu,m,n):=
\sum_
{ \substack{ \substack{ \substack{  0 \leq a <ct_\rho } 
\\ 
a \equiv - \frac{c\cdot(\alpha,N)}{\alpha \gamma} \pmod{ \frac{N}{(\gamma,N)}  }}  \\  ad\equiv1 \pmod c}
} \frac{(S:\tau)^{2-k}
 \exp \left( 2 \pi i 
\left( \frac{a\cdot \frac{(-m+\kappa_\rho)}{t_{\rho}}  + d\cdot \frac{(n+\kappa_\infty)}{t_\infty} }{c}\right)
\right)}{(L^{-1}:LS\tau)^{2-k}(LS:\tau)^{2-k}\nu(LS)},
\end{equation}
where $S=\begin{pmatrix}a&b\\c&d\end{pmatrix}\in \SL_2(\Z).$ If $\nu$ is trivial, we omit it from the notation. We also have that $\delta_{L,S}(m)$ is an indicator function for the cusps $\rho=L^{-1}\infty$ and $\mu=S^{-1}\infty$ 
given by 
$$\delta_{L,S}(m):=\begin{cases}\nu(M)^{-1}e^{2\pi i r\frac{ -m+\kappa_\rho}{t_\rho}}& \text{if } M=L T^rS^{-1} \in \Gamma_0(N),\\
0& \text{if }\mu \not \sim \rho
\text{ in }\Gamma_0(N). \end{cases}
$$
Using this notation, we have the following theorem which gives exact formulas for the coefficients and principal part of $\mathcal{P}_L(\tau,m,N,2-k,\nu),$ which is a generalization of Theorem 3.2 of \cite{BriOno_CoeffHmcMaaFrms}.

\begin{theorem}  \label{PoincareTheorem} Suppose that $\frac{3}{2}\leq k\in \frac{1}{2}\Z$, and
suppose $\rho=L^{-1}\infty$ is a cusp of $\Gamma_0(N).$ If $m$ is a positive integer, then  $\mathcal{P}_{L}(\tau,m,N,2-k,\nu)$
is in $H_{2-k}(\Gamma_0(N),\nu)$. Moreover, the following are true:
\begin{enumerate}
\item
We have 
 $$
 \mathcal{P}_{L}^+(\tau,m,N,2-k,\nu)
 =\delta_{\rho,I}(m)\cdot 
 q^{-m+\kappa_\infty} + \sum_{n \geq 0} a^+(n) q^n.
 $$
 Moreover, if $n>0$, then $a^+(n)$ is given by
 \begin{equation*}
  -i^k 2 \pi \left| \frac{-m +\kappa_\rho}{t_\rho (n+\kappa_\infty)} \right|^{ \frac{k-1}{2}  }  
\hspace{-10pt}\sum_{\substack{ c>0 \\ (c,N)=(\gamma,N)    }} 
\hspace{-10pt}\frac{K_c(2-k,L,\nu,-m,n)}{c} \cdot 
I_{k-1} 
\left(\frac{4 \pi}{ c } \sqrt{    \frac{\left|-m+\kappa_\rho\right| \left|n+\kappa_\infty\right|}{t_\rho} } \right),
\end{equation*} 
where $I_k$ is the usual $I$-Bessel function.
\item
If $S\in \SL_2(\Z)$, then there is some $c\in \CC$ so that the principal part of $\mathcal{P}_{L}(\tau,m,N,2-k)$ at the cusp $\mu=S^{-1}\infty$ is given by 
 $$ 
\delta_{L,S}(m)
   q^{\frac{-m+\kappa_\rho}{t_{\rho}} } +c
 $$
\end{enumerate}
\end{theorem}
\begin{remark}
We shall use Theorem~\ref{PoincareTheorem} to prove Theorem~\ref{distribution}. We note that one could prove Theorem~\ref{distribution} without referring to the theory
of Maass--Poincar\'e series. One could make use of ordinary weakly holomorphic Poincar\'e series. However, we have chosen to use Maass--Poincar\'e series and the theory of harmonic Maass forms
because these results are more general, and because of the recent appearance of harmonic Maass forms in the theory of {\it umbral moonshine} (see \S\ref{sec:um}).
\end{remark}

\begin{proof}[Sketch of the proof]
Writing $\rho=\frac{\alpha}{\gamma}$, Bringmann and the third author prove this theorem for the case that $\gamma \mid N$ and $(\alpha,N)=1,$ along with the assumption that $\mu$ and $\rho$ are in a fixed complete set of inequivalent cusps, so that $\delta_{\mu,\rho}=1$ or $0$. This general form is useful to us particularly since it works equally well for the cusps $\infty$ with $L$ taken to be the identity, and for $0$ with $L$ taken to be $\begin{pmatrix}0&-1\\1&0\end{pmatrix}$.
 
 Here and for the remainder of the paper we let $\mathcal S_N$ denote any complete set of inequivalent cusps of $\Gamma_0(N),$ and for each $\rho\in \mathcal S_N$, we fix some $L_\rho$ with $\rho=L_\rho^{-1}\infty.$ Rankin notes \cite{MR0498390} (Proof of Theorem 4.1.1(iii)) that given some choice of $\mathcal S_N,$ each right coset of $\Gamma_0(N)\backslash \SL_2(\Z)$ is in $\Gamma_0(N)\cdot L_\rho T^r$ for some unique $\rho\in \mathcal S_N.$ Moreover, the $r$ in the statement is unique modulo $t_\rho,$ so the function $\delta_{L,S}(m)$ given above is well-defined on all matrices in $\SL_q(\Z)$.

In the proof given by Bringmann and the third author, the sum of Kloosterman sums $\displaystyle \sum_{\substack{ c>0 \\ (c,N)=(\gamma,N)    }} 
\frac{K_c(2-k,L,\nu,-m,n)}{c} \dots$ 
is written as a sum over representatives of the double coset $\Gamma_\rho\backslash L^{-1}\Gamma_0(N)\slash\Gamma_\infty$ (omitting the identity if present). Following similar arguments, but without the assumptions on $\alpha$ and $\gamma,$ we find the indices of summation given in (\ref{Kloostermann}) and Theorem \ref{PoincareTheorem}. As in their case, we find that the principal part of $ \mathcal{P}_{L_\rho}(\tau,m,N,2-k)$ at a cusp $\mu$ is constant if $\mu \not \sim \rho$ and is $\delta_{L_\rho,L_\mu}q^{\frac{-m}{t_{\rho}} } +c$ for some constant if $\mu=\rho$. Therefore, if $\mu$ is a cusp with $ L_\mu=M^{-1} L_\rho T^{r}$ for some $M \in \Gamma_0(N,)$, then the principal part of $\mathcal{P}_{L_\rho}(\tau,m,N,2-k,\nu)$ at $\mu$ is clearly $\nu(M)^{-1}e^{2\pi i r\frac{-m+\kappa}{t_\rho}}q^{\frac{-m+\kappa_\rho}{t_{\rho}} } +c.$
\end{proof}
Since harmonic Maass forms with a nonholomorphic part have a non-constant principal part at some cusp, we have the following theorem.
\begin{theorem}{\cite[Theorem 1.1]{BriOno_CoeffHmcMaaFrms}}\label{BasisConstruction}
Assuming the notation above, if $\frac{3}{2}\leq k \in \frac12\Z$, and $F(\tau)\in H_{2-k}(\Gamma_0(N),\nu)$ has principal part $P_{\rho}(\tau)=\sum_{m\geq 0}a_\rho(-m)q^{\frac{-m+\kappa_\rho}{t_\rho}}$ for each cusp $\rho\in \mathcal S_N$, then 
$$
F(\tau)=\sum_{\rho\in \mathcal S_N}\sum_{m>0}a_\rho(-m) \mathcal{P}_{\rho}(\tau,m,N,2-k,\nu)+g(\tau),
$$
where $g(\tau)$ is a holomorphic modular form. Moreover, we have that $c=0$ whenever $k>2$, and is a constant when $k=2$.
\end{theorem}

\subsection{Exact formulas for $T_g^{(-m)}$}\label{SectionExactTg}
Using Theorems \ref{PoincareTheorem} and \ref{BasisConstruction}, we can write exact formulas for the coefficients of the $T_g^{(-m)}$ provided we know its principal parts at all cusps of $\Gamma_0(Nh).$ 
With this in mind, we now regard $T_g$ as a modular function on $\Gamma_0(Nh)$ with trivial multiplier. The location and orders of the poles were determined by Harada and Lang \cite{MR1372718}. 
\begin{lemma}{\cite[Lemma 7, 9]{MR1372718}}\label{HaradaLang}
Suppose the $\Gamma_g$ is given by the symbol $N||h+e,f...,$ and let $L=\begin{pmatrix}-\delta&\beta\\ \gamma&-\alpha\end{pmatrix}\in \SL_2(\Z).$  Then $T_g|_0L$ has a pole if and only if $\left(\frac{\gamma}{(\gamma,h)},\frac{N}{h}\right) = \frac{N}{eh},$ for some $e\in \mathcal W_g$ (Note, here we allow $e=1$).
The order of the pole is given by $\frac{(h,\gamma)^2}{e h^2}.$
\end{lemma}

Harada and Lang prove this lemma by showing that if $u$ is an integer chosen such $\frac{u\gamma-\alpha \cdot (h,\gamma)}{h}$ is integral and  divisible by $e$ and $U=\begin{pmatrix}\frac{e\cdot h}{(h,\gamma)}& \frac{u}{h}\\ 0& \frac{(h,\gamma)}{h} \end{pmatrix}$, then $LU$ is an Atkin--Lehner involution $W_e\in E_g.$ Therefore, we have that
\begin{equation}\label{HaradaLangTransformation}
T_g|_0L=\sigma_g(LU)T_g\left(\frac{(h,\gamma)^2}{e h^2}\tau-\frac{u\cdot (h,\gamma)}{eh^2}\right).
\end{equation}
 Harada and Lang do not compute $\sigma_g(LU),$ however we will need these values in order to apply Theorem \ref{BasisConstruction}. Using Lemma~\ref{GammaEigenvalues}, the following procedure allows us to compute $\sigma_g(M)$ for any matrix $M\in E_g.$

Given a matrix $M\in E_g$, we may write $M$ as $M=\begin{pmatrix}ae&\frac bh\\cN&de\end{pmatrix}$ with $e\in \mathcal W_g$ and $ade-bc\frac{N}{eh}=1$. We may also write $h=h_e\cdot h_{\overline e}$, where $h_{\overline e}$ is the largest divisor of $h$ co-prime to $e$. Since $c\frac{N}{eh}$ is co-prime to both $d$ and $e$, we may chose integers $A,B,$ and $C$ $\pmod{h}$ such that:
\begin{itemize}
\item $c\frac{N}{eh}A+d$ is co-prime to $h_{\overline e}$ but is divisible by $h_e,$
\item $B\equiv -(eaA+b)(c\frac{N}{h}A+ed)^{-1}\pmod {h_{\overline e}}$ and $Bc\frac{N}{eh}+b\equiv0\pmod{h_e}$,
\item $C\equiv -c(c\frac{N}{h}A+ed)^{-1}\pmod {h_{\overline e}}$, and $C\equiv 0\pmod{h_e}.$
\end{itemize}

A calculation shows that $\widehat M:=\begin{pmatrix}1&\frac{B}{h}\\0&1\end{pmatrix}M\begin{pmatrix}1&\frac{A}{h}\\0&1\end{pmatrix}\begin{pmatrix}1&0\\CN&1\end{pmatrix}\begin{pmatrix}h_e&0\\0&h_e\end{pmatrix}$ is an Atkin--Lehner involution $W_E$ for $\Gamma_0(Nh)$ where $E=e\cdot h_e^2.$ By Lemma~\ref{GammaEigenvalues}, this implies $\sigma_g(\widehat M)=1$, and therefore 
$$\sigma_g(M)=exp\left(\frac{2\pi i}{h}(A+B+\lambda_gC)\right).$$ 
Combined with Lemma~\ref{HaradaLang}, this leads us to the following proposition.

\begin{proposition}\label{TransformationPrincipalPart}
Given a matrix $L=\begin{pmatrix}-\delta&\beta\\ \gamma&-\alpha\end{pmatrix}\in \SL_2(\Z),$ let $u$ and $U$ be chosen as above, and define 
$$
\epsilon_g(L):=\sigma_g\left(LU\right)\cdot e^{2 \pi i \frac{u\cdot(h,\gamma)}{eh^2}}.
$$
Then by (\ref{HaradaLangTransformation}), we have that
$$
T_g|_0L=\epsilon_g(L)q^{-\frac{(h,\gamma)^2}{e h^2}}+O(q).
$$
\end{proposition}

Using this notation, we are equipped to find exact formulas for the $T_g^{(-m)}$.
\begin{theorem}\label{ExactTg}
Let $g\in \MM,$ with $\Gamma_g=N|h+e,f,\dots,$ and let $\mathcal S_{Nh}$ and $\mathcal W_g$ be as above.
if $m$ and $n$ are positive integers, then there is a constant $c$ for which
$$
T^{(-m)}_g(\tau)=c+\sum_{e\in \mathcal W_g}\sum_{\substack{\frac{\alpha}{\gamma}\in \mathcal S_{Nh}\\ \left(\frac{\gamma}{(\gamma,h)},\frac{N}{h}\right) = \frac{N}{eh}}} \epsilon_g(L_\rho)^m \mathcal{P}_{\alpha/\gamma}^+(\tau,m,Nh,0).
$$
 The $n$-th coefficient of $T_g^{(-m)}(\tau)$ is given by 
\begin{equation*}
\begin{split}
\sum_{e\in \mathcal W_g}\sum_{\substack{\rho=\frac{\alpha}{\gamma}\in \mathcal S_{Nh}\\ \left(\frac{\gamma}{(\gamma,h)},\frac{N}{h}\right) = \frac{N}{eh}}}
\epsilon_g(L_\rho)^m&
2 \pi \left|\frac{-m}{n}\cdot \frac{(h,\gamma)^2 }{eh^2} \right|^{ \frac{1}{2}  } \times \\
&\sum_{\substack{ c>0 \\ (c,Nh)=(\gamma,Nh)   }} 
\frac{K_c(0,L,-m,n)}{c} \cdot 
I_{1} 
\left(\frac{4 \pi}{ c } \sqrt{   \left|\frac{-mn \cdot  (h,\gamma)^2}{eh^2}\right| } \right) .
\end{split}
\end{equation*}
\end{theorem}

\begin{proof} Every modular function is a harmonic Maass form. Therefore, the idea is to exhibit a linear combination of Maass--Poincar\'e series with exactly the same
principal parts at all cusps as $T_g^{(-m)}$. By Lemma~\ref{ppexists} and Theorem~\ref{BasisConstruction}, this form equals $T_g^{(-m)}$ up to an additive constant.
Lemma~\ref{GammaEigenvalues} (c) implies that the coefficients $c_g(n)$ of $T_g$ are supported on the arithmetic progression $n\equiv -1 \pmod h.$ The function $T^{(-m)}_g$ is a polynomial in $T_g,$ and as such must be the sum of powers of $T_g$ each of which is congruent to $m\pmod h.$ Therefore, if $M\in \Gamma_g,$ then $T^{(-m)}|_0M=\sigma_g(M)^mT^{(-m)}.$ Given $L\in \SL_2(\Z)$, let $U$ be a matrix as in (\ref{HaradaLangTransformation}) so that $LU\in \Gamma_g.$ By applying Proposition \ref{TransformationPrincipalPart}, we find
$$T_g^{(-m)}|_0L=\sigma_g(LU)^{m}T_g^{(-m)}|_0U^{-1}=\epsilon_g(L)^mq^{-m\frac{(h,\gamma)^2}{e h^2}}+O(q).$$
Theorem \ref{BasisConstruction}, along with the observations that $t_\rho=\frac{(h,\gamma)^2}{e h^2}$ and $\kappa_\rho=0$ for every $\rho=\frac{\alpha}{\gamma}=L_\rho^{-1}\infty$, implies the first part of the theorem. The formula for the coefficients follows by Theorem \ref{PoincareTheorem}.
\end{proof}

\subsection{Exact formulas for $U_g$ up to a theta function}
Following a similar process to that in the previous section, we may construct a series $\widehat{U}_g(\t)=q^{-\frac{23}{24}}+O(q^{\frac{1}{24}})$ with principal parts 
matching those of $\eta(\t) T_g(\t)$ at all cusps.  Then according to (\ref{eqn:tower-UgTg}), 
the difference $q^{\frac{1}{24}}(U_g-1)-\widehat{U}_g$ 
is a weight $\frac{1}{2}$ holomorphic modular form, which by a celebrated result \cite{SerreStark} of Serre--Stark, is a finite linear combination of
unary theta functions. This will not affect the asymptotics in  Theorem~\ref{distribution}. The functions $T_g$ and $\widehat{U}_g$ differ primarily in their weight, and in that $\widehat{U}_g$ has a non-trivial multiplier $\nu_\eta:M\to \frac{\eta(M\tau)}{(M:\tau)^{1/2}\eta(\tau)}.$ They also have slightly different orders of poles, which is accounted for by the fact that the multiplier $\nu_\eta$ implies that $\kappa_\rho=t_\rho/24$ at every cusp $\rho$ for the $\widehat{U}_g$, rather than $0$ for the $T_g$. The proof of the following theorem is the same as that of Theorem~\ref{ExactTg}, \emph{mutatis mutandis.}

\begin{theorem}\label{ExactYHatg}
Let $g\in \MM,$ with $\Gamma_g=N|h+e,f,\dots,$ and let $\mathcal S_{Nh}$ and $\mathcal W_g$ be as above.
If $m$ is a positive integer then 
$$
\widehat {U}_g=\sum_{e\in \mathcal W_g}\sum_{\substack{\rho=\frac{\alpha}{\gamma}\in \mathcal S_{Nh}\\ \left(\frac{\gamma}{(\gamma,h)},\frac{N}{h}\right) = \frac{N}{eh}}} \epsilon_g(L_\rho) \mathcal{P}_{L_\rho}^+(\tau,1,Nh,1/2,\nu_\eta).
$$
For $n$ a non-negative integer, the coefficient of $q^{n+\frac{1}{24}}$ in $\widehat{U}_g$ is given by
\begin{equation*}
\begin{split}
\sum_{e\in \mathcal W_g}&\sum_{\substack{\rho=\frac{\alpha}{\gamma}\in \mathcal S_{Nh}\\ \left(\frac{\gamma}{(\gamma,h)},\frac{N}{h}\right) = \frac{N}{eh}}}
\epsilon_g(L_\rho)
\frac{1-i}{\sqrt{2}}2 \pi \left|\frac{- \frac{(h,\gamma)^2 }{eh^2}+\frac{1}{24}}{n+\frac{1}{24}} \right|^{ \frac{1}{4}  } \times \\
&\hspace{20pt}\sum_{\substack{ c>0 \\ (c,Nh)=(\gamma,Nh)   }} 
\frac{K_c(\frac{1}{2},L,\nu_\eta,-1,n)}{c} \cdot 
I_{\frac12} 
\left(\frac{4 \pi}{ c } \sqrt{   \left|-\frac{(h,\gamma)^2 }{eh^2}+\frac{1}{24}\right|~ \left|n+\frac{1}{24}\right| } \right) .
\end{split}
\end{equation*}
\end{theorem}
This immediately admits the following corollary.
\begin{corollary}\label{ExactYg}
Given the notation above, there is a weight $\frac12$ linear combination of theta functions $h_g(\tau)$ for which the coefficient $q^n$ in $U_g(\tau)-q^{-\frac{1}{24}}h_g(\tau)$ coincides with the coefficient of $q^{n+\frac{1}{24}}$ in $\widehat U_g$, given explicitly in Theorem \ref{ExactYHatg}.
\end{corollary}

\subsection{Proof of Theorem~\ref{distribution}}
\begin{proof}[Proof of Theorem~\ref{distribution}]

Following Harada and Lang \cite{MR1372718}, we begin by defining the functions 
\begin{equation}\label{T_chi Def}
T_{\chi_i}^{(-m)}(\tau):=\frac{1}{|\MM|}\sum_{g\in \MM} \chi_i(g)T_g^{(-m)}(\tau).
\end{equation}
The orthogonality of characters imply that for $g$ and $h\in \MM$,
\begin{equation}\label{Orthogonality}
\sum_{i=1}^{194} \overline{\chi_i(g)}\chi_i(h)=\begin{cases}|C_\MM(g)| & $ if $ g $ and $ h $ are conjugate,$\\
0 & $ otherwise.$
\end{cases}
\end{equation}
Here $|C_\MM(g)|$ is the order of the centralizer of $g$ in $\MM$. Since the order of the centralizer times the order of the conjugacy class of an element is the order of the group, (\ref{Orthogonality}) and (\ref{T_chi Def}) together imply the inverse relation
$$
T_{g}^{(-m)}(\tau)=\sum_{i=1}^{194} \overline{\chi_i(g)}T_{\chi_i}^{(-m)}(\tau).
$$
In particular we have that $T_e^{(-m)}(\tau)=\displaystyle\sum_{i=1}^{194}\dim(\chi_i)T_{\chi_i}^{(-m)}(\tau)$, and therefore we can identify the $\m_i(-m,n)$ as the Fourier coefficients of the $T_{\chi_i}^{(-m)}(\tau)=\displaystyle\sum_{n=-m}^{\infty}\m_i(-m,n)q^n.$

Using Theorem~\ref{ExactTg}, we obtain exact formulas for the coefficients of $T_{\chi_i}^{(-m)}(\tau)$.
Let $g\in \MM$ with $\Gamma_g=N_g||h_g+e_g,f_g,\dots.$ If $m$ and $n$ are positive integers, then the $n$th coefficient is given exactly by
\begin{equation*}
\begin{split}
\frac{1}{|\MM|}\sum_{g\in \MM}\chi_i(g)\sum_{e\in \mathcal W_g}&\sum_{\substack{\frac{\alpha}{\gamma}\in \mathcal S_{N_gh_g}\\ \left(\frac{\gamma}{(\gamma,h_g)},\frac{N_g}{h_g}\right) = \frac{N_g}{eh_g}}} 
\epsilon_g(L_\rho)^m
2 \pi \left|\frac{-m}{n}\cdot \frac{(h_g,\gamma)^2 }{eh_g^2 } \right|^{ \frac{1}{2}  }  \\
&\hspace{-.3in}\sum_{\substack{ c>0 \\ (c,N_gh_g)=(\gamma,N_gh_g)   }} 
\frac{K_c(2-k,L,\nu,-m,n)}{c} \cdot 
I_{1} 
\left(\frac{4 \pi}{ c } \sqrt{  \left| \frac{-m n \cdot  (h_g,\gamma)^2}{eh_g^2} \right|} \right),
\end{split}
\end{equation*}
where  $\mathcal S_{N_gh_g}$ and  $\mathcal W_g$ are given as above.

Using the well-known asymptotics for the $I$-Bessel function
$$
I_k(x)\sim\frac{e^x}{\sqrt{2\pi x}}\left(1-\frac{4k^2-1}{8x}+\dots\right),
$$
we see that the formula for $\m_i(-m,n)$ is dominated by the $c=1$ term which appears only for $g=\mathrm e$ (so that $N_{\mathrm e}=h_{\mathrm e}=1$). This term yields the asymptotic

$$\m_i(-m,n)\sim
\frac{\chi_i(\mathrm e)\cdot |m|^{1/4}
}{\sqrt{2} n^{3/4}|\MM|} \cdot e^{4 \pi \sqrt{|mn|}}
$$
as in the statement of the theorem. 

The asymptotics for $\n_i(n)$ follows similarly, using the formula
$$
U_{\chi_i}(\tau):=\frac{1}{|\MM|}\sum_{g\in \MM} \chi_i(g)U_g^{(-m)}(\tau).
$$
We note that the coefficients of the theta functions $h_g(\tau)$ in Corollary~\ref{ExactYg} are bounded by constants and so do not affect the asymptotics. This yields 
$$\n_i(n)\sim
\frac{\sqrt{12}~\chi_i(\mathrm e)
}{|24n+1|^{1/2}|\MM|} \cdot e^{\frac{\pi}{6} \sqrt{23|24n+1|}}
$$
as in the theorem.
\end{proof}

\subsection{Examples of the exact formulas}
We conclude with a few examples illustrating the exact formulas for the McKay--Thompson series. 
These formulas for the coefficients generally converge rapidly. However the rate of convergence is not uniform and often requires many more terms to converge to a given precision.

\begin{example}
We first consider the case that $g$ is the identity element. Then we have $\Gamma_g=\SL_2(\Z),$ which has only the cusp infinity. In this case Theorem \ref{ExactTg} reduces to
the well known expansion

\begin{equation*}
T_{g}=J(\tau)-744=q^{-1}+\sum_{n\geq 1}
\frac{2 \pi}{\sqrt{n}}\cdot
\sum_{c>0} 
\frac{K_c(\infty,-m,n)}{c} \cdot 
I_{1} 
\left(\frac{4 \pi \sqrt{   n}}{ c }  \right)  q^n.
\end{equation*}
Table 2 below contains several approximations made by bounding the size of the $c$ term in the summation.

\medskip

\begin{table}[h]
\caption{}\label{Japprox}
\begin{tabular}{|r|c|cc|cc|}\hline
&$n=1$&&$n=5$&&$n=10$\\
\hline
$c\leq25$&$196883.661\dots$&&$333202640598.254\dots$&&$22567393309593598.047\dots$ \\
$\leq50$&$196883.881\dots$&&$333202640599.429\dots$&&$22567393309593598.660\dots$\\
$\leq75$&$196883.840\dots$&&$333202640599.828\dots$&&$22567393309593599.369\dots$\\
$\leq100$&$196883.958\dots$&&$333202640599.827\dots$&&$22567393309593599.681\dots$\\
$\infty$&$196884$&&$333202640600$&&$22567393309593600$\\ \hline
\end{tabular}
\end{table}
\end{example}

\begin{example}
The second example we consider is $g$ in the conjugacy class 4B. In this case we have $\Gamma_g=4||2+2\supset \Gamma_0(8)$. The function $T_g$ has a pole at each of the four cusps of  $\Gamma_0(8)$:
\begin{enumerate}
\item The cusp $\infty$ has $e=1,$ width $t=1,$ and coefficient $\epsilon(L_\infty)=1$.
\item The cusp $0$  has $e=2,$ width $t=8,$ and coefficient $\epsilon(L_0)=1$.
\item The cusp $1/2$  has $e=2,$ width $t=2,$ and coefficient $\epsilon(L_{1/2})=i$.
\item The cusp $1/4$  has $e=1$,  width $t=1$, and $\epsilon(L_{1/4})=-1$.
\end{enumerate}

\noindent Table 3 below contains several approximations as in Table \ref{Japprox}.

\medskip
\begin{table}[h]
\caption{}
\begin{center}
\begin{tabular}{|r|c|cc|cc|}\hline
&$n=1$&&$n=5$&&$n=10$\\
\hline
$c\leq25$&$51.975\dots$&&$4760.372\dots$&&$0.107\dots$\\
$\leq50$&$52.003\dots$&&$4759.860\dots$&&$0.117\dots$\\
$\leq75$&$52.041\dots$&&$4760.066\dots$&&$0.092\dots$\\
$\leq100$&$51.894\dots$&&$4760.049\dots$&&$0.040\dots$\\
$\infty$&$52$&&$4760$&&$0$\\ \hline
\end{tabular}
\end{center}
\end{table}
\end{example}

\section{Umbral Moonshine}\label{sec:um}

In this penultimate section, we review the recently discovered, and rapidly developing field of umbral moonshine. 

\subsection{K3 Surfaces}\label{sec:um:k3}

Eguchi--Ooguri--Tachikawa reignited the field of moonshine with their 2010 observation \cite{Eguchi2010} that dimensions of representations of the largest Mathieu group, $M_{24}$, occur as multiplicities of superconformal algebra characters in the K3 elliptic genus. 

To formulate their observation more precisely, recall that a {\em complex K3 surface} is a compact connected complex manifold $\md$ of dimension $2$, with $\Omega_\md^2\simeq \mathcal{O}_\md$ and $H^1(\md,\mathcal{O}_\md)=0$. 
(See \cite{MR2030225,MR785216} for introductory accounts.) Following Witten \cite{MR970278,MR885560}, the {\em elliptic genus} of a complex manifold $\md$ of dimension $d$ is defined to be 
\begin{gather}\label{eqn:um:k3-ZM}
Z_\md:=\int_\md\ch(\mathbb{E}_{q,y})\td(\md),
\end{gather}  
where $\td(\md)$ is the Todd class of $\md$, and $\ch(\mathbb{E}_{q,y})$ is the Chern character of the formal power series
\begin{gather}\label{eqn:um-Eqy}
	\mathbb{E}_{q,y}=y^{\frac d 2}\bigotimes_{n= 1}^\infty
	\left(
	\bigwedge\nolimits_{-y^{-1}q^{n-1}}T_\md\otimes\bigwedge\nolimits_{-yq^n}T_\md^*
	\otimes \bigvee\nolimits_{q^n}T_\md\otimes\bigvee\nolimits_{q^n}T_\md^*
	\right),
\end{gather}
whose coefficients are virtual vector bundles, obtained as sums of tensor products of the exterior and symmetric powers of the holomorphic tangent bundle $T_\md$, and its dual bundle $T_\md^*$. (Cf. also \S1 of \cite{Gri_EllGenCYMnflds} or Appendix A of \cite{Dijkgraaf1997}.) In (\ref{eqn:um-Eqy}) we interpret $\bigvee_tE$ in direct analogy with $\bigwedge_tE$ (cf. (\ref{eqn:class-bigwedget})), replacing exterior powers $\wedge^kE$ with symmetric powers $\vee^kE$.

Since a complex K3 surface $\md$ has trivial canonical bundle, and hence vanishing first Chern class, its {elliptic genus} is a {\em weak Jacobi form} $Z_\md(\tau,z)$ of weight zero and index $\dim(\md)/2=1$---see \cite{Gri_EllGenCYMnflds} or \cite{2004math......5232H} for proofs of this fact---once we set $q=e(\tau)$ and $y=e(z)$. This means\footnote{See \cite{Dabholkar:2012nd} or \cite{eichler_zagier} for more on Jacobi forms, including the general transformation formula in case of weight different from zero or index different from one.} that $Z_\md(\tau,z)$ is a holomorphic function on $\HH\times\CC$, satisfying
\begin{gather}
	\begin{split}
	Z_\md(\tau,z)
	&=
	e\left(\frac{-cz^2}{c\tau+d}\right)
	Z_\md\left(\frac{a\tau+b}{c\tau+d},\frac{z}{c\tau+d}\right)\\
	&=
	e(\lambda^2\tau+2\lambda z)
	Z_\md(\tau,z+\lambda \tau+\mu)
	\end{split}
\end{gather}
for $\left(\begin{smallmatrix}a&b\\c&d\end{smallmatrix}\right)\in \SL_2(\Z)$ and $\lambda,\mu\in\Z$, with a Fourier expansion $Z_\md(\tau,z)=\sum_{n,r}c(n,r)q^ny^r$ such that $c(n,r)=0$ whenever $n<0$. 

It is known that $Z_\md(\tau,z)$ specializes to the Euler characteristic $\chi(\md)=24$ when $z=0$ (cf. e.g. \cite{Gri_EllGenCYMnflds}). The space of weak Jacobi forms of weight zero and index one is one-dimensional (cf. \cite{eichler_zagier}), so $Z_\md$ is independent of the choice of $\md$. In fact, we have
\begin{gather}\label{eqn:um-ZK3}
	Z_\md(\tau,z)=8\left(
	\frac{\theta_2(\tau,z)^2}{\theta_2(\tau,0)^2}
	+\frac{\theta_3(\tau,z)^2}{\theta_3(\tau,0)^2}
	+\frac{\theta_4(\tau,z)^2}{\theta_4(\tau,0)^2}
	\right),
\end{gather}
where the $\theta_i(\tau,z)$ are the usual Jacobi theta functions:
\begin{gather}
	\begin{split}
	\theta_2(\tau,z)&:=\sum_{n\in\Z}y^{n+1/2}q^{(n+1/2)^2/2}
	=y^{1/2}q^{1/8}\prod_{n>0}(1+y^{-1}q^{n-1})(1+yq^{n})(1-q^n)\\
	\theta_3(\tau,z)&:=\sum_{n\in\Z}y^{n}q^{n^2/2}=\prod_{n>0}(1+y^{-1}q^{n-1/2})(1+yq^{n-1/2})(1-q^n)\\
	\theta_4(\tau,z)&:=\sum_{n\in\Z}(-1)^ny^{n}q^{n^2/2}=\prod_{n>0}(1-y^{-1}q^{n-1/2})(1-yq^{n-1/2})(1-q^n).
	\end{split}
\end{gather}

In Witten's original analysis \cite{MR970278,MR885560} the elliptic genus $Z_\md$ is the graded trace of an integer-valued operator on a Hilbert space arising from a {\em supersymmetric nonlinear sigma model} on $\md$. 
In the case that $\md$ is a K3 surface---see \cite{MR1479699,MR1416354} for analyses of the sigma models associated to K3 surfaces---it is expected that the corresponding Hilbert space admits an unitary action by the (small) $N=4$ superconformal algebra (cf. \cite{Eguchi1987}). At least this is known for some special cases, so  (\ref{eqn:um-ZK3}) can be written as an integer combination of the irreducible unitary $N=4$ algebra characters. This leads (cf. \cite{Eguchi2009a,Eguchi1989}) to an expression 
\begin{gather}
	Z_\md(\tau,z)=24\mu(\tau,z)\cdot \frac{\theta_1(\tau,z)^2}{\eta(\tau)^3}+H(\tau)\cdot \frac{\theta_1(\tau,z)^2}{\eta(\tau)^3},
\end{gather}
where $\theta_1(\tau,z)$ is the Jacobi theta function
\begin{gather}
	\theta_1(\tau,z):=i\sum_{n\in\Z}(-1)^ny^{n+1/2}q^{(n+1/2)^2/2}
	=iy^{1/2}q^{1/8}\prod_{n>0}(1-y^{-1}q^{n-1})(1-yq^n)(1-q^n),
\end{gather}
$\mu(\tau,z)$ denotes the Appell--Lerch sum defined by
\begin{gather}
	\mu(\tau,z):=\frac{i y^{1/2}}{\theta_1(\tau,z)}\sum_{n\in\Z}(-1)^n\frac{y^nq^{n(n+1)/2}}{1-y q^n},
\end{gather}
and $q^{1/8}H(\tau)$ is a power series in $q$ with integer coefficients,
\begin{gather}\label{eqn:um-HFouexp}
	H(\tau)=-2q^{-1/8}+90q^{7/8}+462q^{15/8}+1540q^{23/8}+4554q^{31/8}+11592q^{39/8}+\ldots.
\end{gather}

The surprising observation of \cite{Eguchi2010} is that each coefficient of a non-polar term appearing in (\ref{eqn:um-HFouexp}) is twice the dimension of an irreducible representation (cf. \cite{ATLAS}) of the sporadic simple group $M_{24}$, discovered by \'Emile Mathieu \cite{Mat_1861,Mat_1873} more than 150 years ago. (Generally, the coefficient of a positive power of $q$ in (\ref{eqn:um-HFouexp}) is some non-negative integer combination of dimensions of representations of $M_{24}$.) Thus $H(\tau)$ serves as an analogue of $J(\tau)$ (cf. (\ref{eqn:intro-expJ})) for $M_{24}$, and 
\begin{gather}
90=45+45
\end{gather}
is the Mathieu analogue of McKay's monstrous observation (\ref{eqn:intro-mckobs}). 

The analogy with monstrous moonshine was quickly taken up, with the determination by Cheng \cite{MR2793423}, Eguchi--Hikami \cite{Eguchi2010a}, and Gaberdiel--Hohenegger--Volpato \cite{Gaberdiel2010a, Gaberdiel2010}, of {\em Mathieu McKay--Thompson series} 
\begin{gather}\label{eqn:um-HgtrK}
	H_g(\tau)=-2q^{-1/8}+\sum_{n>0}\tr(g|K_{n-1/8})q^{n-1/8}
\end{gather}
associated to a graded $M_{24}$-module $K=\bigoplus_{n>0}K_{n-1/8}$, such that $K_{7/8}$ is the sum of the two $45$-dimensional irreducible representations of $M_{24}$, and $K_{15/8}$ the sum of the two $231$-dimensional irreducible representations, etc. (See \cite{MR2985326} for a detailed review of Mathieu moonshine, and explicit descriptions of the $H_g$ in particular.) 
We have the following beautiful recent result of Gannon \cite{MR2985326}.

\begin{theorem}[Gannon]
There is a graded $M_{24}$-module $K=\bigoplus_{n>0}K_{n-1/8}$ for which (\ref{eqn:um-HgtrK}) is true (given that the $H_g$ are as described in 
\cite{MR2985326}).
\end{theorem}

\begin{remark}
A concrete construction of $K$ remains unknown.
\end{remark}

The observer may ask: how were the $H_g$ determined, if the module $K$ is as yet unknown? To explain this, note that the subscript in $M_{24}$ is a reference to the fact that $M_{24}$ 
is distinguished amongst permutation groups: it may be characterized as the unique proper subgroup of the alternating group $A_{24}$ that acts quintuply transitively on $24$ points (cf. \cite{MR1010366}). Write $\chi_g$ for the number of fixed points of an element $g\in M_{24}$, in this defining permutation representation. 

The first few terms of $H_g$ are determined by the Eguchi--Ooguri--Tachikawa observation on (\ref{eqn:um-HFouexp}), for it indicates that the coefficient of $q^{7/8}$ in $H_g$ should be the trace of $g$ on the sum of the two $45$-dimensional irreducible representations, and the coefficient of $q^{15/8}$ should be the trace of $g$ on the sum of the two $231$-dimensional irreducible representations, etc. To determine the remaining infinitely many terms, modularity may be used: the series $H_g$, determined in \cite{MR2793423,Eguchi2010a,Gaberdiel2010a, Gaberdiel2010}, have the property that
\begin{gather}\label{eqn:um-Zg}
	Z_g(\tau,z):=\chi_g\mu(\tau,z)\cdot \frac{\theta_1(\tau,z)^2}{\eta(\tau)^3}+H_g(\tau)\cdot \frac{\theta_1(\tau,z)^2}{\eta(\tau)^3}
\end{gather}
is a weak Jacobi form of weight zero and index one for $\Gamma_0^J(N):=\Gamma_0(N)\ltimes\Z^2$ (with non-trivial multiplier when $\chi_g=0$), where $N=o(g)$ is the order of $g$, and 
$\Gamma_0(N)$ is as in (\ref{Gamma0N}).

Thus Mathieu moonshine entails twisted, or {\em twined} versions (\ref{eqn:um-Zg}) of the K3 elliptic genus (\ref{eqn:um-ZK3}), but the single variable series $H_g(\tau)$ may also be studied in their own right, as automorphic objects of a particular kind: it turns out that they are {\em mock modular forms\footnote{The notion of mock modular form has arisen recently, from Zwegers' foundational work \cite{zwegers} on Ramanujan's mock theta functions \cite{MR947735,MR2280843}, and the subsequent contributions \cite{BringmannOno2006} and \cite{zagier_mock}. 
We refer to \cite{Dabholkar:2012nd,Ono_unearthing,zagier_mock} for nice introductions to the theory.}} of weight $1/2$, for various groups $\Gamma_0(N)$, with {\em shadows} $\chi_g\eta(\tau)^3$. This means that the {\em completed} functions 
\begin{gather}
	\widehat{H}_g(\tau):=H_g(\tau)+
	\chi_g\frac{1}{2\sqrt{i}} \int_{-\overline{\t}}^{\infty}
	\overline{\eta\left(-\overline{w}\right)^3}
	\frac{{\rm d}w}{\sqrt{w+\t}}
\end{gather}
are harmonic Maass forms of weight $1/2$, with the same multiplier system as $\eta(\tau)^{-3}$ when $\chi_g\neq 0$. (In case $\chi_g=0$, i.e. when $H_g$ is already a modular form, the multiplier is slightly different. See e.g. \cite{MR2985326}. The groups $\Gamma_0(o(g))$ for which $\chi_g\neq 0$ are characterized in \cite{2012arXiv1212.0906C}.)
The function $H_g(\tau)$, being the holomorphic part of $\widehat{H}_g(\tau)$, is the {\it mock modular form}.

In contrast to the twined K3 elliptic genera $Z_g$, the mock modular forms $H_g$ are distinguished, in a manner directly analogous to the McKay--Thompson series $T_g$ of monstrous moonshine: it is shown in \cite{Cheng2011} that the $H_g$ admit a uniform description in terms of Rademacher sums, in direct analogy with Theorem \ref{thm:radsums-TgTGammag}. 
(We refer to \cite{Cheng2011} or the review \cite{MR2985326} for a precise statement of this result.) Since the coincidence between the monstrous McKay--Thompson series and (normalized) Rademacher sums depends in a crucial way upon the genus zero property of monstrous moonshine, as evidenced by Theorem \ref{thm:intro:qgrav-Rgammainv}, it is natural to identify the Rademacher sum realization of the $H_g$ as the Mathieu moonshine counterpart to the genus zero property of monstrous moonshine.

As we have hinted above, the Rademacher sum property that distinguishes the $T_g$ and $H_g$ does not hold for the weight zero Jacobi forms $Z_g$ (cf. (\ref{eqn:um-Zg})). A Poincar\'e series approach to Jacobi forms is described in \cite{jacteg}, 
using the foundations established in \cite{MR2680205,MR2805582},
and it is verified there that the $Z_g$ are not all realized in this way. On the other hand, the main result of \cite{jacteg} is the Poincar\'e series construction of certain Maass--Jacobi forms of weight one, naturally associated to elements of $M_{24}$. Thus we can expect that Jacobi forms of weight one, rather than the $Z_g$ of (\ref{eqn:um-Zg}), will play an important role in a comprehensive conceptual explanation of the Mathieu moonshine phenomenon.

Note that some of the functions $Z_g$ admit a geometric interpretation in terms of K3 surfaces. Namely, it has been established in \cite{MR3271177} that if $\bar g$ is a symplectic automorphism of a K3 surface $M$ then the natural $\bar g$-equivariant modification of (\ref{eqn:um:k3-ZM}) coincides with $Z_g$, for a suitable element $g\in M_{24}$. However, not all $Z_g$ arise in this way. Please see \S\ref{sec:um:sigma} for a fuller discussion of this.

\subsection{Niemeier Lattices}\label{sec:um:niemeier}

Vector-valued versions of the Rademacher sums that characterize the $H_g$ were used in \cite{UM} to identify Mathieu moonshine as a special case of six directly similar correspondences, between conjugacy classes in certain finite groups and distinguished (vector-valued) mock modular forms of weight $1/2$. 
Since the mock modular forms arising seemed to be characterized by their shadows, this was dubbed {\em umbral moonshine} in \cite{UM}. 

The conjectures of \cite{UM} were greatly expanded in \cite{MUM}, following an observation of Glauberman (cf. the Acknowledgement in \cite{MUM}), that the finite groups identified in \cite{UM} also appear as automorphism groups of codes associated to deep holes in the Leech lattice (cf. \cite{ConParSlo_CvgRadLeeLat} or \cite{ATLAS}). 

To explain the significance of this, recall that an {\em integral lattice} is a free abelian group $L$ together with a symmetric bilinear form $\langle\cdot\,,\cdot\rangle:L\times L\to\Z$. A lattice $L$ is called {\em positive-definite} if $\langle \lambda,\lambda\rangle\geq 0$ for all $\lambda\in L$, with equality only when $\lambda=0$. It is called {\em even} if $\langle\lambda,\lambda\rangle\in 2\Z$ for all $\lambda\in L$, and {\em self-dual} if $L=L^*$, for $L^*$ the {\em dual} of $L$,
\begin{gather}\label{eqn:um-duallat}
	L^*:=\left\{\mu\in L\otimes_{\Z}\QQ\mid \langle\lambda,\mu\rangle\in\Z\Leftarrow \lambda\in L\right\}.
\end{gather}

The even self-dual positive-definite lattices of rank $24$ have been classified \cite{Nie_DefQdtFrm24} (see also \cite{MR666350,MR558941}) by Niemeier: there are $24$ in total, up to isomorphism. They are characterized by their root systems---i.e. the configurations of their vectors with square length equal to $2$---and the Leech lattice $\Lambda$ (cf. \S\ref{previouswork}) is the unique such lattice whose root system is empty. We refer to the remaining $23$ as the {\em Niemeier lattices}. The {\em Niemeier root systems} are the root systems of the Niemeier lattices, and they are described explicitly as
\begin{gather}\label{eqn:um-NX1}
A_1^{24},\;A_2^{12},\;A_3^{8},\;A_4^6,\;A_6^4,\;A_{12}^2,\\
	\begin{split}
A_5^4D_4,\;A_7^2D_5^2,\;A_8^3,\;A_9^2D_6,\;&A_{11}D_7E_6,\;A_{15}D_9,\;A_{17}E_7,\;A_{24},\\\label{eqn:um-NX2}
D_4^6,\;D_6^4,\;D_8^3,\;D_{10}E_7^2,\;&D_{12}^2,\;D_{16}E_8,\;D_{24},E_6^4,\;E_8^3,
	\end{split}
\end{gather}
in terms of the irreducible, simply-laced (i.e. ADE type) root systems. (See \cite{MR1662447} or \cite{MR0323842} for more on root systems.) 

In (\ref{eqn:um-NX1}) and (\ref{eqn:um-NX2}) we use juxtaposition as a shorthand for direct sum, so that $A_1^{24}$ denotes $24$ copies of the $A_1$ root system, and $A_{11}D_7E_6$ is shorthand for $A_{11}\oplus D_7\oplus E_6$, etc.
The subscripts indicate ranks. The {Coxeter numbers} of the ADE root systems are given by 
\begin{gather}\label{eqn:um-Coxeternumbers}
m(A_n)=n+1,\; m(D_n)=2n-2,\; m(E_6)=12,\; m(E_7)=18,\; m(E_8)=30,
\end{gather}
and one can check that the Niemeier root systems (\ref{eqn:um-NX1}), (\ref{eqn:um-NX2}) are exactly those unions of ADE type root systems for which the total rank is $24$, and the Coxeter number is constant across irreducible components. 

For $X$ a Niemeier root system and $N^X$ the corresponding Niemeier lattice, define the {\em outer automorphism group} of $N^X$ by setting 
\begin{gather}
	\Out(N^X):=\Aut(N^X)/W^X,
\end{gather}
where $W^X$ denotes the subgroup of $\Aut(N^X)$ generated by reflections in root vectors. Applying this construction to the Leech lattice, corresponding to $X=\emptyset$, we obtain the {\em Conway group}, 
\begin{gather}\label{eqn:um-Co0}
\Co_0:=\Aut(\Lambda), 
\end{gather}
so named in light of Conway's detailed description \cite{MR0237634,MR0248216} of its structure. A number of the $26$ sporadic simple groups appear as subgroups, or quotients of subgroups of $\Co_0$, including the three sporadic simple Conway groups, $\Co_1$, $\Co_2$ and $\Co_3$. The Conway group $\Co_0$ is a double cover of the first, and largest of these,
\begin{gather}\label{eqn:um-Co1}
\Co_1\simeq \Aut(\Lambda)/\{\pm \Id\}.
\end{gather}
Note that $M_{24}$ is naturally a subgroup of $\Co_0$, and also $\Co_1$, for if $\{\lambda_i\}\subset\Lambda$ is a set of $24$ vectors such that $\langle\lambda_i,\lambda_j\rangle=8\delta_{ij}$, then the subgroup of $\Co_0$ that stabilizes this set $\{\lambda_i\}$ is a copy of $M_{24}$.
\begin{gather}\label{eqn:um-M24}
	M_{24}\simeq\left\{ g\in \Co_0\mid \{g(\lambda_i)\}\subset\{\lambda_i\}\right\}.
\end{gather}

According to Conway--Parker--Sloane \cite{ConParSlo_CvgRadLeeLat}, the Niemeier root systems classify the {\em deepest holes} in the Leech lattice, being the points in $\Lambda\otimes_{\Z}\RR$ at maximal distance from vectors in $\Lambda$. 
Moreover, this correspondence is strong enough that the Niemeier outer automorphism groups $\Out(N^X)$ are also visible inside the Conway group, $\Co_0$. More precisely, if $x\in \Lambda\otimes_{\Z}\RR$ is a deep hole, with corresponding Niemeier root system $X$ according to \cite{ConParSlo_CvgRadLeeLat}, then the stabilizer $\Aut(\Lambda,x)$ of $x$ in $\Aut(\Lambda)$ has a normal subgroup $C^x$ such that 
\begin{gather}\label{eqn:um-AutCOut}
	\Aut(\Lambda,x)/C^x\simeq \Out(N^X).
\end{gather}
The subgroup $C^x$ even encodes a method for constructing $N^X$, as is explained in detail in \cite{MR661720}, for if $L^X$ denotes the sub lattice of $N^X$ generated by roots, then $N^X$ is determined by its image in $(L^X)^*/L^X$ (cf. (\ref{eqn:um-duallat})) under the natural map $N^X\to (L^X)^*/L^X$. Write $C^X$ for this subgroup of $(L^X)^*/L^X$, called the {\em glue code} of $X$ in \cite{MR661720} (see also \cite{MR666350}). Then $C^x$ is isomorphic to $C^X$, according to \cite{MR661720}.
\begin{gather}\label{eqn:um-CAutOut}
	1\to C^X\simeq C^x\to\Aut(\Lambda,x)\to\Out(N^X)\to 1
\end{gather}

Thus $\Out(N^X)$ acts as automorphisms on the glue code $C^X$, and Glauberman's observation suggests an extension of the results of \cite{UM}, whereby distinguished vector-valued mock modular forms $H^X_g=(H^X_{g,r})$ are associated to elements $g$ in the {\em umbral groups} 
\begin{gather}\label{eqn:um-GX}
G^X:=\Out(N^X), 
\end{gather}
for each Niemeier root system $X$. The realization of this suggestion is described in detail in \cite{MUM}. 
For $X=A_1^{24}$, the glue code $C^X$ is a copy of the {\em extended binary Golay code} (cf. \cite{MR1662447} or \cite{MR1667939}), and $G^X$ is its full automorphism group, $M_{24}$. Thus, from the Niemeier root system perspective, Mathieu moonshine is the special case of umbral moonshine corresponding to the root system $A_1^{24}$.

In (\ref{eqn:um-NX1}) we have separated out the Niemeier root systems of the form $A_{n}^d$ with $d=24/n$ even. It is exactly these cases of umbral moonshine that are discussed in \cite{UM}. The original umbral moonshine observation of Eguchi--Ooguri--Tachikawa stemmed from consideration of the weight zero, index one weak Jacobi form $Z_\md$ (cf. (\ref{eqn:um-ZK3})), realized as the K3 elliptic genus. The analysis of \cite{UM} is, to some extent, similarly motivated, including the attachment of a weight zero, index $n$ weak Jacobi form $Z^{(n+1)}_g(\tau,z)$ to each $g\in G^X$, for each Niemeier root system $X=A_n^d$ with $d=24/n$ even. 

A notion of {\em extremal Jacobi form} is formulated in \cite{UM}, motivated by the representation theory of the $N=4$ superconformal algebra, and it is proven\footnote{The main step in the classification given in \cite{UM} is a demonstration that the existence of an extremal Jacobi form of index $m-1$ implies the vanishing of $L(f,1)$ for all new forms $f$ of weight $2$ and level $m$, where $L(f,s)$ is the Dirichlet series naturally attached to $f$ (cf. e.g. \S3.6 of \cite{Shi_IntThyAutFns}). At this point one expects extremal Jacobi forms to be very few in number, on the strength of the Birch--Swinnerton-Dyer conjecture (cf. \cite{MR0179168,MR2238272}), for example. This machinery is evidently quite powerful, and we may anticipate further applications to umbral moonshine in the future.} there that the six functions $Z^{(n+1)}:=Z^{(n+1)}_e$, for $n\in \{1,2,3,4,6,12\}$, exhaust all examples. Thus the cases (\ref{eqn:um-NX1}) of umbral moonshine considered in \cite{UM} are distinguished from the point of view of Jacobi forms of weight zero. 

By contrast, there seems to be no natural way to associate weight zero Jacobi forms to the Niemeier root systems not\footnote{The cases $A_8^3$ and $A_{24}$ do come with weight zero Jacobi forms attached, which are obtained via a slight weakening of the notion of extremal Jacobi form formulated in \cite{UM}. Cf. \S4.3 of \cite{MUM}.} of the {\em pure A-type}, $A_n^d$. Rather, the mock modular forms $H^X_g$ described in \cite{MUM} naturally appear as the {\em theta-coefficients} of {\em finite parts} of certain meromorphic Jacobi forms $\psi^X_g$ of weight $1$ and index $m$, 
\begin{gather}\label{eqn:um-psiXg}
	\psi^X_g(\tau,z)=\psi^{X,P}_g(\tau,z)+\sum_{r\pmod{2m}}H^X_{g,r}(\tau)\theta_{m,r}(\tau,z),
\end{gather}
where $m=m(X)$ is the Coxeter number of any irreducible component of $X$ (cf. (\ref{eqn:um-Coxeternumbers})). 

Here, meromorphic means that we allow poles in the functions $z\mapsto \psi^X_g(\t,z)$, at torsion points $z\in \QQ\t+\QQ$. The Weierstrass $\wp$ function
\begin{gather}
	\wp(\t,z):=\frac{1}{z^2}+\sum_{\substack{\omega\in\Z\t+\Z\\\omega\neq 0}}\frac{1}{(z+\omega)^2}-\frac{1}{\omega^2}
\end{gather}
is a natural example (with weight two and index zero). 

The decomposition (\ref{eqn:um-psiXg}) is described in detail in \cite{MUM}, following the general structural results on  meromorphic Jacobi forms established in \cite{Dabholkar:2012nd,zwegers}. For now let us just mention that the first summand on the right hand side is the {\em polar part} of $\psi^X_g$, defined as in \S8.2 of \cite{Dabholkar:2012nd}, and
\begin{gather}
\theta_{m,r}(\tau,z):=\sum_{k\in\Z}y^{2km+r}q^{(2km+r)^2/4m}
\end{gather} 
evidently depends only on $r$ modulo $2m$. 

A number of the meromorphic Jacobi forms attached to Niemeier root systems by umbral moonshine also appear amongst the specific examples of \cite{Dabholkar:2012nd}, where the main application is the computation of quantum degeneracies of black holes in certain string theories. However, whilst some speculations are offered in \S5.5 of \cite{UM}, no direct relationship between umbral moonshine and string theory has been formulated as yet. 

We have seen in \S\ref{sec:um:k3} that the mock modular forms attached to $M_{24}$ by Mathieu moonshine (i.e. umbral moonshine for $X=A_1^{24}$) may be characterized as Rademacher sums, and this serves as an umbral analogue of the principal modulus/genus zero property of monstrous moonshine, on the strength of Theorem \ref{thm:intro:qgrav-Rgammainv}. It is natural to ask for an extension of this result to all cases of umbral moonshine. 

Conjecture 5.4 of \cite{UM} amounts to the prediction that vector-valued generalizations of the Rademacher sums of \cite{Cheng2011} will recover the $H^X_g$ for $X=A_n^d$ with $d$ even (cf. (\ref{eqn:um-NX1})), and Conjecture 3.2 of \cite{mumcor} is an extension of this to all Niemeier root systems $X$.
Thus a positive solution to Conjecture 3.2 of \cite{mumcor} will verify the umbral analogue of the principal modulus property of monstrous moonshine. So far, the Rademacher sum conjecture for umbral moonshine is known to be true only in the case that $X=A_1^{24}$, but a program to analyze the Rademacher sum conjecture for more general cases of umbral moonshine, via the theory of Maass--Jacobi forms (cf. \cite{MR2680205,MR2805582}), has been initiated in \cite{jacteg}.

A notion of optimal growth was formulated in \S6.3 of \cite{MUM}, following the work \cite{Dabholkar:2012nd}, with a view to extending Conjecture 5.4 of \cite{UM}. It is now known that this condition does not uniquely determine the $H^X_g$ for general $X$ (see \cite{mumcor} for a full discussion of this), but all the $H^X_g$ serve as examples. With this in mind, it is interesting to note that many of Ramanujan's mock theta functions \cite{MR947735,MR2280843} appear as components of the umbral McKay--Thompson series $H^X_g$. (Cf. \S4.7 of \cite{UM} and \S5.4 of \cite{MUM}.)

\subsection{Modules}\label{sec:um:mods}

As we have explained above, the Rademacher sum property of umbral moonshine is a natural counterpart to the principal modulus, or genus zero property of monstrous moonshine (cf. \S\ref{previouswork}), formulated in a detailed way by Conway--Norton \cite{MR554399}. 

The natural counterpart to Thompson's conjecture, Conjecture \ref{conj:Thompson}, verified by the Frenkel--Lepowsky--Meurman construction \cite{FLMPNAS,FLMBerk,FLM} of the moonshine module $\Vnat$, together with Borcherds' work \cite{borcherds_monstrous}, is the 
following (cf. \S6.1 of \cite{MUM}, and \S2 of \cite{mumcor}).

\begin{conjecture}[Cheng--Duncan--Harvey]\label{conj:umc1}
For each Niemeier root system $X$, there is a bi-graded $G^X$-module 
\begin{gather}\label{eqn:um-KX}
\check{K}^X=\bigoplus_{r\in I^X}\bigoplus_{\substack{D\in\Z\\D=r^2\pmod{4m}}}\check{K}^X_{r,-D/4m},
\end{gather} 
such that the vector-valued umbral McKay--Thompson series $H^X_g=(H^X_{g,r})$ is recovered\footnote{In the original formulation, Conjecture 6.1 of \cite{MUM}, the function $H^X_{g,r}$ in (\ref{eqn:um-HXgrKXrd}) is replaced with $3H^X_{g,r}$ in the case that $X=A_8^3$. It also predicted that $\check{K}^X_{r,-D/4m}$ is a virtual $G^X$-module in case $X=A_8^3$ and $D=0$. Recently, a modification of the specification of the $H^X_{g}$ for $X=A_8^3$ has been discovered, which leads to the simpler, more uniform formulation appearing here. We refer to \cite{mumcor} for a full discussion of this.} from the graded trace of $g$ on $\check{K}^X$ via
\begin{gather}\label{eqn:um-HXgrKXrd}
	H^X_{g,r}(\tau)=-2q^{-1/4m}\delta_{r,1}
	+\sum_{\substack{D\in\Z\\D=r^2\pmod{4m}}}\tr(g|\check{K}^X_{r,-D/4m})q^{-D/4m}
\end{gather}
for $r\in I^X$.
\end{conjecture}

In (\ref{eqn:um-KX}) and (\ref{eqn:um-HXgrKXrd}), $m=m(X)$ is the Coxeter number of any irreducible component of $X$, as in (\ref{eqn:um-psiXg}). The $H^X_{g,r}$ satisfy $H^X_{g,-r}=-H^X_{g,r}$, so the umbral McKay--Thompson series $H^X_g$ is determined by its components $H^X_{g,r}$ with $0<r<m$. If the highest rank irreducible component of $X$ is of type D or E then there are more symmetries amongst the $H^X_{g,r}$, and the definition of the set $I^X\subset\Z/2m\Z$ reflects this: if $X$ has an A-type component then $I^X:=\{1,2,3,\ldots,m-1\}$. If the highest rank component of $X$ is of type D then $m=2\mod 4$, and $I^X:=\{1,3,5,\ldots,m/2\}$. The remaining cases are $X=E_6^4$, in which case $I^X:=\{1,4,5\}$, and $X=E_8^3$, for which $I^X:=\{1,7\}$.

As mentioned in \S\ref{sec:um:k3}, the existence of the module $\check{K}^X$ for $X=A_1^{24}$ has been proven by Gannon \cite{Gannon:2012ck}. More specifically, Gannon has shown that the coefficients of the non-negative powers of $q$ in $H^X_{g}$ for $X=A_1^{24}$ are traces of elements of $M_{24}$ on direct sums of irreducible $M_{24}$-modules. A priori, we might have needed $\CC$-linear combinations of such traces in order to recover the $H^X_{g,r}$. 

In forthcoming work \cite{DGO_UMC1}, the authors confirm the validity of Conjecture \ref{conj:umc1}.
\begin{theorem}[Duncan--Griffin--Ono]\label{thm:um:mods}
Conjecture~\ref{conj:umc1} is true.
\end{theorem}
Theorem \ref{thm:um:mods} serves, to a certain extent, as the umbral counterpart to Borcherds' result, Theorem \ref{thm:intro:class-borcherds}. Indeed, the method of \cite{DGO_UMC1} may be used to give an alternative proof of the existence of the $\MM$-module $V^{\natural}$, for which the associated graded trace functions are the normalized principal moduli of the genus zero groups $\Gamma_g$, identified by Conway--Norton in \cite{MR554399}.

Nonetheless, there is still work to be done, for in order to have a direct counterpart to Theorem \ref{thm:intro:class-borcherds} we require concrete constructions of the $\check{K}^X$. 
In the case of monstrous moonshine, the construction of $\Vnat$ due to Frenkel--Lepowsky--Meurman came equipped with rich algebraic structure, ultimately leading to the notion of vertex operator algebra, and powerful connections to physics.
We can expect that a full explanation of the umbral moonshine phenomena will require analogues of this for all the $\check{K}^X$.

Just such an analogue for $X=E_8^3$ has recently been obtained in \cite{mod3e8}, where a super vertex operator algebra $V^X$ is constructed, together with an action of $G^X\simeq S_3$, such that the components of the vector-valued mock modular forms $H^X_{g}=(H^X_{g,r})$ are recovered from traces of elements of $G^X$ on canonically-twisted modules for $V^X$. The main ingredient in the construction of \cite{mod3e8} is an adaptation of the familiar (to specialists) lattice vertex algebra construction (cf. \cite{Bor_PNAS,FLM}), to cones in indefinite lattices. The choice of cone is in turn inspired by Zwegers' work \cite{MR2558702} on a particular pair of the fifth order mock theta functions of Ramanujan.

In \cite{umvan2,umvan4} a different approach to the module problem is considered, whereby meromorphic Jacobi forms associated to the $H^X_g$ are recovered as graded traces on canonically-twisted modules for certain super vertex algebras. In \cite{umvan4} constructions are given for the $\psi^X_g$ of (\ref{eqn:um-psiXg}), for $X\in\{A_3^8, A_4^6, A_6^4,A_{12}^2\}$. In \cite{umvan2} certain half-integral index analogues of the $\psi^X_g$ are recovered, for $X\in\{D_6^4, D_8^3, D_{12}^2, D_{24}\}$.

As we will explain in more detail in the next section, recent work \cite{2014arXiv1406.5502C,2015arXiv150307219C} constructs modules underlying assignments of vector-valued mock modular forms to the sporadic simple groups $M_{24}$, $M_{23}$ and $M_{22}$. Here, $M_{23}$ denotes the maximal subgroup of $M_{24}$ composed of elements fixing any given point in the defining permutation representation (cf. \S\ref{sec:um:k3}), and $M_{22}$ is obtained similarly from $M_{23}$, as the subgroup stabilizing a point in its natural permutation representation of degree $23$. Although the mock modular forms realized in \cite{2014arXiv1406.5502C,2015arXiv150307219C} are not directly related to the $H^X_g$, it seems likely that the construction used therein holds important hints for future developments in the module problem for umbral moonshine.

\subsection{Sigma Models}\label{sec:um:sigma}

Recall from \S\ref{sec:um:k3} that $\Omega_M^2\simeq\mathcal{O}_M$ for $M$ a complex K3 surface. An automorphism of $M$ that induces the trivial action on $H^0(M,\Omega_M^2)$ is called {\em symplectic}. It is a celebrated result of Mukai \cite{Mukai} (cf. also \cite{Kondo}), that the finite groups of symplectic automorphisms of complex K3 surfaces are, up to isomorphism, precisely the subgroups of $M_{24}$ that have at least five orbits in the unique non-trivial permutation representation on $24$ points, including at least one fixed point. 

Since a symplectic automorphism of a complex K3 surface $M$ induces a supersymmetry preserving automorphism of a sigma model attached to $M$ (cf. \S\ref{sec:um:k3}), and since it is the supersymmetry preserving automorphisms of a K3 sigma model that can be used to twine the K3 elliptic genus (\ref{eqn:um-ZK3}), the problem of classifying the supersymmetry preserving automorphism groups of nonlinear K3 sigma models was considered by Gaberdiel--Hohenegger--Volpato in \cite{MR2955931}.

One might have anticipated that all supersymmetry preserving K3 sigma model automorphism groups would be contained in $M_{24}$, but this is not the case. Rather, the main result of \cite{MR2955931}, being a quantum analogue of Mukai's classification of finite symplectic automorphism groups of K3 surfaces, is that 
the groups of supersymmetry preserving automorphisms of K3 sigma models are, up to isomorphism, precisely the subgroups of $\Co_0=\Aut(\Lambda)$ (cf. (\ref{eqn:um-Co0})) that fix a sublattice of $\Lambda$ with rank at least four.

Note that the results of \cite{MR2955931} are obtained subject to certain conjectural assumptions about the structure of the moduli space of K3 sigma models. Nonetheless, it seems fair to conclude that the K3 sigma models do not furnish quite the right theoretical setting for solving the mysteries of umbral moonshine. For not all of the $M_{24}$-twinings (\ref{eqn:um-Zg}) of the K3 elliptic genus (\ref{eqn:um-ZK3}) arise as twinings defined by K3 sigma model automorphisms, since, for example, there are elements of $M_{24}$ (cf. (\ref{eqn:um-M24})) that do not fix a rank four lattice in the Leech lattice, $\Lambda$. (Cf. also, the last sentence of \S\ref{sec:um:k3}.)

That notwithstanding, we can expect to learn useful information about umbral moonshine from further investigation of K3 sigma models. The history of monstrous moonshine provides a useful point of comparison: in advance of his proof of the Conway--Norton conjectures, Borcherds considered a certain BKM algebra (cf. \S\ref{previouswork}) in \cite{MR1069386}, which was, at the time, called the monster Lie algebra, although it turned out to be only indirectly connected to the monster. The Lie algebra constructed in \cite{MR1069386} is now known as the {\em fake monster Lie algebra} (cf. \S2 of \cite{borcherds_monstrous}), and has found a number of applications outside of moonshine. For example, the denominator function of the fake monster Lie algebra (cf. Example 2 in \S10 of \cite{Bor_AutFmsSngGrs}) is used to prove facts about families of K3 surfaces in \cite{MR1620702,MR1686793,MR2087068}.

At the level of vertex operator algebras, the fake monster Lie algebra corresponds to the lattice vertex algebra $V_{\Lambda}$ attached to the Leech lattice. This may be regarded as a {``fake'' moonshine module}, for it has exactly the same graded dimension as $\Vnat$, up to the constant term,
\begin{gather}
	\dimq V_{\Lambda} 
	=J(\tau)+24.
\end{gather}
(Cf. (\ref{eqn:mod-dimqM}).) There is no action of the monster on $V_{\Lambda}$, although there is an action by a group\footnote{Given groups $A$ and $B$, say that a group $G$ has the {\em shape} $A.B$, and write $G=A.B$, if $G$ contains a normal subgroup isomorphic to $A$ such that $G/A\simeq B$. In this setting it is typical to write $p^n$ as a shorthand for an elementary abelian $p$-group with $p^n$ elements.} of the shape $2^{24}.\Co_0=2^{24}.(2.\Co_1)$ (cf. (\ref{eqn:um-Co1})), whereas the monster contains a maximal subgroup with the shape $(2.2^{24}).\Co_1$.
The Frenkel--Lepowsky--Meurman construction of $\Vnat$ takes $V_{\Lambda}$ as a main ingredient. (Cf. (\ref{eqn:class-Leechtorusorb}).)

It is striking that the Conway group $\Co_0=2.\Co_1$ plays a prominent role in so many of the objects we have discussed: it is visible within the monster, and within the automorphism group of $V_{\Lambda}$. It serves for K3 sigma models as $M_{24}$ does for K3 surfaces, as discussed above, and all of the umbral moonshine groups (\ref{eqn:um-GX}) are visible within $\Co_0$, according to (\ref{eqn:um-AutCOut}). 

Moreover, there is moonshine for the Conway group, in direct analogy with that for the monster, in the sense that there is an assignment of normalized principal moduli $T^s_g$ to elements $g\in \Co_0$ which are realized as trace functions on a graded infinite-dimensional $\Co_0$-module. A proof of this statement has recently appeared in \cite{vacogm}. 

To explain this, take $g\in \Co_0$, let $\{\varepsilon_i\}$ be the eigenvalues associated to the action of $g$ on $\Lambda\otimes_{\Z}\CC$, and define
\begin{gather}
	T^s_g(\tau):=q^{-1/2}\prod_{n>0}\prod_{i=1}^{24}(1-\varepsilon_iq^{n-1/2})+\chi_g,
\end{gather}
where 
\begin{gather}\label{eqn:um-chig}
	\chi_g:=\sum_i\varepsilon_i
\end{gather} 
is the character value associated to the action of $g$ on $\Lambda\otimes_{\Z}\CC$. Then $T^s_g(2\tau)=q^{-1}+O(q)$ is the normalized principal modulus for a genus zero group, according to Conway--Norton \cite{MR554399} and Queen \cite{MR628715}. (See also \cite{MR780666}.) 

It has been demonstrated in \cite{vacogm} that the functions $T^s_g$ are the graded traces attached to the action of $\Co_0$ on a distinguished\footnote{The super vertex operator algebra $\Vsnat$ admits actions by both $\Co_0$ (cf. (\ref{eqn:um-Co0})) and the simple group $\Co_1$ (cf. (\ref{eqn:um-Co1})). It's construction as a $\Co_1$-module was sketched first in \S15 of \cite{FLMBerk}, described later in \S5 of \cite{MR1390654}, and subsequently studied in detail in \cite{Dun_VACo}. The $\Co_0$-module structure on $\Vsnat$ is mentioned in \cite{Dun_VACo}, following \cite{MR1390654}, but it seems that the modular properties of the trace functions associated to the $\Co_0$-action were not considered until \cite{vacogm}.} super vertex operator algebra, $\Vsnat=\bigoplus_{n=-1}^\infty\Vsnat_{n/2}$. 
\begin{gather}\label{eqn:um:sigma-Tsg}
	T^s_g(\tau)=\sum_{n=-1}^{\infty}\tr(\gt{z}g|\Vsnat_{n/2})q^{n/2}
\end{gather}
(In (\ref{eqn:um:sigma-Tsg}) we write $\gt{z}$ for the super space involution, acting as $(-1)^n\Id$ on $\Vsnat_{n/2}$.) Thus the super vertex operator algebra $\Vsnat$ solves the Conway moonshine analogue of Thompson's Conjecture \ref{conj:Thompson}, and $\Vsnat$ is the natural analogue of the moonshine module $\Vnat$ for the Conway group, $\Co_0$.

The Conway module $\Vsnat$ is closely related to monstrous moonshine, for, in addition to being directly analogous to $\Vnat$, many of the discrete groups $\Gamma_g<\SL_2(\RR)$, for $g\in \MM$, also arise as invariance groups of principal moduli attached to $\Co_0$ via its action on $\Vsnat$. (Cf. \cite{vacogm}.) On the other hand, $\Vsnat$ enjoys a close connection to K3 sigma models, for it is shown in \cite{vacoeg} that the data defining a K3 sigma model gives rise to a bi-grading on a distinguished, canonically-twisted\footnote{Twisted modules for vertex algebras are discussed in \S\ref{sec:intro:mod}. In the case of a super vertex algebra there is a canonically defined involution coming from the superspace structure, called the {\em parity involution}. A module for a super vertex algebra that is twisted with respect to the parity involution is called {\em canonically-twisted}.} $\Vsnat$-module 
\begin{gather}\label{eqn:um-Vsnattwnr}
\Vsnat_\tw=\bigoplus_{n,r}(\Vsnat_\tw)_{n,r},
\end{gather} 
such that the associated graded traces of compatible elements of $\Co_0$ are weak Jacobi forms. 

More specifically, following \S2.1 of \cite{MR2955931}, we may regard the data of a K3 sigma model as equivalent\footnote{This convention excludes some interesting K3 sigma models, such as those considered in \cite{2014arXiv1406.0619C}. We refer to \cite{MR1479699,MR1416354} for detailed discussions of K3 sigma model moduli.} to a choice of positive-definite $4$-space $\Pi<I\!I_{4,20}\otimes_{\Z}\RR$ (cf. (\ref{eqn:class-II11})), such that 
\begin{gather}\label{eqn:um-PiperpII}
	\delta\in\Pi^{\perp}\cap I\!I_{4,20}\implies \langle\delta,\delta\rangle\neq -2.
\end{gather}
Then the supersymmetry preserving automorphism group of the nonlinear sigma model defined by $\Pi$ is the group 
\begin{gather}\label{eqn:um-GPi}
G_{\Pi}:=\Aut(I\!I_{4,20},\Pi),
\end{gather} 
composed of orthogonal transformations of $I\!I_{4,20}$ that extend to the identity on $\Pi$, according to \S2.2 of \cite{MR2955931}. One of the main results of \cite{MR2955931} is that $G_{\Pi}$ may be identified with a subgroup of $\Co_0$.

The construction of \cite{vacoeg} uses $\Vsnat_\tw$ to attach a graded trace function
\begin{gather}\label{eqn:um-phig}
	\phi_g(\tau,z):=-\sum_{n,r}\tr(\gt{z} g|(\Vsnat_\tw)_{n,r})q^ny^r
\end{gather}
to each pair $(g,\Pi)$, where $\Pi<I\!I_{4,20}\otimes_\Z\RR$ satisfies (\ref{eqn:um-PiperpII}), and $g\in G_\Pi$, and $\gt{z}$ is a certain naturally defined involution on $\Vsnat_\tw$ (analogous to the $\gt{z}$ in (\ref{eqn:um:sigma-Tsg})). It is shown in \cite{vacoeg} that $\phi_g$ is a weak Jacobi form of weight zero and index one for $\Gamma_0^J(N)$ (cf. (\ref{eqn:um-Zg})), for some $N$, for all choices of $\Pi$ and $g\in G_{\Pi}$. Moreover, $\phi_g$ is found to coincide with the $g$-twined K3 elliptic genus associated to the sigma model defined by $\Pi$, for all the examples computed in \cite{MR2955931,Gaberdiel:2012um,2014arXiv1403.2410V}. (These examples account for about half of the conjugacy classes of $\Co_0$ that fix a $4$-space in $\Lambda\otimes_{\Z}\RR$.) In particular, taking $g=e$ in (\ref{eqn:um-phig}) recovers the K3 elliptic genus (\ref{eqn:um-ZK3}), but in the form
\begin{gather}
	\phi_e(\tau,z)=
	-2^{11}\frac{\theta_2(\tau,z)^2}{\theta_2(\tau,0)^2}\frac{\Delta(2\tau)}{\Delta(\tau)}
	+\frac{1}{2}\frac{\theta_3(\tau,z)^2}{\theta_3(\tau,0)^2}\frac{\Delta(\tau)^2}{\Delta(2\tau)\Delta(\tau/2)}
	-\frac{1}{2}\frac{\theta_4(\tau,z)^2}{\theta_4(\tau,0)^2}\frac{\Delta(\tau/2)}{\Delta(\tau)},
\end{gather}
where $\Delta(\tau):=\eta(\tau)^{24}=q\prod_{n=1}^\infty(1-q^n)^{24}$.

Thus $\Vsnat_\tw$ serves as a kind of universal object for K3 sigma models. This is interesting, for generally it is difficult to construct the Hilbert spaces underlying a K3 sigma model, and therefore difficult to compute the associated twined K3 elliptic genera, for instance, for all but a few special examples.

In \cite{2013arXiv1309.6528H}, Huybrechts has related the positive-definite $4$-spaces $\Pi<I\!I_{4,20}\otimes_\Z\RR$ satisfying (\ref{eqn:um-PiperpII}) to pairs $(X,\sigma)$, where $X$ is a projective complex K3 surface, and $\sigma$ is a stability condition on the bounded derived category of coherent sheaves on $X$. In this way he has obtained an alternative analogue of Mukai's result \cite{Mukai}, whereby symplectic automorphisms of K3 surfaces are replaced by symplectic derived autoequivalences. (The results of \cite{vacoeg} are formulated in this language.)

A number of the functions $Z_g$ (cf. (\ref{eqn:um-Zg})) arising in Mathieu moonshine are realized as $\phi_g$ for some $g\in \Co_0$. So the construction of \cite{vacoeg} relates $\Vsnat$ to umbral moonshine, but the connection goes deeper, for it is shown in \cite{vacoeg} that a natural generalization of the definition (\ref{eqn:um-phig}) recovers a number of the Jacobi forms attached to other root systems of the form $X=A_n^d$ (cf. \S\ref{sec:um:niemeier}), beyond the special case $X=A_1^{24}$. It is interesting to compare this to the results of \cite{2014arXiv1406.0619C} (see also the precursor \cite{2013arXiv1307.7717H}), which demonstrate a role for K3 surface geometry in all cases of umbral moonshine (i.e., for all the Niemeier root systems), by considering sigma models attached to K3 surfaces admitting du Val singularities. Indeed, many of the Jacobi forms computed in \cite{2014arXiv1406.0619C} appear also in \cite{vacoeg}. 

From the discussion above we see that the Conway module $\Vsnat$ provides evidence for a deep connection between monstrous and umbral moonshine. Further support for the notion that monstrous and umbral moonshine share a common origin is obtained in \cite{2014arXiv1403.3712O}, where the generalized Borcherds products of \cite{MR2726107} are used to relate the trace functions of monstrous and umbral moonshine directly.

We elaborate now upon the results of \cite{2014arXiv1406.5502C,2015arXiv150307219C}, which, as mentioned at the end of \S\ref{sec:um:mods}, fall outside of umbral moonshine as formulated in \cite{MUM}, but are nonetheless related. Actually, these works are further applications of the Conway moonshine module $\Vsnat$, for in \cite{2014arXiv1406.5502C} the canonically-twisted $\Vsnat$-module $\Vsnat_\tw$ (cf. (\ref{eqn:um-Vsnattwnr})) is equipped with module structures for the $N=2$ and $N=4$ superconformal algebras, which in turn give rise to an assignment of distinguished vector-valued mock modular forms to elements of the sporadic simple Mathieu groups $M_{23}$ and $M_{22}$, respectively. This work furnishes the first examples of concretely constructed modules for sporadic simple groups, such that the associated graded trace functions define mock modular forms. The methods of \cite{2014arXiv1406.5502C} are extended in \cite{2015arXiv150307219C}, to the case of the ${Spin}(7)$ algebra (a certain extension of the $N=1$ superconformal algebra, cf. \cite{2014arXiv1412.2804B} and references therein) and vector-valued mock modular forms for $M_{24}$ are obtained. Interestingly, Conway's sporadic groups $\Co_2$ and $\Co_3$, and the sporadic groups of McLaughlin and Higman--Sims (all rather larger than $M_{24}$, or any of the other umbral groups $G^X$) also appear in the analysis of \cite{2014arXiv1406.5502C,2015arXiv150307219C}.

As a further indication of the important role that K3 sigma models will play in illuminating umbral moonshine, we mention the interesting work \cite{Gaberdiel:2013psa,MR3106313,Taormina:2013mda,Taormina:2013jza}, which seeks to explain the Mathieu moonshine observation by formulating a precise mechanism for combining symmetries of distinct K3 sigma models into a single group. We note in particular, that a fixed-point-free maximal subgroup of $M_{24}$ is constructed in this way in \cite{Taormina:2013jza}.

We conclude this section with references \cite{Cheng:2013kpa,Harrison:2013bya,2013arXiv1307.7717H,Harvey:2014cva,2014arXiv1407.3181K,Paquette:2014rma,Wrase:2014fja} to a number of other occurrences of umbral groups in geometry and physics, all promising connections to Mathieu moonshine, or umbral moonshine more generally. We also note the recent work \cite{ChedeLWha_GUM}, which analyzes all cases of generalized umbral moonshine, thereby extending the investigation of generalized Mathieu moonshine that was initiated in \cite{MR3108775}.

\section{Open Problems}\label{sec:open}

We conclude the article by identifying some open problems for future research that are suggested by our results in \S\ref{sec:tower} and \S\ref{sec:mmdist}, and the developments described in \S\ref{sec:um}.

\begin{problem}
We have seen that the known connections between monstrous moonshine and physics owe much to the Frenkel--Lepowsky--Meurman construction of the moonshine module $\Vnat$, and its associated vertex operator algebra structure. Just as the vertex operator algebra structure on $\Vnat$ gives a strong solution to Thompson's conjecture, Conjecture \ref{conj:Thompson}, we can expect concrete constructions of the $\check{K}^X$---whose existence is now guaranteed thanks to Theorem \ref{thm:um:mods}---to be necessary for the elucidation of the physical origins of umbral moonshine. 
As we have described in \S\ref{sec:um:mods}, progress on this problem has been obtained recently in \cite{umvan2,mod3e8,umvan4}, and the related work \cite{2014arXiv1406.5502C} may also be useful, in the determination of a general, algebraic solution to the module problem for umbral moonshine.
\end{problem}

\begin{problem}
Norton's generalized moonshine conjectures were discussed in \S\ref{sec:intro:mod}, and generalized umbral moonshine has been investigated in \cite{ChedeLWha_GUM,MR3108775}. 
A special case of generalized moonshine for the Conway group is established in \cite{vacogm}, but the full formulation and proof of generalized Conway moonshine remains open. Given the close connections between $\Vsnat$ and umbral moonshine discussed in \S\ref{sec:um:sigma}, it will be very interesting to determine the precise relationship between the corresponding generalized moonshine theories. We can expect that the elucidation of these structures will be necessary for a full understanding of the role that umbral moonshine plays in physics. 
\end{problem}

\begin{problem}
As discussed in \S\ref{sec:um:niemeier}, the 
fact (cf. Theorem \ref{thm:radsums-TgTGammag}) that the McKay--Thompson series of monstrous moonshine are realized as Rademacher sums admits conjectural analogues for umbral moonshine. (See Conjecture 3.2 of \cite{mumcor} for a precise formulation.) So far this has been established only for $X=A_1^{24}$, corresponding to Mathieu moonshine (cf. \cite{Cheng2011}), and the general case remains open. 
As explained in \S\ref{sec:um:niemeier}, a positive solution to Conjecture 3.2 of \cite{mumcor} will establish an umbral moonshine counterpart to the principal modulus/genus zero property of monstrous moonshine. 
\end{problem}

\begin{problem}
As discussed above, in \S\ref{sec:intro:radsums} and in \S\ref{sec:um:niemeier}, Rademacher sums play a crucial role in both monstrous and umbral moonshine, by serving to demonstrate the distinguished nature of the automorphic functions arising in each setting. In the case of monstrous moonshine, the Rademacher sum property also indicates a potentially powerful connection to physics, via three-dimensional gravity, as explained in \S\ref{sec:intro:qgrav}. Thus it is an interesting problem to formulate umbral moonshine analogues of the conjectures of \cite{DunFre_RSMG}, discussed here in \S\S\ref{sec:intro:qgrav},\ref{sec:tower}.
\end{problem}

\begin{problem}
Relatedly, it follows from the results of \cite{DunFre_RSMG} that the McKay--Thompson series $T^s_g$ (cf. (\ref{eqn:um:sigma-Tsg})), attached to elements $g$ in the Conway group $\Co_0$ via its action on $\Vsnat$ (cf. \S\ref{sec:um:sigma}), are also realized as Rademacher sums. That is, Theorem \ref{thm:radsums-TgTGammag} generalizes naturally to Conway moonshine. Thus it is natural to investigate the higher order analogues $V^{s(-m)}$ of the super vertex operator algebra $\Vsnat$, and the Conway group analogues of the three-dimensional gravity conjectures of \cite{DunFre_RSMG}. Some perspectives on this are available in \cite{Hoe_SDVOSALgeMinWt,MalWit_QGPtnFn3D,Witten2007}.
\end{problem}

\begin{problem}
The notion of extremal vertex operator algebra is defined by (\ref{eqn:qgrav-ext}). So far the only known example is the moonshine module $\Vnat$. As explained in \S\ref{sec:intro:qgrav}, the construction of a series of extremal vertex operator algebras, with central charges the positive integer multiples of $24$, would go a long way towards the construction of a chiral three-dimensional quantum gravity theory. This problem also has a super analogue, cf. \cite{Hoe_SDVOSALgeMinWt}.
\end{problem}

\begin{problem}
The monster modules $V^{(-m)}$, defined in \S\ref{sec:tower}, cannot be vertex operator algebras for $m>1$, for an action of the Virasoro algebra would generate a non-zero vector with non-positive eigenvalue for $L(0)-{\bf c}/24$, and this would violate the condition $\sum_n\dim(V^{(-m)}_n)q^n=q^{-m}+O(q)$. Nonetheless, we may ask: do the $V^{(-m)}$ admit vertex algebra structure? Or, is there another natural algebraic structure, which characterizes the monster group actions on the $V^{(-m)}$? 
\end{problem}

\begin{problem}
Relatedly, $\Vir\times\MM$-modules satisfying the extremal condition (\ref{eqn:qgrav-ext}) may be easily constructed from the monster modules $V^{(-m)}$, as is mentioned at the conclusion of \S\ref{sec:tower}.
What is the algebraic significance of these spaces? We know from \cite{Gaiotto:2008jt,Hoe_SDVOSALgeMinWt} that they cannot admit vertex operator algebra structure compatible with the given $\Vir\times \MM$ actions. Is there some other kind of algebraic structure which is compatible with this symmetry?
\end{problem}

\begin{problem}\label{prob:Griess}
The result of Corollary \ref{cor:dist} implies that the $V^{(-m)}_n$ and $W^{\natural}_n$ tend to direct sums of copies of the regular representation of $\MM$, as $n\to \infty$. This means that if we write each homogeneous subspace of each module, particularly the moonshine module $\Vnat$, as the sum of a {free} part (free over the group ring of $\MM$) and a non-free part, then the non-free part tends to $0$ (relative to the free part) as $n\to \infty$. Is there something to be learnt from an analysis of the non-free parts of $V^{(-m)}$, $W^{\natural}$? As one can see from Table \ref{tab:dist-exp}, some irreducible representations of the monster feature more often in the non-free part than others. We thank Bob Griess for posing this question.
\end{problem}

\begin{problem}
It is a natural problem to generalize the methods employed in \S\ref{sec:mmdist}, to determine the distributions of irreducible representations of the umbral groups $G^X$ (cf. (\ref{eqn:um-GX})) in the umbral moonshine modules $\check{K}^X$ (cf. (\ref{eqn:um-KX})). Similarly, one may also consider the distributions of the $\Co_0$-modules in $\Vsnat=V^{s(-1)}$, and in the $V^{s(-m)}$ more generally. In all of these cases, questions analogous to Problem \ref{prob:Griess} may reward investigation.
\end{problem}

\begin{problem}\label{prob:qdim}
In Corollary \ref{cor:qdims} we have used our asymptotic results (Theorem \ref{distribution}) on multiplicities of monster modules inside $\Vnat$ to compute the quantum dimensions of the monster orbifold, and in so doing confirmed a special case of Conjecture 6.7 of \cite{MR3105758}. How generally can this method be applied, to orbifolds $V^G$, where $V$ is a vertex operator algebra and $G$ is a compact group of automorphisms of $V$? Note the following strengths of the asymptotic approach: thanks to Proposition 3.6 of \cite{MR3105758}, we did not need to verify that the monster orbifold of $\Vnat$ is a rational vertex operator algebra, nor did we need to assume the positivity condition of Theorem 6.3 in \cite{MR3105758} (which in any case does not hold for $\Vnat$).
\end{problem}

\clearpage

\appendix
\section{Monstrous Groups}\label{Appendix}

The table below contains the symbols $\Gamma_g=N||h+e,f,\dots,$ for each conjugacy class of the monster. Following \cite{ConMcKSebDiscGpsM}, if $h=1,$ we omit the `$||1$' from the symbol. If $\mathcal W_g=\{1\}$, then we write $N||h,$ whereas if it contains every exact divisor of $N/h,$ we write $N||h+$.

The naming of the conjugacy classes is as in \cite{ATLAS}. We follow the convention of writing $23AB$ as a shorthand for $23A\cup 23B$, since these conjugacy classes are related by inversion in the monster. (There are 22 such pairs. The Monster group has $196$ conjugacy classes in total.) Since the monstrous McKay--Thompson series have real coefficients, $T_g=T_{g^{-1}}$ and $\Gamma_g=\Gamma_{g^{-1}}$ for all $g$ in the monster. Note however that $27A$ and $27B$ are not related by inversion, even though $\Gamma_{27A}=\Gamma_{27B}$. To the authors best knowledge, this coincidence has not yet been explained.

\begin{multicols}{3}
\begin{tabular}{ll}
1A  & $ 1 $\\
2A  & $ 2+ $\\
2B  & $ 2 $\\
3A  & $ 3+ $\\
3B  & $ 3 $\\
3C  & $ 3||3 $\\
4A  & $ 4+ $\\
4B  & $ 4||2+ $\\
4C  & $ 4 $\\
4D  & $ 4||2 $\\
5A  & $ 5+ $\\
5B  & $ 5 $\\
6A  & $ 6+ $\\
6B  & $ 6+6 $\\
6C  & $ 6+3 $\\
6D  & $ 6+2 $\\
6E  & $ 6 $\\
6F  & $ 6||3 $\\
7A  & $ 7+ $\\
7B  & $ 7 $\\
8A  & $ 8+ $\\
8B  & $ 8||2+ $\\
8C  & $ 8||4+ $\\
8D  & $ 8||2 $\\
8E  & $ 8 $\\
8F  & $ 8||4 $\\
9A  & $ 9+ $\\
9B  & $ 9 $\\
10A  & $ 10+ $\\
10B  & $ 10+5 $\\
10C  & $ 10+2 $\\
10D  & $ 10+10 $\\
10E  & $ 10 $\\
11A  & $ 11+ $\\
12A  & $ 12+ $\\
12B  & $ 12+4 $
\end{tabular}

\columnbreak

\begin{tabular}{ll}

12C  & $ 12||2+ $\\
12D  & $ 12||3+ $\\
12E  & $ 12+3 $\\
12F  & $ 12||2+6 $\\
12G  & $ 12||2+2 $\\
12H  & $ 12+12 $\\
12I  & $ 12 $\\
12J  & $ 12||6 $\\
13A  & $ 13+ $\\
13B  & $ 13 $\\
14A  & $ 14+ $\\
14B  & $ 14+7 $\\
14C  & $ 14+14 $\\
15A  & $ 15+ $\\
15B  & $ 15+5 $\\
15C  & $ 15+15 $\\
15D  & $ 15||3 $\\
16A  & $ 16||2+ $\\
16B  & $ 16 $\\
16C  & $ 16+ $\\
17A  & $ 17+ $\\
18A  & $ 18+2 $\\
18B  & $ 18+ $\\
18C  & $ 18+9 $\\
18D  & $ 18 $\\
18E  & $ 18+18 $\\
19A  & $ 19+ $\\
20A  & $ 20+ $\\
20B  & $ 20||2+ $\\
20C  & $ 20+4 $\\
20D  & $ 20||2+5 $\\
20E  & $ 20||2+10 $\\
20F  & $ 20+20 $\\
21A  & $ 21+ $\\
21B  & $ 21+3 $\\
21C  & $ 21||3+ $\\
\end{tabular}

\columnbreak

\begin{tabular}{ll}

21D  & $ 21+21 $\\
22A  & $ 22+ $\\
22B  & $ 22+11 $\\
23AB  & $ 23+ $\\
24A  & $ 24||2+ $\\
24B  & $ 24+ $\\
24C  & $ 24+8 $\\
24D  & $ 24||2+3 $\\
24E  & $ 24||6+ $\\
24F  & $ 24||4+6 $\\
24G  & $ 24||4+2 $\\
24H  & $ 24||2+12 $\\
24I  & $ 24+24 $\\
24J  & $ 24||12 $\\
25A  & $ 25+ $\\
26A  & $ 26+ $\\
26B  & $ 26+26 $\\
27A  & $ 27+ $\\
27B  & $ 27+ $\\
28A  & $ 28||2+ $\\
28B  & $ 28+ $\\
28C  & $ 28+7 $\\
28D  & $ 28||2+14 $\\
29A  & $ 29+ $\\
30A  & $ 30+6,10,15 $\\
30B  & $ 30+ $\\
30C  & $ 30+3,5,15 $\\
30D  & $ 30+5,6,30 $\\
30E  & $ 30||3+10 $\\
30F  & $ 30+2,15,30 $\\
30G  & $ 30+15 $\\
31AB  & $ 31+ $\\
32A  & $ 32+ $\\
32B  & $ 32||2+ $\\
33A  & $ 33+11 $\\
33B  & $ 33+ $\\
\end{tabular}

\columnbreak

\begin{tabular}{ll}
34A  & $ 34+ $\\
35A  & $ 35+ $\\
35B  & $ 35+35 $\\
36A  & $ 36+ $\\
36B  & $ 36+4 $\\
36C  & $ 36||2+ $\\
36D  & $ 36+36 $\\
38A  & $ 38+ $\\
39A  & $ 39+ $\\
39B  & $ 39||3+ $\\
39CD  & $ 39+39 $\\
40A  & $ 40||4+ $\\
40B  & $ 40||2+ $\\
40CD  & $ 40||2+20 $\\
41A  & $ 41+ $\\
42A  & $ 42+ $\\
42B  & $ 42+6,14,21 $\\
42C  & $ 42||3+7 $\\
42D  & $ 42+3,14,42 $\\
44AB  & $ 44+ $\\
45A  & $ 45+ $\\
46AB  & $ 46+23 $\\
\end{tabular}

\columnbreak

\begin{tabular}{ll}
46CD  & $ 46+ $\\
47AB  & $ 47+ $\\
48A  & $ 48||2+ $\\
50A  & $ 50+ $\\
51A  & $ 51+ $\\
52A  & $ 52||2+ $\\
52B  & $ 52||2+26 $\\
54A  & $ 54+ $\\
55A  & $ 55+ $\\
56A  & $ 56+ $\\
56BC  & $ 56||4+14 $\\
57A  & $ 57||3+ $\\
59AB  & $ 59+ $\\
60A  & $ 60||2+ $\\
60B  & $ 60+ $\\
60C  & $ 60+4,15,60 $\\
60D  & $ 60+12,15,20 $\\
60E  & $ 60||2+5,6,30 $\\
60F  & $ 60||6+10 $\\
62AB  & $ 62+ $\\
66A  & $ 66+ $\\
66B  & $ 66+6,11,66 $\\
\end{tabular}

\columnbreak

\begin{tabular}{ll}
68A  & $ 68||2+ $\\
69AB  & $ 69+ $\\
70A  & $ 70+ $\\
70B  & $ 70+10,14,35 $\\
71AB  & $ 71+ $\\
78A  & $ 78+ $\\
78BC  & $ 78+6,26,39 $\\
84A  & $ 84||2+ $\\
84B  & $ 84||2+6,14,21 $\\
84C  & $ 84||3+ $\\
87AB  & $ 87+ $\\
88AB  & $ 88||2+ $\\
92AB  & $ 92+ $\\
93AB  & $ 93||3+ $\\
94AB  & $ 94+ $\\
95AB  & $ 95+ $\\
104AB  & $ 104||4+ $\\
105A  & $ 105+ $\\
110A  & $ 110+ $\\
119AB  & $ 119+ $\\
\end{tabular}

\end{multicols}

\providecommand{\bysame}{\leavevmode\hbox to3em{\hrulefill}\thinspace}
\providecommand{\MR}{\relax\ifhmode\unskip\space\fi MR }

\providecommand{\MRhref}[2]{  
\href{http://www.ams.org/mathscinet-getitem?mr=#1}{#2}
}
\providecommand{\href}[2]{#2}

\end{document}